\documentclass[10pt]{amsart}
\usepackage{amsmath}
\usepackage{amssymb}
\usepackage{amsthm}
\usepackage{amsfonts}
\usepackage{graphicx}
\usepackage{enumitem,booktabs}

\usepackage[dvipsnames]{xcolor}

\definecolor{limegreen}{rgb}{0.196,0.804,0.196}
\definecolor{darkgreen}{rgb}{0.0,0.5,0.0}
\definecolor{darkbluegreen}{rgb}{0,0.3,0.6}
\definecolor{badgerred}{rgb}{0.715,0.004,0.004}

\usepackage{hyperref}
\hypersetup{colorlinks,
citecolor=badgerred,
filecolor=black,
linkcolor=blue,
urlcolor=darkgreen,
bookmarksopen=false,
bookmarksdepth=2}

\newcommand{\R}{{\mathbb R}}
\newcommand{\N}{{\mathbb N}}
\newcommand{\cL}{{\mathcal L}}

\newcommand{\cI}{{\mathcal I}}
\newcommand{\cQ}{{\mathcal Q}}
\newcommand{\bU}{U}

\newcommand{\cM}{{\mathcal M}}
\newcommand{\cN}{{\mathcal N}}

\newcommand{\cR}{{\mathcal R}}
\newcommand{\cO}{{\mathcal O}}

\newcommand{\pd}{\partial}

\newtheorem{theorem}{Theorem}
\newtheorem{proposition}[theorem]{Proposition}
\newtheorem{lemma}[theorem]{Lemma}

\newtheorem{corollary}[theorem]{Corollary}
\theoremstyle{remark}

\newtheorem{claim}[theorem]{Claim}

\numberwithin{equation}{subsection}
\numberwithin{theorem}{subsection}

\title[Type II smoothing]{Type II smoothing in Mean curvature flow }
\author[Angenent]{Sigurd Angenent}
\address{Department of Mathematics, University of Wisconsin -- Madison}
\author[Daskalopoulos]{Panagiota Daskalopoulos}
\address{Department of Mathematics, Columbia University, New York}
\author[Sesum]{Natasa Sesum}
\address{Department of Mathematics, Rutgers University, New Jersey}

\thanks{
  P.~Daskalopoulos thanks the NSF for support in DMS-1266172.
  N.~Sesum thanks the NSF for support in DMS-1056387 and in DMS-1811833.
}

\begin{document}

\maketitle

\begin{abstract}

  In 1994 Velázquez~\cite{Velaz} constructed a smooth \(O(4)\times O(4)\) invariant
  Mean Curvature Flow that forms a type-II singularity at the origin in
  space-time.  Stolarski \cite{S} very recently showed that the mean curvature
  on this solution is uniformly bounded.  Earlier, Velázquez~\cite{AIV} also provided
  formal asymptotic expansions for a possible smooth continuation of the
  solution after the singularity. 
  
  Here we prove short time existence of Velázquez' formal continuation, and we
  verify that the mean curvature is also uniformly bounded on the continuation.
  Combined with the earlier results of Velázquez--Stolarski we therefore show
  that there exists a solution \(\{M_t^7\subset\R^8 \mid -t_0 <t<t_0\}\) that
  has an isolated singularity at the origin \(0\in\R^8\), and at \(t=0\);
  moreover, the mean curvature is uniformly bounded on this solution, even though the second fundamental form is unbounded near the singularity.

%  We show that one of the mean curvature flow solutions constructed by Velázquez
%  in \cite{Velaz} that develops a Type II singularity at the origin, at time \(t
%  = 0\), and is modeled on one of the Alencar minimal surfaces, can be extended
%  past the singular time to a smooth solution for \(t > 0\).  In \cite{S}, Stolarski
%   showed that one of the Velázquez solutions, even though develops a
%  singularity at time \(t =0\), still has uniformly bounded mean curvature up to
%  time \(t = 0\).  We show that a smooth extension of this solution
%  for \(t > 0\), also has uniformly bounded mean curvature  all the way back to
%  the singular time \(t = 0\).  This way we produce a solution to the
%  mean curvature flow equation, that exists for all times \(t\in [-t_0, t_0]\), for
%  some \(t_0 > 0\), has a Type II singularity at the origin, at time \(t = 0\), and
%  has uniformly bounded mean curvature everywhere on a surface, for all times
%  \(t\in [-t_0,t_0]\setminus \{0\}\).
\end{abstract}

\section{Introduction} 

We say that a family of hypersurfaces \(\{M_t\}_{t\in [0,T)} \subset \mathbb{R}^{n+1}\)
moves by the mean curvature flow if
\begin{equation}
  \label{eq-mcf}
  \frac{\partial \vec F}{\partial t} = \vec H\tag{MCF}
\end{equation}
where \(\vec H(\cdot,t)\) is the mean curvature vector of the hypersurface
\(M_t\),  and \(\vec F(\cdot,t): M \to M_t \subset \mathbb{R}^{n+1}\) is a
smooth family of parametrizations of the moving hypersurface.  In the case of
closed hypersurfaces, Huisken showed the norm of the second fundamental form
blows up at finite time \(T < \infty\), that is
\[
  \limsup_{t\to T} \max_{M_t} |A|(\cdot,t) = \infty.
\]
Very often, even in a complete, noncompact setting, mean curvature flow
\eqref{eq-mcf} develops a singularity at a finite time \(T < \infty\).  It is
very natural to ask whether the mean curvature also needs to blow up at a finite
time singularity, or equivalently, whether a uniform bound on \(|\vec
H|\) for all \(t\in [0,T)\) guarantees the existence of smooth solution
past time \(T\).

For mean convex flows it is well known \cite{Hu} that the mean curvature bounds the
second fundamental form \(A\), i.e.~\(|A|/|\vec H|\) attains its maximum at \(t=0\) and therefore is uniformly bounded.
This implies that for mean convex flows the mean curvature is never bounded
near a singularity.  Dropping the assumption of mean convexity, it was shown in
\cite{LS1,LS2,LS3} by Lin-Sesum and Le-Sesum, and in \cite{XYZ} by Xu-Ye-Zhao
that for mean curvature flow of closed hypersurfaces the mean curvature needs
to blow up at the first singular time, given some extra assumptions, such as
having only Type I singularities or being close to a sphere in the \(L^2\)
sense.  More recently, in \cite{LW}, Li and Wang showed, using a quite involved
argument that in the case of closed surfaces in \(\mathbb{R}^3\) the mean
curvature always blows up at the first singular time.  
The question of boundedness of the mean curvature on a singular mean curvature flow is therefore completely settled in the case of compact surfaces in \(\R^3\), and a variety of extra assumptions for hypersurfaces in higher dimensions.

For \(n \ge 4\), in \cite{Velaz} Velázquez constructed \(N =
2n-1\)-dimensional, \(O(n)\times O(n)\) symmetric solutions that converge to
the Simons cone at parabolic scales around the singularity, and converge to a
smooth minimal surface desingularizing Simons cone at the scale at which the
norm of the second fundamental form blows up at the origin.  Using formal
asymptotic expansions Velázquez \cite{AIV} also suggested a way in which the
solution \(\{M_t\}\) might be continued smoothly after the singularity,
i.e.~for \(t>0\).    

\smallskip 
It was believed that these complete noncompact solutions should provide
examples of higher dimensional mean curvature flow with the property that the
mean curvature stays bounded at the first singular time.  In \cite{S} Stolarski
used precise asymptotics of these solutions together with sophisticated blow up
arguments to rigorously prove that this is indeed the case for \(t<0\), i.e.~he
showed that before the singularity forms the mean curvature on some of
Velázquez' solutions is uniformly bounded.  (To be precise: he requires the
parameter \(k\) that appears in Velázquez' solutions to be even and not less
than \(4\).)

\smallskip 
Here we consider the case \(n=4\), i.e.~the case of \(7\)-dimensional
hypersurfaces in \(\R^8\).  We first prove existence and regularity of
Velázquez' formal extension of the Velázquez--Stolarski solutions and we
thereby obtain a solution \(\{M_t\subset \R^8 \mid -t_0<t<t_0\}\) of MCF that
is smooth everywhere except at the origin \((0,0)\in\R^8\times(-t_0, t_0)\)
in space-time, and whose {\em mean curvature is uniformly bounded}  even though its
{\em second fundamental form blows up near}  \((0,0)\).  In particular, we show that
the singular hypersurface \(M_0=\lim_{t\nearrow 0} M_t\) that remains after the
Velázquez--Stolarski solution forms its singularity can be used as initial data
for MCF, and that at least one of the ensuing solutions has uniformly bounded
mean curvature.

\smallskip
In \cite{S} Stolarski indicates he expects his result to be true for closed
mean curvature flow that can be obtained by compactifying Velázquez examples,
but it still remains open.  Another question that remains completely open is
what happens in dimensions \(3 \le N \le 6\) where neither an example of a
singular solution with bounded mean curvature nor a theorem proving the
impossibility of such an example are known.

\subsection*{Acknowledgement} The authors would like to thank J.J.L.Velázquez
for helpful conversations about formal asymptotics and the construction of
solutions to MCF.

\subsection{Outline}
In this paper we consider an \(O(4)\times O(4)\) symmetric hypersurface \(M_0\)
defined by the profile function
\[
  u=u_0(x)
\]
where \(u_0:(0,\infty)\to\R\) is a smooth function, that near the origin satisfies 
\begin{equation}
  \label{eq-init-data1}
  u(x,0) = x + K_0 \, x^{2(k-1)} + o(x^{2(k-1)}) \qquad (x\searrow 0),
\end{equation}
for some integer 
\[
  k\geq 4
\] 
and some constant \(K_0 >0\).  
We will also assume that for all \(x>0\) one has
\begin{equation}
  \label{eq-init-data2}
  0 \leq u_0'(x) \leq C_0, \qquad  |u_{0}''(x)| \leq C_0,   \qquad 
  |u_{0}''(x)|\leq C_0\, x^{2k-4}
\end{equation}
for some constant \(C_0>0\).  The last assumption implies, after integration, that for all \(x>0\) one has
\begin{equation}
  \label{eq-init-data3}
  |u_{0}'(x) - 1| \leq C \, x^{2k-3}
\end{equation}
for some constant \(C>0\), depending on \(C_0\).  This implies that for \(x\) small enough we have \(u_0'(x) \ge \frac 12\).  By rescaling we may assume that
\begin{equation}
  \label{eq-der-below}
  u_0'(x) \ge c > 0, \qquad \mbox{for} \qquad x\in [0,1].
\end{equation}

It turns out that such a function \(u_0(x)\) is the profile near a singularity
\((0,0)\) of the \(O(4) \times O(4)\) MCF solution \(M_t\), \(- t_1 < t <0\), for some
small \(t_1 <0\), which was constructed by Velázquez in \cite{Velaz}.  It was
recently shown in \cite{S} that the Velázquez solution has bounded mean curvature at
the singularity, that is the mean curvature of \(M_t\) remains bounded as \(t \to
0^-\) near \((0,0)\). 

\smallskip 
Our goal in this paper is to show that  the MCF starting at \(M_0\) can be
continued for  \(0 < t < t_0\), for some \(t_0 >0\) small, with a smooth solution
\(M_t\), \(t \in (0, t_0)\) which is \(O(4) \times O(4)\) symmetric.  Furthermore, the
mean curvature of \(M_t\) as \(t \to 0^+\) will {\em remain uniformly bounded } despite the
fact that \(M_0\) is singular at \(x=0\). 

\smallskip 

The   solution \(M_t\) will be defined by a profile function \(u:(0,\infty)\times(0, t_0)\to(0,\infty)\), 
that satisfies the initial value problem 
\begin{subequations}
  \begin{gather}
    \label{eq-u-pde}
    u_t = \frac{u_{xx}}{1+u_x^2} + \frac 3x u_x - \frac 3u\\
    \label{eq-u-bcond}
    \lim_{x\to0}u_x(x, t) = 0 \\
    \label{eq-u-initcond}
    \lim_{t\to0}u(x, t) = u_0(x).
  \end{gather}
\end{subequations}
Note the condition \( \lim_{x\to0}u_x(x, t) = 0\) assures that \(u_0(x,t)\) defines a
\(O(4) \times O(4)\) hypersurface \(M_t\) that is smooth at the origin and hence
everywhere.

\smallskip

We will prove the following Theorem:

\subsection{Main Theorem}\itshape Assume that \(M_0\) is a
\(O(4)\times O(4)\) symmetric hypersurface defined by the profile function
\(u_0:[0,\infty)\to\R\) which is smooth for \(x >0\) and at \(x=0\) satisfies
condition \eqref{eq-init-data1},  for some \(k > 3\).   Then,
there exists \(t_0>0\) and a \(C^\infty\)-smooth \(O(4) \times O(4)\) symmetric MCF
solution \(M_t\), \(0 < t \leq t_0\) defined by a profile function
\(u:(0,\infty)\times(0, t_0]\to(0,\infty)\) which satisfies the initial value
problem \eqref{eq-u-pde}--\eqref{eq-u-initcond}.  Furthermore the mean curvature
\(H(x,t)\) of the hypersurface \(M_t\) satisfies
\[  
\sup_{(x,t) \in [0,1] \times (0,a]} |H(x,t) |< +\infty
\]
for some \(0 < a \leq t_0\), i.e., \(H(x,t)\) is uniformly bounded near the origin
as \(t \to 0^+\) despite the fact that the mean curvature of \(M_0\) is undefined
at the origin. 

\upshape\medskip

As a corollary of the Main Theorem and the results in \cite{S} we have the following result.

\begin{corollary} There exists a \(O(4)\times O(4)\) symmetric complete
  noncompact mean curvature flow solution \(\{M_t\}_{t\in (-t_0, t_0)}\), so
  that \(M_t\) is smooth for all \(t\in (-t_0, t_0)\backslash \{0\}\), has a Type
  II singularity at the origin, at time \(t = 0\), and has uniformly bounded mean
  curvature away from \(t = 0\).  More precisely, there exists a uniform constant
  \(C\) so that
  \[
    \sup_{\mathbb{R}\times (-t_0,t_0)\backslash\{0\}} |H(x,t)| \le C.
  \]
\end{corollary}

The short time existence of a smooth MCF solution starting at \(M_0\) follows by
standard quasilinear parabolic PDE theory.  The challenge here is to establish the
{\em uniform bound}  on \(H(\cdot, t)\) near the singularity \((0,0)\).  For this purpose we
will construct sharp upper and lower barriers  which will capture the exact
behavior of the  profile function \(u(x,t)\) of our solution \(M_t\) as \((x,t)  \to
(0,0)\).  This will be done in section \ref{sec-barriers}.  In section
\ref{sec-existence} we will then construct the profile function  \(u(x,t)\), namely 
a solution of the initial boundary value problem \eqref{eq-u-pde}-\eqref{eq-u-initcond}.  The
boundary condition \(u_x(0,t)=0\) and the fact that \(u >0\) will guarantee that
\(u(x,t)\) defines a smooth MCF solution \(M_t\) which is  \(O(4)\times O(4)\)
symmetric.  In section \ref{sec-Hbounded} we will show that \(H(x,t)\) remains
bounded as \(t \to 0\) near the origin.  The barrier construction in section
\ref{sec-barriers} is based on the formal asymptotic expansion of the profile
solution \(u(x,t)\) as \((x,t) \to (0,0)\).  For the convenience of the reader we
will start by giving this expansion in the next section. 

\section{Formal asymptotic expansion of \(u(x,t)\)}
\label{sec-formal}
We start with Velázquez' construction in \cite{AIV} of a formal asymptotic
expansion of the profile solution \(u(x,t)\) for small \(t>0\).  This
construction motivates our choice of barriers in different regions later in
order to rigorously prove the existence of a mean curvature flow past the
singular time with the following properties.  Our solution before the
singularity at \(t = 0\) coincides with the Velázquez solution constructed in
\cite{Velaz}, it continues as a smooth solution for \(t \in (0, t_1)\), for some
\(t_1 > 0\), and has uniformly bounded mean curvature for all times \(t < 0\),
for which it exists, and all \(t\in (0,t_1)\).

\subsection{Outer variables}
We can approximate any smooth solution for small \(t>0\) by using the Taylor
expansion \(u(x, t) = u(x,0) + t \, u_t(x, 0) + o(t)\).  In view of the
PDE~\eqref{eq-u-pde} this implies that any solution \(u(x, t)\) must satisfy
\begin{equation}
  \label{eq-outer-approximation}
  u(x, t) = u_0(x) + 
  t\,\left\{\frac{u_0''(x)}{1+u_0'(x)^2} +\frac 3x u_0'(x) - \frac 3{u_0(x)}\right\}
  +o(t^2),\qquad (t\to0).
\end{equation}
We will see that under our assumptions \eqref{eq-init-data1}--\eqref{eq-init-data3}
on the initial data,  the expansion \eqref{eq-outer-approximation} holds if
\(x^2 \gg t\).  To describe possible solutions for \(x^2\sim t\) we introduce a new
set of coordinates, the intermediate variables.

\subsection{Intermediate variables}
\label{subsec-formal-intermed}
Consider the function \(v(y,\tau)\) defined by
\begin{equation}
  \label{eq-par-scaling}
  u(x, t) = \sqrt{t} \; v\left(\frac{x}{\sqrt t}, \log t\right).
\end{equation}
It satisfies
\begin{equation}\label{eq-v}
  v_\tau
  = \frac{v_{yy}}{1+v_y^2}
  + \Bigl(\frac 3y + \frac y2\Bigr) v_y
  -\frac v2 - \frac 3v.
\end{equation}
Assuming that \(v(y,\tau)\) is close to the cone, we set
\[
  v(y,\tau) = y +  f(y,\tau),
\]
and compute the equation for \(f\)
\begin{equation}
  \label{eq-f}
  f_\tau = \cL\,f + \cN[f],
\end{equation}
where \(\cL\) is the linear differential operator
\begin{equation}
  \label{eqn-defnL}
  \cL  f \stackrel{\rm def}{=}
  \frac 12  f_{yy} + \left(\frac 3y +\frac y2\right)  f_y 
  + \left(\frac{3}{y^2}- \frac 12 \right) f,
\end{equation}
and where
\begin{equation}
  \label{eqn-defnN}
  \cN[f] \stackrel{\rm def}{=} 
  -3\frac{f^2}{y^2(y+ f)} 
  -\frac{2+ f_y}{1 + (1 + f_y)^2} f_y f_{yy}
\end{equation}
collects the nonlinear terms in the equation for \(f\).  

If we assume that the nonlinear terms are much smaller than the linear terms
then \(f\) should be approximated by a solution of the linear equation \(f_\tau =
\cL f\).  The outer approximation \(u(x, t) = u_0(x) + \cO(t)\) together with the
assumption that the initial function satisfies \(u(x,0) = x + K_0 x^{2(k-1)} +
\cdots\) lead to
\begin{equation}\label{eq-outer-in-intermed}
  v(y, \tau) = y + K_0 e^{(k-\frac 32)\tau} y^{2(k-1)} + \cdots
\end{equation}
for \(y \gg  e^{-\tau/2}\).  
This prompts us to look for approximate solutions of the form 
\begin{equation}
  \label{eq-intermediate}
  v(y,\tau) = y + K_1 e^{(k-\frac 32)\tau} \varphi_k(y)
\end{equation}
where \(\varphi_k\) is a solution of the differential equation
\[
  \cL\varphi_k = \left(k-\frac 32\right)\varphi_k.
\]
It turns out that there are positive and convex solutions of this equation that
are defined for all  \(y>0\).  Their asymptotic behavior for small and large
values of \(y\) is given by
\[
  \varphi_k(y) = \frac{1+o(1)}{y^2} \quad(y\to0),\qquad
  \varphi_k(y) = \frac{1+o(1)}{(2k+1)!!}\, y^{2k-2} \quad (y\to\infty).
\]
In appendix \ref{sec-appendix-linear} we present some more details regarding the
eigenfunctions \(\varphi_k\).

This implies that our intermediate solution \(v(y,\tau)\) from
\eqref{eq-intermediate} is given by 
\[
  v(y,\tau) = y + K_1 e^{(k-\frac 32)\tau}\frac{y^{2(k-1)}}{(2k+1)!!} + \cdots
\]
when \(y\) is large\footnote{Notation: \((2k+1)!! =
1\cdot3\cdot5\cdots(2k-1)\cdot(2k+1)\)}.  Comparing with
\eqref{eq-outer-in-intermed} we see that \(K_0\) and \(K_1\) are related by
\begin{equation}\label{eq-K0-K1-relation}
  K_1 = K_0 \, (2k+1)!!.
\end{equation}

\subsection{Inner variables}
\label{subsec-inner-formal}
One can only expect the intermediate approximation to hold if the nonlinear
terms are small compared with the linear terms.  Since the linear terms are all
of order \(\sim f/y^2\) and the nonlinear terms are of order \(f^2/y^3\) we see that
the nonlinear terms are dominated by the linear terms if \(|f/y|\ll 1\).

When \(y\) is small we have \(f(y,\tau) \sim e^{-(k-3/2)\tau}y^{-2}\), so
\(|f/y|\ll 1\) holds if
\[
  e^{(k-\frac 32)\tau} y^{-3} \gg 1, \quad \text{ i.e. }
  y\ll e^{\left(\frac k3 - \frac 12\right) \tau} = e^{\gamma\tau}
\]
where we abbreviate
\[
  \gamma = \frac k3 -\frac 12.
\]
In the original \((x,t)\) coordinates we have \(y=e^{\gamma\tau}\) exactly if
\(x=t^{k/3}\).

This leads us to introduce the new variable
\[
  z=ye^{-\gamma \tau}  = x  t^{-k/3}
\]
and a new function \(w(z, \tau)\) defined by 
\begin{equation}\label{eqn-v100}
  v(y,\tau) = e^{\gamma\tau}\, w(ye^{-\gamma\tau}, \tau). 
\end{equation}
The equation \eqref{eq-v} is equivalent to 
\begin{equation}
  \label{eq-w}
  \frac{w_{zz}}{1+w_z^2} + \frac 3z w_z - \frac 3w =
  e^{2\gamma \tau}
  \left\{ w_\tau + \frac k3 (w - zw_z) \right\}. 
\end{equation}
For \(\tau\to-\infty\) we assume the terms on the right vanish so it is natural
to look for an approximate solution of the form 
\begin{equation}
  \label{eq-inner-approximation}
  w(z,\tau; K_2) = K_2 W\left(\frac {z}{K_2}\right) + \text{correction terms}
\end{equation}
where \(W(z)\) is Alencar's solution\footnote{Alencar considered
  \(SO(m)\times SO(m)\) invariant minimal surfaces of this type in \cite{Alencar},
  although he mostly considered the cases \(m=2,3\) in that first paper.  Velázquez
  dealt with the case \(m\geq 4\) in \cite{Velaz}, and later Alencar, Barros, Palmas,
  Reyes, and Santos gave a complete classification in \cite{Alencar2005}.} of the
minimal surface equation
\begin{equation}
  \label{eqn-Alencar}
  \frac{W''(z)}{1+W'(z)^2} + \frac{3}{z}W'(z) - \frac{3}{W(z)} = 0.
\end{equation}
By scaling invariance of the minimal surface equation, \(KW(z/K)\), with \(K>0\)
an arbitrary constant, is always a solution of \eqref{eqn-Alencar} if \(W\) is
one.  We choose \(W\) so that it is normalized by 
\begin{equation}\label{eqn-Was} 
  W(z) = z+\frac{1}{z^2} + o(z^{-2}) \qquad (z\to\infty).
\end{equation}
The matching condition for the inner solution \(w(z, \tau) = K_2 W(z/K_2) +
\cdots\) with the intermediate solution \(v(y, \tau) = y + K_1 e^{(k-\frac
32)\tau} \varphi_k(y)+\cdots\) is then
\[
  w(z, \tau) \approx e^{-\gamma\tau} v(e^{\gamma\tau}z, \tau),
\]
i.e.
\[
  z+ \frac{K_2^3}{z^2}  + \cdots = z + K_1 \, \frac{e^{(k-\frac 32)\tau}
  e^{-3\gamma\tau}}{z^2}  + \cdots
  = z + \frac{K_1}{z^2}  + \cdots.
\]
Hence the constants \(K_1\) and \(K_2\) are related by
\begin{equation}
  \label{eq:K-def}
  K_2^3 = K_1=K_0 \, (2k+1)!!
\end{equation}
and our approximate inner solution is given by
\[
  w(z, \tau) = K_1^{1/3} W\bigl(K_1^{-1/3}z\bigr).
\]

\section{Barriers}
\label{sec-barriers}
\subsection{The three regions}
\label{subsec-regions-const}
Our goal in this section is to construct upper and lower barriers for
\begin{equation}
  \tag{\ref{eq-u-pde}}
  u_t = \frac{u_{xx}}{1+u_x^2} + \frac 3x u_x - \frac 3u
\end{equation}
that are valid for all \(x \in (0,+\infty)\) and \(0 < t \leq  t_0\), for some small
enough \(t_0>0\).

To do this we modify the approximate solutions from Section~\ref{sec-formal} in
each of the three regions and glue the resulting locally defined barriers into
one set of globally defined upper and lower barriers.

First we define \emph{the three regions}.  In what follows we regard the three
regions as subsets of space time and use the different sets of coordinates \((x,
t)\), \((y,\tau)\), and \((z, \tau)\) on space time to describe them.

\begin{itemize}[leftmargin= 1em, rightmargin=1em, itemindent= 1em]

  \item For any given \(M>0\) we define the \emph{outer region} to be
    \[
      \cO_{M} = \{(x,t)\mid x \ge M\sqrt{t}, \,\,\,\, 0 < t < M^{-2}\}.
    \]
   {\em  We will assume that } \(M >1\). 
  \item For any \(R>0\) and \(\tau_*\in \R\) we define the \emph{intermediate region}
    to be
    \[
      \cM_{R,\tau_*} = 
      \left\{(y, \tau) \mid 
      R \, e^{\gamma\tau} \leq y \leq e^{-\tau/2}, 
      \tau\leq\tau_* 
      \right\}.
    \]
    Since \(y=x/\sqrt{t}= x\, e^{-\tau/2}\) the intermediate region is defined up to
    \(x=1\), hence the intermediate and outer regions clearly overlap.  

  \item Finally, we declare the \emph{inner region} to be
    \[
      \cI_{Z,\tau_*} = \left\{(z, \tau) \mid 0\leq z\leq Z, \tau\leq \tau_*\right\}.
    \]
    Since \(z=e^{-\gamma\tau}y\) we see that the intermediate and inner regions
    overlap if \(Z>R\).

\end{itemize} 

In section \ref{sec-existence} we will construct a nested sequence of barriers
\[
  u_{\delta_{n-1}}^- < u_{\delta_n}^- < u_{\delta_n}^+ < u_{\delta_{n-1}}^+,
\]
where \(\delta_n = 2^{-n} \, \delta_0\), for some \(\delta_0 >0\).  These barriers will be defined for all \(\tau \leq \tau_{\delta_n} \) where \(\tau_{\delta_n} \to -\infty \) as \(\delta_n \to 0\).  As a result we will see that we need to take \(Z=Z_{\delta_n}\) and \(\tau^*=\tau_{\delta_n}\) in the definitions of the intermediate and inner regions above.  In addition we will see that \(Z_{\delta_n} \to +\infty\) as \(\delta_n \to 0\).

\subsection{Fixing the parameters}
\label{sec-fix-parameters}
From here on we fix the parameters \(k > 3\) and \(K_0 >0\), and we let \(K_1\), \(K_2\) be
defined by \eqref{eq:K-def}.  In all our estimates \(c\) and \(C\) will be generic
constants that can depend \emph{only} on \(k, K_0, K_1\), and \(K_2\).
We use \(C\) in upper bounds, and \(c\) in lower bounds.

\subsection{Barriers in the outer region}
\label{subsec-barriers-outer}

\begin{lemma}
\label{lemma-outer}
For sufficiently large \(M>0\) the functions
\begin{equation}
  \label{eq-outer-upper-barrier}
  u^\pm(x,t) = u_0(x) \pm Mt\min\{1, x^{2k-4}\}
\end{equation}
are super-solution or sub-solution in the outer region \(\cO_M\).
\end{lemma}
\begin{proof}
We only consider the upper barrier \(u^+\).  Similar arguments apply to the lower
barrier.

When \(x>1\) we have \(u^+(x,t) = u_0(x)+Mt\) so that for \(t\in(0, M^{-2})\)
one has $ u^+(x,t) \geq \inf_{x\geq 1} u_0(x) =: c$. This implies
\[
  \left|\frac{u^+_{xx}}{1+(u^+_x)^2}+ \frac 3x u^+_x - \frac 3{u^+}\right| \leq C
\]
for all \(x\geq1\) and \(t\leq M^{-2}\).  Here \(C\) does not depend on \(M\).  On
the other hand \(u^+_t = M\), so for large enough \(M\) we get
\[
  u^+_t\geq\frac{u^+_{xx}}{1+(u^+_x)^2}+ \frac 3x u^+_x - \frac 3{u^+},
\]
i.e.~\(u^+\) is an upper barrier for \(x\geq 1\).

If \(x \geq M \sqrt{t}\) and \(x\leq 1\),  we have \(u^+(x,t) = u_0(x) + M t \, x^{2k-4}\), so that 
\[
  |u^+_{xx}|\leq |u_{0,xx}|+CMt \, x^{2k-6} \leq C\, x^{2k-4} + CMt\, x^{2k-6} \leq
  Cx^{2k-4}.
\]
Similar estimates hold for \(u^+_x-1\) and \(u^+(x,t)-x\), namely,
\[
  x^2|u^+_{xx}| + x|u^+_x-1| + |u^+-x| \leq C\, x^{2k-2}.
\]
Hence
\[
  \frac{|u^+_{xx}|}{1+(u^+_x)^2}\leq C \, x^{2k-4},
\]
and also
\[
  \left|\frac 3x u^+_x - \frac 3{u^+}\right| \leq \frac 3x |u^+_x-1| +
  3\frac{|u^+-x|}{xu^+} \leq Cx^{2k-4}.
\]
Together we get
\[
  \left| \frac{u^+_{xx}}{1+(u_x^+)^2}+ \frac 3x u^+_x - \frac 3{u^+} \right|\leq
  C\, x^{2k-4},
\]
where \(C\) does not depend on \(M\).
On the other hand,   \(u^+_t  = M \, x^{2k-4}\).
Hence,  it now follows that
\(u_0(x)+Mt\, x^{2k-4}\) is an upper barrier if \(M\) is large enough.

Finaly we observe that at  the point \(x=1\) the function \(u^+(x,t)\) has a concave corner, so that
\(u^+(x,t)= u_0(x) + Mt \, \min \{1, x^{2k-4}\}\) is indeed an upper barrier for all
\(x\geq M\sqrt{t}\), \(t<M^{-2}\).

Similar arguments show that \(u^-(x,t)= u_0(x) - Mt \, \min\{1,x^{2k-4}\}\) is a
lower barrier in the same region. The only difference is that one now uses for
$x >1$, \(t\in(0, M^{-2})\) the lower bound
$u^-(x,t) \geq \inf_{x \geq 1} u_0(x) - Mt \geq \frac 12 \, c$, for $M$ sufficiently
large, where $c:= \inf_{x \geq 1} u_0(x) $.

\end{proof}

\subsection{Barriers in the intermediate region}
\label{subsec-inter} We model the upper and lower barriers in the intermediate
region on the approximate solution \(v(y,\tau) = y + f(y,\tau)\) from
\S~\ref{subsec-formal-intermed}, where \(f\) is assumed to be a small function
that satisfies \eqref{eq-f}, i.e.~\(f_\tau = \cL f + \cN[f]\).  A function
\(f\) defines an upper barrier for \eqref{eq-f} in \(\cM_{R,\tau_*} \) if
\begin{equation} \label{eq-barrier-condition}
  f_\tau - \cL f \geq \cN[f]
\end{equation}
holds throughout \(\cM_{R,\tau_*} \).  For a lower barrier the reverse
inequality must hold.

It turns out that the approximate solution \(f_0(y,\tau) =
Ke^{3\gamma\tau}\varphi_k(y)\) is neither a sub- nor super-solution for any
choice of the constant \(K\).  To obtain barriers we therefore add a small
correction term \(f_1(y,\tau)\).  While the resulting function \(f_0(y, \tau) +
f_1(y, \tau)\) does provide a barrier, it does not match the barrier we
construct later in the inner region.  To remedy this we add a second correction
term \(f_2(y, \tau)\).  The resulting barriers \(f_0+f_1+f_2\) will contain a
small parameter \(\delta>0\).  By choosing \(\delta>0\) smaller we get more
accurate barriers, but we also have to reduce the time interval
\(-\infty<\tau\leq \tau_\delta\) on which they are defined.  In the end this
will allow us to prove convergence as \(\tau\to-\infty\) of the actual solution
that we construct using our barriers.

Our construction uses an auxiliary function \(g:(0,\infty)\to\R\), which is the
solution of the following boundary value problem:
\begin{equation}
  \label{eq-g}
  \left\{\;
  \begin{aligned}
    6\gamma g(y) - \cL g(y) &= y^{-7}+y^{4k-7} & (0<y<\infty),\\
    g(y) &= -\frac 13 y^{-5} + o\bigl(y^{-5}\bigr)  & (y\to 0),\\
    g(y) &= y^{4k-7} + o\bigl(y^{4k-7}\bigr) &(y\to\infty).
  \end{aligned}
  \right.
\end{equation}
The choice of forcing term in the equation for \(g\) above will become apparent
in what follows.  In \S~\ref{sec-g-exists} we prove:
\begin{lemma}\label{lem-g} The
equations \eqref{eq-g} have a unique solution \(g:(0,\infty)\to\R\).
\end{lemma}

Assuming  that Lemma \ref{lem-g} holds, we  look for barriers in the following family of functions,
\begin{equation}\label{eq-barriers-med-v}
  v_\delta^\pm(y,\tau) = y + f_\delta^\pm(y,\tau)
\end{equation}
where
\begin{equation}
  \label{eq-barriers-med}
  f_\delta^\pm(y, \tau)
  = f_0^\pm(y, \tau, \delta) \pm \left\{f_1(y, \tau) + f_2(y, \tau)\right\}
\end{equation}
and
\begin{equation}\label{eq-fj-defined}
  \begin{aligned}
    f_0^\pm(y,\tau, \delta) &= \bigl(K_1 \pm \delta\bigr) e^{3\gamma\tau}\varphi_k(y) \\
    f_1(y,\tau) &= BK_1^2 e^{6\gamma\tau}g(y) \\
    f_2(y,\tau) &= e^{(p+1)\gamma\tau} y^{-p}.
  \end{aligned}
\end{equation}
Here, as in \S~\ref{sec-fix-parameters}, we have \(K_1 = (2k+1)!!K_0\), while
\(B, \delta > 0\) and \(p\in(2, 3)\) are parameters.

\begin{proposition} \label{lem-intermed-barriers} There exist \(B_*\), \(R_*\),
  and \(\tau_*\) that only depend on \(k, K_0\) such that for all
  \(\delta\in(0,\frac{1}{2}K_1)\), \(p\in(2,3)\),  the functions
  \(f_\delta^\pm\) defined in \eqref{eq-barriers-med}--\eqref{eq-fj-defined}
  are upper and lower barriers in the intermediate region \(\cM_{R_*, \tau_*}\).
  It follows that the functions \( v^\pm_\delta \) defined in
  \eqref{eq-barriers-med-v} are upper and lower barriers for equation
  \eqref{eq-v} in \(\cM_{R_*,\tau_*}\).
\end{proposition}

We begin with two lemmas that will simplify the proof of
Proposition~\ref{lem-intermed-barriers}.
\begin{lemma}\label{prop-Nf-pre-estimate}
Wherever \(f(y, \tau)\geq 0\) holds, one has
\[
  \big| \cN[f] \big|\leq \frac{3}{y^3} [f]_2^2,
\]
where, by definition, for any function \(F(y,\tau)\) we define
\begin{equation}\label{eqn-not2}
  [F]_2(y,\tau) := |F(y,\tau)| + |y F_y(y,\tau)| + |y^2 F_{yy}(y,\tau)| \,.
\end{equation}

\end{lemma}
\begin{proof}
Using \(2|1+x|\leq 1+(1+x)^2\) one finds for all \(x\in\R\)
\[
  \left|\frac{2+x}{{1+(1+x)^2}}\right|\leq
  \frac{1}{1+(1+x)^2} + \frac{|1+x|}{1+(1+x)^2}
  \leq\frac 32.
\]
Using \(f(y,\tau)\geq 0\) this implies 
\begin{align*}
  \big| \cN[f] \big|
  &= \left| \frac{-3f^2}{y^2(y+f)} - \frac{2+f_y}{1+(1+f_y)^2}f_yf_{yy}
  \right|\\
  &\leq 3\,\frac{f^2}{y^3} + \frac 32|f_y f_{yy}| \\
  &\leq \frac 3{y^3}\bigl\{f^2 + |yf_y|\,|y^2f_{yy}|\bigr\} \\ 
  & \leq \frac{3}{y^3}[f]_2^2 \,.
\end{align*}
\end{proof}

\begin{lemma}\label{prop-Nf-estimate}
For any \(B\) there exist \(R(B)>0\) and \(\tau(B) \in \R\) such that if \(0<\delta<\frac 12 K_1\), then \(f_\delta^\pm\) as defined in \eqref{eq-barriers-med}--\eqref{eq-fj-defined}, satisfies
\[
  f_\delta^\pm(y, \tau) > 0
\]
and
\[
  \big|\cN[f_\delta^\pm]\big| \leq C_* e^{6\gamma\tau} \bigl(y^{-7} +
  y^{4k-7}\bigr)
\]
in the intermediate region \(R(B)e^{\gamma\tau}\leq y\leq e^{-\tau/2}\),
\(\tau\leq \tau(B)\).  
\end{lemma}
As promised in section~\ref{sec-fix-parameters}, the constant \(C_*\) only depends
on the constants \(k, K_0\) but not on \(B\).

\begin{proof}

Recall the notation from \eqref{eqn-not2}. The explicit expression \eqref{eq-phik} for \(\varphi_k\) implies
\[
  [\varphi_k]_2\leq C y^{-2}\bigl(1+y^{2k}\bigr),
\]
and the construction of the auxiliary function \(g\) implies
\[
  [g]_2 \leq C y^{-5}\left(1+y^{4k-2}\right).
\]
We also have for all \(y>0\)
\[
  \left[y^{-p}\right]_2 = y^{-p}+py^{-p}+p(p+1)y^{-p} =(p+1)^2 y^{-p} < 16
  y^{-p},
\]
because \(2<p<3\).  Hence the three terms \(f_j\) in \eqref{eq-fj-defined} that
add up to \(f_\delta^\pm\) satisfy
\begin{align*}
  [f_0]_2 & \leq C e^{3\gamma\tau}y^{-2}\bigl(1+y^{2k}\bigr)\\ 
  [f_1]_2 & \leq CB e^{6\gamma\tau} y^{-5} \bigl(1+y^{4k-2}\bigr) \\
  [f_2]_2 & \leq C e^{(p+1)\gamma\tau}y^{-p} ,
\end{align*}
assuming that \(0<\delta \leq \frac 12 K_1\).

If \(Re^{\gamma\tau}\leq y\leq e^{-\tau/2}\), then we can estimate
\(f_\delta^\pm\) as follows
\begin{align*}
  \left[f_\delta^\pm\right]_2
  &\leq  C \frac{e^{3\gamma\tau}}{y^2}\bigl(1+y^{2k}\bigr)  +
  CB  \frac{e^{6\gamma\tau}}{y^5}  \bigl(1+y^{4k-2}\bigr) +
  C \frac{e^{(p+1)\gamma\tau}}{y^{p}} 
  \\ 
  &\leq C  \frac{e^{3\gamma\tau}}{y^2} \bigl(1+y^{2k}\bigr)  
  \left\{ 1  + B \frac{e^{3\gamma\tau}}{ y^3}
  + B e^{3\gamma\tau}  y^{2k-5} 
  +  \frac{e^{(p-2)\gamma\tau}}{y^{p-2}}
  \right\}  \\ 
  &\leq C  \frac{e^{3\gamma\tau}}{y^2} \bigl(1+y^{2k}\bigr)
  \left\{1 + BR^{-3} + B  e^\tau + R^{-(p-2)}  \right\}, 
\end{align*}
where in estimating the third term in the bracket we used \(3\gamma = k-3/2\). 
Thus, if we require 
\begin{equation}
  \label{eq-RB-tauB-first-choice}
  R \geq \max\{ 1, B^{1/3}\} \text{ and } \tau \leq \tau(B) :=  - \log  B
\end{equation}
then \(1 + BR^{-3} + B  e^\tau + R^{-(p-2)}\leq 4 \) and so
\[
  [f_\delta^\pm]_2 \leq C e^{3\gamma\tau}y^{-2}\bigl(1+y^{2k}\bigr).
\]
Combined with Lemma~\ref{prop-Nf-pre-estimate} this yields
\[
  \big|\cN[f_\delta^\pm]\big|
  \leq \frac{3}{y^3} C e^{6\gamma\tau}y^{-4}\left(1+y^{2k}\right)^2
  \leq \tilde C e^{6\gamma\tau}y^{-7}
  \left(1+y^{4k}\right)
\]
in the intermediate region, provided that we verify \(f_\delta^\pm\geq 0\) when \(Re^{\gamma\tau}\leq y \leq e^{-\tau/2}\).

To prove \(f_\delta^\pm \geq 0\) in the intermediate region we recall the
assumption \(\delta<\frac 12 K_1\), which implies
\[
  f_\delta^\pm(y,\tau)
  \geq \frac{1}{2}K_1 e^{3\gamma\tau}\varphi_k(y)
  - \left\{BK_1^2 e^{6\gamma\tau}|g(y)| + e^{(p+1)\gamma\tau}y^{-p}\right\}. 
\]
Use the lower bound \(\varphi_k(y) \geq c y^{-2}(1+y^{2k})\), and the upper bound
\(|g(y)|\leq Cy^{-5}(1+y^{4k-2})\) to arrive at
\[
  f_\delta^\pm(y,\tau)
  \geq c \frac{e^{3\gamma\tau}}{y^{2}}\left(1+y^{2k}\right) 
  - \left\{CB \frac{e^{6\gamma\tau}}{y^{5}} \left(1+y^{4k-2}\right)  
  + \frac{e^{(p+1)\gamma\tau}}{y^{p}}\right\}  ,
\]
which, because \(\frac{1+xy}{1+x}\leq 1+y\) for all \(x,y\geq0\), implies
\[
  \frac {y^2e^{-3\gamma\tau}}{c\,(1+y^{2k})}f_\delta^\pm(y,\tau)
  \geq 1 - {CB}\frac{e^{3\gamma\tau}}{y^{3}}\left(1+y^{2k-2}\right)
  -\frac{1}{c(1+y^{2k})} \frac{e^{(p-2)\gamma\tau}}{y^{p-2}}.
\]
In the region \(Re^{\gamma\tau} \leq y \leq e^{-\tau/2}\) we get
\[
  \frac {y^2e^{-3\gamma\tau}}{ c\,(1+y^{2k})}f_\delta^\pm(y,\tau)
  \geq 1 - \frac{CB}{R^3} - CB  e^{\tau} -\frac{1}{cR^{p-2}}.
\]
We adjust our choice of \(R(B), \tau(B)\) in \eqref{eq-RB-tauB-first-choice} to
\begin{equation}
  \label{eq-RB-tauB-final-choice}
  R(B) = \tilde{C}\max\{1, B^{1/3}\}, \qquad \tau(B) = -\log(\tilde{C}B)
\end{equation}
for large enough \(\tilde{C}\geq1\).
Then, for \(y\geq R(B)\) and \(\tau\leq \tau(B)\), we have
\[
  \frac {2y^2e^{-3\gamma\tau}}{c(1+y^{2k})}f_\delta^\pm(y,\tau)
  \geq \frac 12 >0,
\]
and thus \(f_\delta^\pm(y, \tau)>0\).
\end{proof}

\begin{proof} [Proof of Proposition~\ref{lem-intermed-barriers}]
We consider the case of upper barriers, where we have
\begin{equation}
  \label{eq-lemma3.3-f}
  \bigl(\partial_\tau - \cL\bigr)f_{\delta}^+ = \bigl(\partial_\tau - \cL\bigr)f_0^+
  + \bigl(\partial_\tau - \cL\bigr)f_1 + \bigl(\partial_\tau - \cL\bigr)f_2.
\end{equation}
The first term vanishes because \(f_0^\pm\) is a solution of the linear equation \(f_\tau = \cL f\).  For the last term in \eqref{eq-lemma3.3-f} we note that for any \(r\in\R\) one has
\[
  \cL[y^r] = \frac 12(r+2)(r+3) y^{r-2} + \frac 12 (r-1) y^r.
\]
Hence, if \(p\in(2, 3)\) then \(\cL[y^{-p}] < 0\) for all \(y>0\).  It follows that
\[
  \bigl(\partial_\tau-\cL\bigr)f_2 > \partial_\tau f_2 = (p+1)\gamma f_2 >0.
\]
The middle term in \eqref{eq-lemma3.3-f} satisfies
\[
  (\partial_\tau-\cL)f_1 =BK_1^2e^{6\gamma\tau}\bigl(6\gamma g - \cL g\bigr)
  =BK_1^2e^{6\gamma\tau}\bigl(y^{-7} + y^{4k-7}\bigr).
\]

If we choose \(B_*=C_*K_1^{-2}\) where \(C_*\) is the constant from
Lemma~\ref{prop-Nf-estimate},  and if we set \(R_*=R(B_*)\), \(\tau_*=\tau(B_*)\) according to \eqref{eq-RB-tauB-final-choice},  then we clearly have
\(\bigl(\partial_\tau-\cL\bigr) f_\delta^+ > \cN[f_\delta^+]\) in the
intermediate region \(\cM_{R_*,\tau_*}\).

We conclude that \(f^+_\delta\) is an upper barrier, i.e.~equation
\eqref{eq-barrier-condition} holds.  With minor modifications this argument also shows
that \(f^-_\delta\) is a lower barrier.
\end{proof}

We next show that the barriers \(f^\pm_\delta\) form a nested sequence, in the
sense of the lemma below.  The nesting of barriers will allow us to construct a
solution that is bounded by all barriers at once and will   enable us to prove
the   convergence of our solution in the inner region  to the Alencar minimal
surface, as \(\tau \to -\infty\).
\begin{lemma}\label{lem-intermed-barrier-ordering} 
The constant \(R_*\) from Proposition~\ref{lem-intermed-barriers}
can be chosen so that 
\begin{equation} \label{eq-intermed-barrier-ordering}
  f_\delta^-(y, \tau)
  < f_{\delta/2}^-(y, \tau)
  < f_{\delta/2}^+(y, \tau)
  < f_\delta^+(y, \tau)
\end{equation}
for all \((y,\tau) \) with \(R_*e^{\gamma\tau}\leq y\).
\end{lemma}
\begin{proof}
We can write the barrier functions \(f_\delta^\pm\) as
\[
  f_\delta^\pm(y, \tau)
  = K_1 e^{3\gamma\tau}\varphi_k(y) \pm
  \left\{ 
  \delta e^{3\gamma\tau}\varphi_k(y) + B_*K_1^2 e^{6\gamma\tau}g(y) + e^{(p+1)\gamma \tau}y^{-p}
  \right\}. 
\]
Since \(\varphi_k(y) > 0\) for all \(y>0\),  it is immediately clear that 
\[
  f^-_\delta(y, \tau) < f^-_{\delta/2}(y, \tau) 
  \text{ and }
  f^+_{\delta/2}(y, \tau) < f^+_\delta(y, \tau)
\]
for all \(y, \tau\).

To prove the middle inequality we note that \(f^-_{\delta/2}(y, \tau) <
f^+_{\delta/2}(y, \tau)\) holds if and only if 
\[
  \frac\delta2 e^{3\gamma\tau}\varphi_k(y) 
  +  B_*K_1^2 \,  e^{6\gamma\tau}g(y) 
  +  e^{(p+1)  \gamma  \tau}  y^{-p}
  > 0, 
\]
which, in view of \(\varphi_k(y)>0\) will certainly hold if
\begin{equation}\label{eq-nesting1}
  B_*K_1^2 \,  e^{6\gamma\tau}g(y) 
  +  e^{(p+1)  \gamma  \tau}  y^{-p}
  > 0.
\end{equation}
Since \(g(y)>0\) for large \(y>0\), there is a constant \(C_g>0\) such that
\(g(y)\geq -C_g y^{-5}\) for all \(y>0\).  Hence \eqref{eq-nesting1} follows from
\[
  e^{(p+1)  \gamma  \tau} y^{-p} -  C_g B_*K_1^2  e^{6\gamma\tau}y^{-5} > 0,
  \, \,\text{ i.e. }
  y e^{-\gamma\tau} > \left( C_g B_*K_1^2 \right)^{1/(5-p)}.
\]
\end{proof}

\subsection{Barriers in the inner region}
\label{subsec-inner}
In this section we present a family of sub- and super-solutions to the
equation~\eqref{eq-w} for \(w(z,\tau)\) in the inner region \(0\leq z\leq Z\).

We recall our notation from section \ref{subsec-inner-formal} where \(W(z)\)
denotes the unique Alencar solution to~\eqref{eqn-Alencar}, normalized so that
\begin{equation}\label{eqn-Was}
  W(z) = z + \frac{1}{z^2} + \frac{\Gamma}{z^3} + \cO\bigl(z^{-5}\bigr) \qquad
  (z\to\infty)
\end{equation}
holds for certain constant \(\Gamma\in\R\).  

\begin{lemma}\label{lem-WK-starshaped}
For all \(z > 0\) one has \( W_K(z)>zW_K'(z)\).
\end{lemma}
\begin{proof} The inequality is invariant under rescaling, so we may assume \(K=1\).
The asymptotics \eqref{eqn-Was} show that \(W(z)-zW_z(z) \to 0\) as \(z\to\infty\).  On
the other hand, convexity of \(W\) implies \((W-zW_z)_z = - z W_{zz} < 0\) for all
\(z > 0\).  Hence \( W(z)-zW_z(z) > \lim_{Z\to\infty} W(Z)-ZW_z(Z) = 0\) for all
\(z\geq 0\).
\end{proof}

\begin{lemma}\label{lemma-inner-barrier1} For any \(K >0\) function \(w^+(z,\tau)
= W_{K}(z)\) is a super-solution of equation \eqref{eq-w} on
\([0,\infty)\times\R\).
\end{lemma}
\begin{proof}
The function  \(w^+\) satisfies \(w^+_\tau=0\) and 
\[
  \frac{w^+_{zz}}{1+(w^+_z)^2} + \frac{3}{z} w^+_z - \frac{3}{w^+}=0. 
\]
From Lemma~\ref{lem-WK-starshaped} we have \(w^+-zw^+_{z}>0\), and thus
\[
  e^{2\gamma\tau}\left(w^+_\tau + \frac k3(w^+-zw^+_z)\right) >  
  \frac{w^+_{zz}}{1+(w^+_z)^2} + \frac{3}{z}w^+_z - \frac{3}{w^+}
\]
as claimed.
\end{proof}

\begin{lemma}\label{lemma-inner-barrier2}
There exist \(D_*>0\), \(\zeta>0\) such that for all
\(K\in (\frac12 K_2, 2K_2)\), and \(D\geq D_*\) there is a \(\tau_*(D)\) such that
\[
  w^-(z, \tau) := W_{K}(z) +  D\, e^{2\gamma\tau}
\]
is a sub-solution of \eqref{eq-w} for \(0\leq z\leq \zeta e^{-\gamma\tau}\),
\(\tau \leq \tau_*(D)\).
\end{lemma}
\begin{proof}
Choose
\[
  \tau_*(D) \leq \frac1{2\gamma}\log \frac{W_K(0)}{D}.
\]
Then \(\tau\leq \tau_*(D)\) and \(z\geq 0\) implies
\[
  De^{2\gamma\tau} \leq W_K(0)\leq W_K(z)
\]
so that
\[
  W_K(z) \leq w^-(z, \tau) \leq 2W_K(z).
\]

If we substitute \(w=w^-\) in \eqref{eq-w} and use \(2\gamma+\frac 13k =
k-1\), then on one hand
\[
  e^{2\gamma\tau}\Big (w_\tau^- + \frac k3(w^- -zw_z^-)\Big) =
  e^{2\gamma\tau}\Big(
  (k-1) D e^{2\gamma\tau} + \frac k3 \bigl(W_{K}-z\, W_{K}' \bigr)
  \Big),
\]
and on the other hand, 
\[
  \frac{w_{zz}^-}{1+(w_z^-)2} + \frac{3}{z}w_z^- - \frac{3}{w^-} =\frac{W_{K}''}{1+(W_{K}')^2} +
  \frac{3}{z}W_{K}' - \frac{3}{w^-} =\frac{3}{W_{K}}-\frac{3}{w^-}
  =\frac{3De^{2\gamma\tau}}{W_{K} \,  w^-}.
\]
Hence \(w^-\) is a sub-solution if
\begin{equation} \label{eq-inner-sub-proof}
\frac{3D}{W_{K}(z) w^-(z,\tau)} > (k-1) De^{2\gamma\tau} + \frac k3
\bigl(W_{K}(z)-zW_{K}'(z)\bigr). 
\end{equation}
Since \(W_K\leq w^-\leq 2W_K\leq C(1+z)\) there is a constant \(C_1\) such that the
terms on the left are bounded from below by
\[
  \frac{3D}{W_{K}(z) w^-(z,\tau)}\geq \frac{C_1D}{(1+z)^2}.
\]
The terms on the right in \eqref{eq-inner-sub-proof} satisfy
\[
  (k-1)e^{2\gamma\tau}\leq C_2 \frac{\zeta^2}{(1+z)^2} 
\]
in the region \(1+z\leq \zeta e^{-\gamma\tau}\), and, due to the asymptotic expansion
of \(W_{K}(z)\) as \(z\to\infty\) (which follows from \eqref{eqn-Was}), they also
satisfy
\[
  W_{K}(z)-zW_{K}'(z) \leq \frac{C_3}{(1+z)^2} \quad \text{ for all } \,\, z\geq 0.
\]
Hence
\[
  (k-1) De^{2\gamma\tau} + \frac k3 \bigl(W_{K}(z)-zW_{K}'(z)\bigr)
  \leq \frac{C_2\zeta^2D +C_3}{(1+z)^2}.
\]  
Choose  \(\zeta < \sqrt{C_1/2C_2}\), and choose \(D\) so large that
\(C_3 < \frac12 C_1D \).  Then we have
\[
  (k-1) De^{2\gamma\tau} + \frac k3 \bigl(W_{K}(z)-zW_{K}'(z)\bigr)
  < \frac{C_1D}{(1+z)^2} \leq \frac{3D}{W_K(z)w(z,\tau)},
\]
which implies \eqref{eq-inner-sub-proof}, and thus that \(w^-\) is a lower barrier in
the region \(1+z\leq \zeta e^{-\gamma\tau}\).  Choose \(\tau_*\) so that
\(\zeta e^{-\gamma\tau_*}\geq 2\).  Then \(1+z\leq \zeta e^{-\gamma\tau}\) holds for
all \(z\leq 1\) and \(\tau\leq \tau_*\), while for \(z\geq 1\) it follows from
\(2z\leq \zeta e^{-\gamma\tau}\) that  \(1+z\leq \zeta e^{-\gamma\tau}\).

Thus \(w^-\) is a lower barrier in the region \(z\leq \frac12\zeta e^{-\gamma\tau}\),
\(\tau\leq \tau_*\).
\end{proof}

\subsection{Matching outer and intermediate barriers}
\label{subsec-outer-inter}
We show that upper and lower barriers constructed in the inner, the intermediate, and
the outer regions match in the overlapping region.  We begin here with the overlap of
the outer and intermediate regions.

We start with an \(M>0\) sufficiently large so that the functions
\( u^\pm(x,t) = u_0(x) \pm Mt\min\{1, x^{2k-4}\}\) are sub- and super-solutions  of
\eqref{eq-u-pde} in the outer region \(\cO_M\) (see Lemma \ref{lemma-outer}).  In
order to match the outer barriers with the barriers in the intermediate region, we
express the outer barriers \(u=u^\pm(x, t)\) in the intermediate variables
\((v,y,\tau)\):
\[
  v_{\rm out}^\pm(y, \tau) \stackrel{\rm def}= e^{-\tau/2} u^\pm(e^{\tau/2}y, e^\tau).
\]
In \eqref{eq-outer-upper-barrier} we defined \(u^\pm(x, t) = u_0(x) \pm
Mt\, x^{2k-4} \) for \(0 < x\leq 1\).   If we write the assumption \eqref{eq-init-data1} on
the initial data in the form
\begin{equation}\label{eqn-uK0}
  u_0(x) = x + \bigl(K_0+\epsilon_0(x)\bigr) x^{2k-2},
\end{equation}
where \(\epsilon_0:(0,\infty)\to\R\) satisfies \(\lim_{x\to 0}\epsilon_0(x)=0\),
then we get the following expression for the outer barriers in the intermediate
variables:
\begin{equation}
  \label{eq-v-pm-out}
  v_{\rm out}^\pm(y,\tau) 
  = y + \left(K_0 + \epsilon_0(ye^{\tau/2})\right) e^{3\gamma\tau}y^{2k-2}
  \pm Me^{3\gamma\tau}y^{2k-4}.
\end{equation}
The outer barriers only contain the parameter \(M\) and thus do not depend on other
parameters such as \(\delta, B\) that appeared in the barriers we constructed for the
intermediate and inner regions.

We now consider the intermediate barriers, continuing to use the conventions from
Section~\ref{sec-fix-parameters} which relate the constants \(K_0, K_1\), etc.

In Proposition~\ref{lem-intermed-barriers} we found \(B_*\), \(R_*\), and \(\tau_*\),
such that for any \(\delta\in(0,\frac{1}{2}K_1)\) and \(p \in (2,3)\) the functions
\[
  v^{\pm}_\delta(y,\tau) = y+ (K_1\pm \delta) e^{3\gamma\tau} \varphi_k(y) 
  \pm \left\{e^{(p+1)\gamma\tau} y^{-p} + B_*K_1^2 e^{6\gamma\tau} g(y)\right\},
\]
are upper and lower barriers in the intermediate region
\(\cM_{R_*,\tau_*} = \{R_* e^{\gamma\tau} \leq y \leq e^{-\tau/2}, \tau \leq
\tau_*\} \).

To compare \(v_{\rm out}^\pm\) and \(v_\delta^\pm\) we rewrite them as
\begin{align*}
  e^{-3\gamma\tau}\left(v_{\rm out}^\pm(y, \tau) - y\right) 
  &= \left(K_0+\epsilon_0(ye^{\tau/2})\right) y^{2k-2} 
  \pm M y^{2k-4}\\
  e^{-3\gamma\tau}\left(v_{\delta}^\pm(y, \tau) - y\right) 
  &= (K_1\pm\delta)\varphi_k(y)  \pm e^{(p-2)\gamma \tau}y^{-p} \pm B_*K_1^2 e^{3\gamma\tau}g(y).
\end{align*}
We now let \(\tau\to-\infty\) and conclude that 
\begin{equation}
  \label{eq-vpm-limits}\left\{\;
  \begin{aligned}
    e^{-3\gamma\tau}\left(v_{\rm out}^\pm(y, \tau) - y\right) 
    &\to K_0 \, y^{2k-2} \pm M y^{2k-4}\\
    e^{-3\gamma\tau}\left(v_{\delta}^\pm(y, \tau) - y\right) 
    &\to \left(K_1\pm \delta\right) \varphi_k(y)
  \end{aligned}
  \right.
\end{equation}
uniformly for bounded \(y\).

The explicit expression \eqref{eq-phik} for \(\varphi_k\) implies
\[
  \varphi_k (y) = \frac{y^{2k-2}}{(2k+1)!!} + c(y)y^{2k-4}
\]
where 
\[
  c(y) = c_0 + \frac{c_1}{y^2}+\cdots+ \frac{c_{k-1}}{y^{2k-2}},
  \qquad
  c_j = \frac{\binom{k}{j+1}}{(2(k-j)-1)!!}.
\]
Substitute this expression for \(\varphi_k\) in \eqref{eq-vpm-limits} and keep in mind that \(K_1 = (2k+1)!! K_0\).  Then
\[
  e^{-3\gamma\tau}\left(v_{\rm out}^\pm(y, \tau) - v_\delta^\pm(y, \tau)\right)
  \to
  \pm y^{2k-4} \left\{
    -\frac{\delta y^2}{(2k+1)!!} + M -  c(y)
    \right\}.
\]
The function \(c(y)\) is clearly bounded for \(y\geq 1\) so if \(M \) is sufficiently
large, one can neglect \(c(y)\) and conclude that
\(v^\pm_{\rm out}(y, \tau)-v_\delta^\pm(y, \tau)\) changes sign when
\[
  \frac{\delta y^2}{(2k+1)!!}=M-c(y)\approx M.
\]
To make this more precise we introduce \(Y_{\delta} := 2\,\sqrt{(2k+1)!!M/\delta}\) and
compare the barriers \(v_{\rm out}^\pm(y, \tau) \) and \(v_\delta^\pm(y, \tau)\) at the
endpoints \(y_{\delta}(\tau) \in (\frac 14 Y_{\delta},Y_{\delta})\).

\begin{lemma} \label{lemma-match-outer-inter}
For any \(\delta > 0\) there is a \(\tau_\delta\in\R\) such that for all \(\tau
\leq \tau_\delta\) one has
\[
  v_{\rm out}^+(Y_{\delta}/4, \tau) > v_\delta^+(Y_{\delta}/4, \tau)\text{ and }
  v_{\rm out}^-(Y_{\delta}/4, \tau) < v_\delta^-(Y_{\delta}/4, \tau).
\]
Moreover,  we also have
\[
  v_{\rm out}^+(Y_\delta, \tau) < v_\delta^+(Y_\delta, \tau)\text{ and }
  v_{\rm out}^-(Y_\delta, \tau) > v_\delta^-(Y_\delta, \tau)
\]
for all \(\tau\leq \tau_\delta\).
\end{lemma}

\begin{proof}
We only consider the upper barriers, the other case being nearly identical.

We have found that as \(\tau\to-\infty\)
\[
  e^{-3\gamma\tau}\left(v_{\rm out}^+ (Y_{\delta}/4, \tau) - v_\delta^+ (Y_{\delta}/4, \tau)\right)
  \to
  (Y_{\delta}/4)^{2k-4} \Big\{
    -\frac M4 + M -  c(M) 
    \Big\}. 
\]
Since \(c(y)\) is bounded for \(y\geq 1\), given any large \(M\) we will still have
\[
  \frac{3M}{4} - c(M) > 0.
\]
Hence 
\[
  \lim_{\tau\to-\infty}
  e^{-3\gamma\tau}\left(v_{\rm out}^+ (Y_{\delta}/4, \tau) - v_\delta^+ (Y_{\delta}/4, \tau)\right) >
  0,
\]
which implies that for \(-\tau\) sufficiently large one has \(v_{\rm
out}^+(Y_{\delta}/4, \tau) > v_\delta^+(Y_{\delta}/4, \tau)\), as claimed.

If on the other hand we compare \(v_{\rm out}^+\) and \(v_\delta^+\) at
\(y=Y_\delta\), then we find that for \(\tau\to-\infty\)
\begin{align*}
  e^{-3\gamma\tau}\left(v_{\rm out}^+ (Y_\delta, \tau) - v_\delta^+ (Y_\delta, \tau)\right)
  &\to Y_\delta^{2k-4}\left\{ -4M + M  -   c(Y_\delta) \right\}\\ 
  &=-Y_\delta^{2k-4}\left\{ 3M  +  c(Y_\delta) \right\}.
\end{align*}
Since \(c(y)\) is bounded for \(y\geq 1\),  it follows  that for \(M\) large enough
we indeed have \(v_{\rm out}^+(Y_\delta, \tau) < v_\delta^+(Y_\delta,\tau)\), as
\(\tau\to-\infty\).
\end{proof}

\subsection{Matching intermediate and inner barriers}
\label{subsec-inner-inter}

For any \(\delta \in (0,\frac 12 K_1)\), \(p \in (2,3)\) and \( B = B_*\) the barriers
\(v_\delta^\pm(y, \tau) = y+f_\delta^\pm(y, \tau)\) constructed above are defined in
the intermediate region
\(\cM_{R_*, \tau_*}=\{R_*e^{2\gamma\tau}\leq y\leq e^{-\tau/2}, \tau\leq \tau_*\}\).
If we assume that \(Z > 2R_*\), then it follows \( v_\delta^\pm(y, \tau)\) are defined in
parts of the inner region
\( \cI_{Z,\tau_*} = \left\{(z, \tau) \mid 0\leq z\leq Z, \tau\leq \tau_*\right\}\).
Define
\[
  w^\pm_{\rm md}(z, \tau) := e^{-\gamma\tau} v_\delta^\pm\left(e^{\gamma\tau}z,
  \tau\right).
\]
Then
\begin{align*}
  w^\pm_{\rm md}(z,\tau)
  & = z + \frac{K_1\pm\delta}{z^2} \bigl(1+\epsilon_1(z,\tau)\bigr)
  \pm \frac{1}{z^p} \pm \frac{B_*K_1^2}{z^5} \bigl(1+\epsilon_2(z,\tau)\bigr)
\end{align*}
where \(\epsilon_i(z,\tau)\) are generic functions for which
\(\epsilon_i(z,\tau) \to 0\) as \(\tau \to -\infty\), uniformly for \(0\leq z\leq Z\).  In
particular, for all \(z\in[0,Z]\) we have
\begin{equation}\label{eqn-wa14}
  \lim_{\tau\to-\infty} w_{\rm md}^\pm(z,\tau)
  = z + \frac{K_1}{z^2} \pm \left\{\frac{\delta}{z^2}
  + \frac{1}{z^p} + \frac{B_*K_1^2}{z^5} \right\}.
\end{equation}

We will now use Lemmas~\ref{lemma-inner-barrier1}
and~\ref{lemma-inner-barrier2} to match \(w_{\rm md}^\pm(z,\tau)\) with
appropriately chosen barriers \(w^\pm_\delta(z,\tau)\) in the inner region \(0
\leq z \leq Z\).  For suitable $\delta$-dependent  constants \(K_2^\pm\in(\frac12 K_2,
2K_2)\), with \((K_2)^3=K_1\), we consider
\[
  w_{\delta}^+(z,\tau)  \stackrel{\rm def}= W_{K_2^+}(z),\qquad
  w_{\delta}^-(z,\tau)  \stackrel{\rm def}= W_{K_2^-}(z) + D\, e^{2\gamma\tau}
\]
where \(D\) depends on \(K_2^-\) and \(Z\) as described in Lemma~\ref{lemma-inner-barrier2}. 

It follows from Lemmas~\ref{lemma-inner-barrier1},~\ref{lemma-inner-barrier2},
that for each \(K_2^+ >0\) and \(K_2^->0\), \(w_\delta^+\) and \(w_\delta^-\) are upper
barrier and lower barriers for \eqref{eq-w} in the inner region.  Furthermore
the asymptotics at infinity of the Alencar solution in \eqref{eqn-Was} imply
that 
\[
  \lim_{\tau \to -\infty} w_{\delta}^\pm(z,\tau)
  = z + \frac{(K_2^\pm)^3}{z^2} + \frac{\Gamma (K_2^\pm)^4}{z^3} + \cO(z^{-5})
  \quad (z \gg 1).
\]
Comparing the asymptotic  expansions of \(  w^\pm_{\rm md}\) and \(w_{\delta}^\pm\)
we see that they match when \((K_2^\pm)^3=K_1\pm \delta\).  However with this
choice the barriers \( w^\pm_{\rm md}\) and \(w_{\delta}^\pm\) may not intersect.
For this reason we choose the constants \(K_2^\pm\) such that 
\[
  (K_2^\pm)^3=K_1\pm 2\delta.
\]
With this choice we then have
\begin{equation}\label{eqn-wa15}
  \lim_{\tau \to -\infty} w^\pm_\delta(z,\tau)
  = z + \frac{K_1\pm 2\delta}{z^2} + \frac{\Gamma (K_1\pm 2\delta)^{4/3}}{z^3} + \cO(z^{-5})
  \quad (z \gg 1).
\end{equation}

\begin{lemma} \label{lemma-match-middle-inner} Let \(p \in (2,3)\) be given, and let
\(B = B_k \) as in Proposition~\ref{lem-intermed-barriers}.  Then there exist
\(\bar\delta>0\) and \(R=R(B)\) so that for any \(\delta \in(0,\bar\delta)\) and
\(\tau\leq \tau_\delta\) the barriers \(w^\pm_\delta\) and \(w^\pm_{\rm md}\) cross in
the interval \(\big(\frac{1}{2}Z_{\delta}, Z_{\delta}\big)\), where
\(Z_\delta:= \frac 43 \, \delta^{\frac{-1}{p-2}}\), in the sense that
\[
  w_{\rm md}^+(Z_\delta/2, \tau) > w_\delta^+(Z_\delta/2, \tau)\quad \text{ and } \quad 
  w_{\rm md}^-(Z_\delta/2, \tau) < w_\delta^-(Z_\delta/2, \tau).
\]
and 
\[
  w_{\rm md}^+(Z_\delta, \tau) < w_\delta^+(Z_\delta, \tau) \quad \text{ and } \quad 
  w_{\rm md}^-(Z_\delta, \tau) > w_\delta^-(Z_\delta, \tau). 
\]  
\end{lemma}

\begin{proof}
We only consider the upper barriers, the other case being nearly
identical.  Proposition~\ref{lem-intermed-barriers} asserts that for
\(\delta < \frac 12 K_1\), the function \(w^+_{\rm md}(z,\tau)\) is an upper barrier in
the intermediate region \(R_* \leq z \leq e^{-(k/3 )\tau} \) and it
satisfies \eqref{eqn-wa14} with this choice of constants, that is
\[
  \lim_{\tau\to-\infty} w^+_{\rm md}(z,\tau)
  = z + \frac{K_1+\delta}{z^2} + \frac{1}{z^p} +\cO(z^{-5})
  \quad (z\to\infty)
\]
where the \(\cO(z^{-5})\) term is uniform in \(\delta\in(0, \frac12 K_1)\).
We have also seen that
\[
  \lim_{\tau\to-\infty} w^+_\delta(z,\tau)
  = z + \frac{K_1+2\delta}{z^2} +  \cO(z^{-3}) \quad(z\to\infty)
\]
where \(\cO(z^{-3})\) is again uniform in \(\delta\).
Therefore
\[
  \lim_{\tau\to-\infty} w_{\delta}^+(z,\tau) - w^+_{\rm md}(z,\tau)
  = \frac{\delta}{z^2} - \frac{1}{z^p} + O(z^{-3}) \quad(z\to\infty).
\]
Consider \(Z_\delta:= \frac 43 \delta^{-\frac{1}{p-2}}\).  For small enough
\(\delta>0\) one has \(Z_\delta \geq 2R_*\), so that \(w^\pm_\delta(z,\tau)\) and
\(w_{\rm md}^\pm (z, \tau)\) are defined for all \(z\geq \frac 12 Z_\delta\) and all
\(\tau\leq \tau_*\).  We evaluate these differences at \(z=Z_\delta\) and
\(z=\frac 12 Z_\delta\).  Eliminating \(\delta\) by using
\(\delta=(\frac 34 Z_\delta)^{-(p-2)}\) we find
\begin{equation*}
  \lim_{\tau\to-\infty} w_\delta^+(Z_\delta,\tau) - w^+_{\rm md}(Z_\delta,\tau)
  = \left(\bigl(\tfrac 43\bigr)^{p-2}-1\right) Z_\delta^{-p} +  \cO(Z_\delta^{-3}).
\end{equation*}
For small enough \(\delta>0\), \(Z_\delta\) is large, and thus the first term dominates
the second.  This implies that for small \(\delta>0\) there is a \(\tau_\delta<0\) such
that
\[
  w_\delta^+(Z_\delta,\tau) - w^+_{\rm md}(Z_\delta,\tau) > 0
\]
for all \(\tau\leq \tau_\delta\).  Similarly, we have
\[
  \lim_{\tau\to-\infty} w_\delta^+(Z_\delta/2,\tau) - w^+_{\rm md}(Z_\delta/2,\tau) = 
  \left(\left(\tfrac 23\right)^{p-2}-1\right)2^pZ_\delta^{-p}  +  \cO(Z_\delta^{-3}).
\] 
This implies that if \(\delta>0\) is small then there is a \(\tau_\delta<0\) such that
\[
  w_\delta^+(\tfrac 12 Z_\delta,\tau) - w^+_{\rm md}(\tfrac 12 Z_\delta,\tau) < 0
\]
for all \(\tau\leq \tau_\delta\).  
\end{proof}

\subsection{A summary of our construction so far}
\label{sec-choice-constants}

The initial data \(u_0\) determines two constants \(k\geq 4\) and \(K_0\).  Throughout the
paper we let \(K_1 = (2k+1)!!K_0\) and \(K_2=K_1^{1/3}\).

In section \S\ref{subsec-barriers-outer} we chose a constant \(M>0\) so that
Lemma~\ref{lemma-outer} holds and constructed upper and lower barriers \(u^\pm(x, t)\)
in the outer region \(\cO_M\).

For any small enough \(\delta>0\) we then constructed a family of barriers
\(v_\delta^\pm\) in the intermediate region defined by
\(R_* e^{\gamma\tau}\leq y \leq e^{-\tau/2}\), \(\tau\leq \tau_\delta\).  Here
Propositions~\ref{lem-intermed-barriers} and \ref{lem-intermed-barrier-ordering}
specify \(R_*\), while \(\tau_\delta\) is determined when we match the intermediate
and inner barriers in Lemma~\ref{lemma-match-outer-inter}.

For small \(\delta>0\) we then considered the inner region
\(\cI_{Z_\delta,\tau_\delta} = \{(z, \tau) \mid 0\leq z\leq Z_\delta, \tau\leq
\tau_\delta\}\) with \(Z_\delta:= \frac 43 \, \delta^{-\frac{1}{p-2}}\) and where
\(\tau_\delta\) is as above.
Since \(\delta>0\) is small and \(R_*\) does not depend on \(\delta\), we have
\(\delta < \big ( \frac 32 R_* \big )^{2-p}\), which implies \(Z_\delta > 2R_*\).  Hence
the intermediate and inner regions overlap at least on \(\frac12 Z_\delta\leq z\leq  Z_\delta\).

Lemma \ref{lemma-inner-barrier1} with \(K_2^+\) satisfying \((K_2^+)^3=K_1 + 2\delta\)
defines the upper barrier \(w_\delta^+\) in the inner region
\(\cI_{Z_\delta,\tau_\delta}\) and Lemma \ref{lemma-inner-barrier2} with \(K_2^-\)
satisfying \((K_2^-)^3=K_1- 2\delta\), defines the constant \(D=D(K_2^-)\) and the lower
barrier \(w_\delta^-\) in \(\cI_{Z_\delta,\tau_\delta}\).

\subsection{The upper and lower barriers \(U_\delta^+(x,t)\), \(U_\delta^-(x,t)\)}

In the previous subsections, we constructed upper barriers
\(u^+(x,t), v^+_\delta(y,\tau), w^+_\delta(z,\tau)\) and lower barriers
\(u^-(x,t), v^-_\delta(y,\tau), w^-_\delta(z,\tau)\) in the outer, intermediate, and
inner regions respectively, and showed that they are correctly ordered in the
overlaps between the three regions.  These barriers exist for all
\(0 < t \leq t_\delta\) or equivalently \(- \infty < \tau \leq \tau_\delta\).  Therefore,
the barrier \(U^+_\delta(x,t)\) constructed by taking the minimum of the upper barriers
when all are expressed in the un-rescaled \((x,t)\) variables, that is
\begin{equation}\label{eqn-ubu} 
  U_\delta^+(x,t) = \min \left \{ \,
  u^+(x,t),\,
  t^{1/2} v^+_\delta\Bigl( \frac{x}{t^{1/2}}, \log t\Bigr),\,
  t^{k/3} w^+_\delta\Bigl(\frac{x}{t^{k/3}} ,\log t\Bigr) \, \right\}
\end{equation}
is a weak supersolution of equation \eqref{eq-u-pde} and similarly the barrier
\(U^-_\delta(x,t)\) constructed by taking the maximum of the lower barriers when all
are expressed in the un-rescaled \((x,t)\) variables, that is
\begin{equation}\label{eqn-lbu}
  U_\delta^-(x,t) = \max \left \{ \,
  u^-(x,t),\,
  t^{1/2} v^-_\delta\Bigl( \frac{x}{t^{1/2}}, \log t\Bigr),\,
  t^{k/3} w^-_\delta\Bigl(\frac{x}{t^{k/3}} ,\log t\Bigr) \, \right\}
\end{equation} 
is a weak sub-solution of equation \eqref{eq-u-pde}.  This is summarized in the
following proposition.

\begin{proposition}\label{prop-barriers}
  There exist a number \(\delta_0>0\) and a sequence of times \(t_n\searrow 0\) such
  that the functions \(U_{\delta_n}^\pm(x,t)\) given in \eqref{eqn-ubu},
  \eqref{eqn-lbu} with \(\delta_n=2^{-n}\delta_0\), define weak super- and
  sub-solutions of equation \eqref{eq-u-pde}, for all \(0 < t \leq t_n\).

  Moreover, one has
  \begin{equation}\label{eqn-ordering} 
    U_{\delta_{n}}^-(x, t) \leq  U_{\delta_{n+1}}^-(x, t) <
    U_{\delta_{n+1}}^+(x, t) \leq  U_{\delta_{n}}^+(x, t) 
  \end{equation}
  for all \(x>0\) and \( 0 < t \leq t_{n+1}\).
\end{proposition}

\begin{proof} The fact that \(U_{\delta_n}^\pm(x,t)\), \(0 < t \leq t_n\) define weak
super- and sub-solutions of equation \eqref{eq-u-pde} follows from Lemma
\ref{lemma-outer}, Proposition \ref{lem-intermed-barriers}, Lemmas
\ref{lemma-inner-barrier1} -- \ref{lemma-inner-barrier2} and the matching of our
barriers in subsections \ref{subsec-outer-inter} and \ref{subsec-inner-inter}.

For \eqref{eqn-ordering}, we recall that our barriers \(u^\pm(x,t)\) in the outer
region do not depend on \(\delta\), hence they are ordered in their common domain and
furthermore it is clear that \(u^-(x,t) < u^+(x,t)\).  In
Proposition~\ref{lem-intermed-barriers} we proved
\eqref{eq-intermed-barrier-ordering}, which implies that \eqref{eqn-ordering} holds
in the intermediate region for \(0 < t \leq t_{n+1}\).  To finish the proof of
\eqref{eqn-ordering} it is sufficient to show that for any \(\delta \leq \delta_0\)
the inequalities
\begin{equation}\label{eqn-w-ordering} w^{-}_{\delta}(z,\tau) <
  w^{-}_{\delta/2}(z,\tau) < w^{+}_{\delta/2}(z,\tau) <
  w^{+}_{\delta}(z,\tau)
\end{equation}
hold for all \(0 \leq z \leq Z_\delta\), \(\tau \leq \tau_{\delta}\).  This follows
from the definition of \(w^\pm_{\delta}(z,\tau)\) in subsection
\ref{subsec-inner-inter} by observing that the rescaled Alencar solutions
\(W_K(z):= K \, W \big ( \frac zK \big )\), are ordered for \(K>0\), that is
\begin{equation}\label{eqn-Worder}
  \kappa < \bar \kappa  \implies W_\kappa(z) < W_{\bar \kappa}(z),
  \quad \mbox{for all}\,\, z \in [0, +\infty).
\end{equation}
To see this, recall the inequality \(W-zW_z>0\), \(z \geq 0\) which is a
consequence of the convexity of \(W\) and was shown in Lemma~\ref{lem-WK-starshaped}.
This inequality implies that
\begin{equation}
  \label{eqn-WKd} \frac{d}{d\kappa} W_\kappa(z)
  = \frac{d}{d\kappa} \big ( \kappa \, W \big ( \frac z\kappa \big )  \big )
  =  W \big ( \frac z\kappa \big )
  - \frac z\kappa \,  W'  \big ( \frac z\kappa \big )
  >0
\end{equation}
i.e.~\(\kappa \to W_\kappa(z)\) is monotone increasing in \(\kappa\).  We conclude that
\eqref{eqn-w-ordering} holds which finishes the proof of \eqref{eqn-ordering} and
the proof of the proposition.
\end{proof}

%%%%%%%%%%%%%%%%%%%%%%%%%%%%%%%%%%%%%%%%%%%%%%%%%%%%%%%%%%

\section{Existence of a smooth solution}
\label{sec-existence}

\subsection{Outline of the existence proof}
In this section we return to the \(O(4)\times O(4)\) symmetric hypersurface \(M_0\) with
profile function \(u_0:[0,\infty)\to\R\).  Recall that \(u_0\) is smooth for \(x >0\) and
satisfies conditions \eqref{eq-init-data1} and \eqref{eq-init-data2} for some fixed
\(k > 3\) and some constant \(C_0 >0\).  In Proposition \ref{prop-barriers} we
constructed sequences of nested upper and lower barriers for \eqref{eq-u-pde}.  We
will show in this section how to use them to prove the existence of a smooth solution
\(u(x,t)\) to the initial value problem \eqref{eq-u-pde}--\eqref{eq-u-initcond} defined
for all \(0 < t \leq t_0\), for some \(t_0 >0\).  Our main result in this section is as
follows.

\begin{theorem}[Existence of a smooth solution]
\label{thm-existence}
Assume that \(M_0\) is an \(O(4)\times O(4)\) symmetric hypersurface defined by a
profile function \(u_0:[0,\infty)\to\R\) which is smooth for \(x >0\) and satisfies
conditions \eqref{eq-init-data1}--\eqref{eq-init-data2}.  Then there exists \(t_0>0\)
and a \(C^\infty\)-smooth \(O(4) \times O(4)\) symmetric MCF solution \(M_t\),
\(0 < t \leq t_0\) defined by a profile function
\(u:(0,\infty)\times(0, t_0]\to(0,\infty)\) which satisfies the initial value
problem \eqref{eq-u-pde}--\eqref{eq-u-initcond}.  Furthermore, \(u(x,t)\) satisfies
\begin{equation}\label{eqn-between}
  U_{\delta_n}^-(x,t) \leq u(x,t) \leq U_{\delta_n}^+(x,t), 
  \qquad (x,t) \in [0,\infty) \times (0,t_n) 
\end{equation} 
where $\delta_n = 2^{-n} \, \delta_0$ and \(U_{\delta_n}^\pm(x,t)\), for \(t\in (0,t_n)\) are the upper and lower barriers
constructed in Proposition \ref{prop-barriers}.

It follows from \eqref{eqn-between} that
\begin{equation}
  \label{eqn-inner-description}
  \lim_{t\searrow 0} t^{-k/3} \, u\bigl(t^{k/3 } z, t\bigr) = W_{K_2}(z)
\end{equation}
uniformly for bounded \(z\geq 0\).
\end{theorem}

Since the equation \eqref{eq-u-pde} is singular at \(u=0\), we
cannot directly apply one of the standard short time existence results to obtain
our solution \(u(x, t)\).  Instead, we will construct it as the limit of a
sequence of approximating solutions \(u_n(x,t)\), each of which is defined on some
time interval starting at a carefully chosen initial time \(s_n\), where \(s_n
\searrow 0\).  
We will define the approximating solutions \(u_n\) by choosing their initial times
\(s_n\) and values \(u_n(x, s_n)\) in such a way that they satisfy
\begin{equation}
  \label{eqn-un-barriers-init}
  U^-_{\delta_n}(x, s_n) \leq u_n(x, s_n)
  \leq U^+_{\delta_n}(x, s_n) \qquad \text{for all }x\geq 0,
\end{equation}
where \(\delta_n := 2^{-n} \delta_0\) and where \(U^\pm_{\delta_n}(\cdot,t)\) are
the barriers constructed in Proposition~\ref{prop-barriers}.  

The barrier \(U^-_{\delta_n}\) is bounded away from \(u=0\), and this allows us
to invoke a classical short time existence theorem for the quasilinear
parabolic initial value problem~\eqref{eq-u-pde}--\eqref{eq-u-bcond}.  The
short-time existence theorem guarantees that our solution exists for \(s_n\leq
t < \bar t_n\), i.e.~until some time \(\bar t_n>s_n\).  This time may exceed
the life time \(t_n\) of the barriers \(U_{\delta_n}^\pm\).  In fact, by
finding \emph{a priori} estimates for the solutions \(u_n(x, t)\) we will show
that there is an \(n_0\) such that for all \(n\geq n_0\) we have  \(\bar t_n > t_{n_0}\), and that we can extract a convergent subsequence
\(u_{n_j}(x,t)\) whose limit \(u(x, t)\) is a solution of the full initial
value problem \eqref{eq-u-pde}--\eqref{eq-u-initcond}, and which is defined for
\(x\geq 0\) and \(0\leq t\leq t_{n_0}\).

The first \textit{a priori} estimate we derive for the \(u_n\) follows directly
from the maximum principle applied to the barriers \(U_{\delta_n}^\pm\).  Since
the barriers are ordered by \eqref{eqn-ordering}, the \textit{a priori} bound
\eqref{eqn-un-barriers-init} implies that for all \(n_0\), \(n\geq n_0\) and
\(x\geq 0\) one has
\begin{equation}
\label{eq-in-between}
  U^-_{\delta_{n_0}}(x, s_n) 
  \leq U^-_{\delta_n}(x, s_n) 
  \leq u_n(x, s_n)
  \leq U^+_{\delta_n}(x, s_n) 
  \leq U^+_{\delta_{n_0}}(x, s_n). 
\end{equation}
The maximum principle tells us that for all \(n\geq n_0\) and \(x\geq 0\) one has
\begin{equation}
  \label{eqn-UUU}
  U^-_{\delta_{n_0}}(x, t) \leq u_n(x, t) \leq U^+_{\delta_{n_0}}(x, t)
\end{equation}
for all \(t\geq s_n\) at which \(U_{\delta_{n_0}}^\pm (x, t )\) and \(u_n(x, t)\) are
defined, i.e.~for  \(s_n\leq t < \min\{\bar t_n, t_{n_0}\}\).

Thereafter we establish \textit{a priori} estimates for the higher order derivatives
of the \(u_n\).  We conclude this work in the next  section~\ref{sec-Hbounded} by showing that the mean curvatures \(H_n(x, t)\) of the
evolving surfaces corresponding to the approximating solutions \(u_n(x, t)\) are
uniformly bounded for all \(x, n, t\), and hence that the mean curvature of the limit
solution \(u(x, t)\) also is uniformly bounded.

\smallskip

The simplest choice for the initial value for \(u_n\) would be to simply set
\(u_n(x, s_n) = U_{\delta_n}^-(x, s_n)\), but this function is not necessarily smooth
in the overlaps between inner, intermediate, and outer regions, and this complicates
the estimation of the higher derivatives of \(u_n\).  Furthermore, to  prove that the mean curvatures \(H_n(x, t)\) are
uniformly bounded, it will be important to have \(H_n(x,s_n)=0\) on
\(0 \leq x \leq \epsilon s_n^{1/2}\) for some small fixed \(\epsilon >0\).  For these
reasons we will construct \(u_n(x, s_n)\) by smoothly gluing the lower barrier
\(U_{\delta_n}^-(x, s_n)\) to an Alencar surface in the inner region
\(x\leq \epsilon s_n^{1/2}\).  Let us now turn to the details of this construction.

\subsection{Short time existence and the comparison principle}
\label{sec-short-time-existence} Equation~\eqref{eq-u-pde} for \(u(x, t)\) has
a singular term at \(x=0\) which is there because we consider radially
symmetric solutions only.  To derive short time existence from existing
results, it is more convenient to consider the more general case of
hypersurfaces that are only partially symmetric, i.e.~with \(\{1\}\times O(4)\)
rather than \(O(4)\times O(4)\) symmetry.  For any positive function
\(r:\R^4\times[0, t_0)\to\R\) we consider the family of hypersurfaces
parameterized by \(F:\R^4\times S^3\times[0, t_0)\to \R^8\) where
\[
  F(x, \Omega, t) = (x, r(x, t)\Omega).
\]
A direct computation shows that \(F\) evolves by MCF if and only if \(r\)
satisfies \begin{equation} \label{eq-r-pde} r_t = g^{ij}(Dr)r_{x_ix_j} -
  \frac3r,
\end{equation}
in which
\[
  g_{ij}(p) = \delta_{ij}+p_ip_j, \qquad
  g^{ij}(p) = \delta_{ij}-\frac{p_ip_j}{1+|p|^2}.
\]
As long as \(Dr\) is uniformly bounded,~\eqref{eq-r-pde} is a uniformly
parabolic quasilinear equation.  The solutions that interest us are not
bounded, so we choose a reference function \(R:\R^4\to \R\) that is uniformly
bounded from below, has uniformly bounded derivatives up to third order, and
for which \(R(x) - u_0(\|x\|)\) is uniformly bounded.

All initial data we  prescribe in the following sections are bounded
perturbations of \(R(x)\).  We therefore consider solutions of the form \(r(x,
t) = R(x) + a(x, t)\), and derive the equation for \(a\):
\begin{equation}
  \label{eq-a-pde}
  a_t = g^{ij}(DR+Da)a_{x_ix_j} +g^{ij}(DR+Da)R_{x_ix_j}- \frac{3}{R+a}
\end{equation}
Since we assume that \(DR\) and \(D^2R\) are uniformly bounded, this equation is
uniformly parabolic, as long as \(Da\) is bounded.  By assumption \(D^m R\) with
\(m\leq 3\) are all uniformly bounded, so \eqref{eq-a-pde} is of the form
\[
  a_t = A_{ij}(x, Da)a_{x_ix_j} + B(x, a, Da)
\]
where \(A_{ij}\) are uniformly parabolic, and where the functions \(A_{ij}\),
\(B\) are \(C^1\) in \(x\in\R^4\) and real analytic in \((a, Da)\).

This implies the existence of a short time solution \(a(x, t)\) for any initial
\(a(x, 0)\) with \(a(\cdot, 0) \in C^{1,\alpha}(\R^4)\), and for which
\(\inf_x R(x)+a(x, 0) > 0\).  The classical theory for quasilinear parabolic
equations \cite[\S VI.1]{LSU}
% \marginpar{\footnotesize\color{badgerred} Quoting
%   Ladyzhenskaya et al \cite{LSU} here --- that book is really hard to read, and I'm
%   not sure \S VI.1 really says what we want it to say}
implies that as long as
\(\sup_x|a(x,t)|\) and \(\sup_x|Da(x, t)|\) are bounded, and as long as
\(\inf_x R(x)+a(x, t)\) has a positive lower bound, one can show that
\(Da(\cdot, t)\) is uniformly Hölder continuous.  This in turn implies higher
derivative bounds, and hence that the solution can be extended to a larger time
interval.

For such solutions the standard comparison principle also
holds: if \(a_\pm : \R^4\times[0, t_0)\to\R\) are two solutions with \(Da_\pm\)
bounded, for which \(a_-(x, 0) \leq a_+(x, 0)\) holds for all \(x\in\R^4\), then
\(a_-(x, t) \leq a_+(x, t)\) for all \(x\in\R^4\) and \(t<t_0\).

\subsection{The approximating sequence of solutions \(u_n\) with \(n\geq n_0\)}
For a fixed small \(\epsilon >0\) (independent of \(n\)) we choose functions
\(\Psi\), \(\psi_n\) with
\[
  \psi_n(x)=\Psi\Bigl(\frac{x}{\epsilon\sqrt{s_n}}\Bigr), \qquad\Psi\in
  C^\infty(\R),\qquad \Psi(\xi) =
  \begin{cases}
    1 & 0\leq \xi\leq 1, \\ 0 & \xi\geq 2.
  \end{cases}
\]
We define
\begin{equation}
  \label{eq-u0n-defined}
  u_{0n}(x) := \psi_n(x)\,s_n^{k/3}\, W_{K_2}\left({x s_n^{-k/3}}\right) +
  (1-\psi_n(x)) \, U_{\delta_n}^-(x,s_n)
\end{equation}
and let \(u_n: (0,\infty) \times [s_n, \bar t_n) \to (0, \infty)\) be the solution
to the initial value problem \eqref{eq-u-pde}-\eqref{eq-u-initcond} with initial data
\(u_n(\cdot,s_n) = u_{0n}(x)\) instead of \(u_0(x)\).

We will only consider the initial data for sufficiently large \(n\), i.e.~we choose
an \(n_0\in\N\), and only consider those solutions \(u_n\) with \(n\geq n_0\).
Throughout this section ``for all \(n\)'' will mean ``for all \(n\geq n_0\),'' and in
each Lemma we assume that \(n_0\) has been chosen large enough for the statement to
hold.

In Corollary \ref{cor-between} we verify that our chosen initial data are caught
between the barriers, as in~\eqref{eqn-between}.
Before doing that we establish some derivative bounds for \(u_{0n}(x)\).

\begin{lemma}[Monotonicity and derivative bounds]
\label{lemma-uniform-C1}
For large enough \(n_0\) and any \(n\geq n_0\) there is an \(s_n\in(0, t_n)\) such
that the sequence \(\{s_n:n\geq n_0\}\) is decreasing, and such that
\(u_n(x, s_n)\) satisfies the following estimates for all \(n\):

\smallskip\noindent{\bfseries\upshape(i)} The function \(x\mapsto u_n(x, s_n)\) is
locally Lipschitz and
\begin{equation}
  \label{eq-un-1st-deriv-bound}
  0\leq (u_n)_x (x,s_n) \leq C_1
\end{equation}
for almost all \(x > 0\), for some \(C_1>0\)

\smallskip\noindent{\bfseries\upshape(ii)} The function \(x\mapsto u_n(x, s_n)\) is
\(C^3\) on the interval \(0\leq x\leq Ms_n^{1/2}\), where for \(j=2,3\), and all
\(n\), one has
\begin{equation}
  \label{eq-un-higher-deriv-bound}
  \big(1+s_n^{- k/3} x \big)^{j+2} \, |\partial^j_x u_n(x,s_n)|
  \leq C\, s_n^{ -(j-1) k/3} .
\end{equation}
\end{lemma}

We present the proof in the following
subsections~\ref{sssec-1st-deriv}--\ref{sssec-un-monotone}.  Along the way we finally
choose the initial times \(s_n \searrow 0\), and we use generic constants \(C\) that only
depend on the various parameters defining the barriers, and the fixed small parameter
\(\epsilon\), but not on \(n\).

\subsection{Proof of the first derivative bound \eqref{eq-un-1st-deriv-bound}}
\label{sssec-1st-deriv}
We have
\begin{equation}
  \begin{split}
    \label{eq-un-der}
    (u_n)_x (x,s_n)= &\,\, \psi_n' \, s_n^{k/3} W_{K_2}(x s_n^{k/3}) - \psi_n'
    U_{\delta_n}^- \\
    &\qquad + \psi_n\, W_{K_2}'(x s_n^{-k/3}) + (1-\psi_n)\, (U_{\delta_n}^-)'.
  \end{split}
\end{equation}
We estimate these terms one by one.

The terms in \eqref{eq-un-der} involving \(\psi_n'\) vanish outside the interval
\(\epsilon s_n^{1/2}\leq x\leq 2\epsilon s_n^{1/2}\).  Thus we have
\begin{align*}
  |\psi_n'(x) s_n^{k/3} W_{K_2}(x s_n^{-k/3})|
  &\leq \max_{x\geq 0} |\psi_n'(x)| \cdot
  \max_{x\leq 2\epsilon s_n^{1/2}}\big|s_n^{k/3}  W_{K_2}(x s_n^{-k/3}) \big| \\
  &\leq Cs_n^{-1/2}  \cdot Cs_n^{1/2} \leq C,
\end{align*}
where we have estimated \(W_{K_2}(z)\leq C(1+z)\) for all \(z\geq 0\).

To estimate the other term involving \(\psi_n'(x)\) we recall that \(U_\delta^-\) is
defined in \eqref{eqn-lbu} as the minimum of \(w_\delta^-\), \(v_\delta^-\), and \(u^-\),
appropriately rescaled, and that, according to Lemmas~\eqref{lemma-match-outer-inter}
and~\eqref{lemma-match-middle-inner}, in the region \(z\geq Z_{\delta_n}\),
\(y\leq \frac14 Y_{\delta_n}\) the function \(v_{\delta_n}^-\) is the largest of these.
If we choose \(s_n>0\) so small that
\(\epsilon s_n^{-\gamma} > Z_{\delta_n} = \frac43 \delta_n^{\frac{-1}{p-2}}\) then in
the region \(\epsilon s_n^{1/2}\leq x\leq 2\epsilon s_n^{1/2}\) we have
\[
  U_{\delta_n}^-(x, s_n) = s_n^{1/2}v_{\delta_n}(xs_n^{-1/2}, \log s_n),
\]
and thus also
\[
  |\psi_n'| |U_{\delta_n}^-| \le Cs_n^{-1/2} \left|s_n^{1/2}
  v_{\delta_n}(xs_n^{-1/2}, \log s_n)\right| \le v_{\delta_n}(y, \log
  s_n)
\]
where \(y=x/\sqrt{s_n}\) lies in the interval \([\epsilon,2\epsilon]\).  This implies
that \(\psi_n'(x)U^-_{\delta_n}(x, s_n)\) is uniformly bounded.

To estimate the third term we recall that \(0\leq W_{K_2}'(z) \leq 1\), which implies
\[
  |\psi_n(x) W_{K_2}'\bigl(x s_n^{-k/3}\bigr)| \leq \psi_n(x) \leq 1.
\]

Finally, the term \((1-\psi_n)\bigl(U_{\delta_n}^-\bigr)'\) vanishes for
\(x\leq \epsilon\sqrt{s_n}\).  For \(x\geq \epsilon\sqrt{s_n}\) we have
\[
  U_{\delta_n}^-(x, s_n) =
  \begin{cases}
    \sqrt{s_n}v(\frac x{\sqrt{s_n}}, \log s_n)&
    x\leq \frac14 Y_{\delta_n}\sqrt{s_n}\\
    \max \bigl\{\sqrt{s_n}v_{\delta_n}^-(\frac x{\sqrt{s_n}}, \log s_n),\; u^-(x,
    s_n)\bigr\} &
    \frac14 Y_{\delta_n}\sqrt{s_n} \leq x\leq Y_{\delta_n}\sqrt{s_n}      \\
    u^-(x, s_n)& x\geq Y_{\delta_n} \sqrt{s_n}
  \end{cases}
\]
with \(Y_{\delta_n}= 2\sqrt{(2k+1)!!M/\delta_n}\) as in
Lemma~\ref{lemma-match-outer-inter}.

It follows that \(x\mapsto U^-_{\delta_n}(x, s_n)\) is a Lipschitz continuous function
whose derivative is almost everywhere given by \(\bigl(v_{\delta_n}^-\bigr)_y\) or
\(u_x^-(x, s_n)\).  If \(y=\frac x{\sqrt{s_n}}\in [\epsilon, Y_{\delta_n}]\) then
\[
  (v_{\delta_n})_y\left(y, \log s_n\right) = 1 + (K_1^--\delta_n)\, s_n^{3\gamma}
  \varphi_k'(y) - BK_1^2 s_n^{6\gamma} g'(y) + p \frac{s_n^{(p+1)\gamma}}{ y^{p+1}}
  \le C,
\]
for a uniform constant \(C\), independent of \(n\) and for \(n \ge n_0\), sufficiently big.

On the other hand, \(u^-(x,s_n) = u_0(x) - M s_n\min\{1,x^{2k-4}\}\).  For
\(x\geq 1\) we have \(u^-_x(x, s_n) = u_0'(x)\), which is uniformly bounded by the
assumption~\eqref{eq-init-data2}, while for \(x<1\) we have
\(u^-_x(x, s_n) = u_0'(x)-(2k-4)Ms_n x^{2k-5}\), which is also uniformly bounded
because we assume \(k\geq 4\).

Combining all these estimates together with \eqref{eq-un-der} yields the uniform
Lipschitz bound on \(u_n\).

\subsection{Proof of the second derivative estimate \eqref{eq-un-higher-deriv-bound}}
\label{sssec-2nd-deriv}
We will show
\begin{equation}
  \label{eq-to-show-in-part-ii}
  |(u_n)_{xx} (x,s_n)| \leq Cs_n^{-k/3} \bigl(1+ x s_n^{-k/3}\bigr)^{-4}.
\end{equation}
for all \(x \in [0, M\sqrt{s_n}]\).

Writing \(z=x s_n^{-k/3}\), we estimate the terms on the right hand side of
\begin{equation}
  \label{eq-2-der}
  \begin{split}
    (u_n)_{xx} = \psi_n'' \, s_n^{k/3}
    &W_{K_2}(z) + 2\psi_n' \, W_{K_2}'(z) + \psi_n\, W_{K_2}''(z) \, s_n^{-k/3} \\
    &+ (1 - \psi_n) (U_{\delta_n}^-)_{xx} - 2\psi_n' \, ( U_{\delta_n}^-)_x -
    \psi_n'' \, U_{\delta_n}^-\, .
  \end{split}
\end{equation}

For \(0 \le x \le \epsilon{s_n^{1/2}}\) we have
\[
  (u_n)_{xx}(x,s_n) = s_n^{-k/3} W_{K_2}''(z).
\]
The asymptotic expansion~\eqref{eqn-Was} for \(W\) implies that for all \(z\geq 0\)
\[
  0\leq W_{K_2}''(z)\leq C(1+z)^{-4}
\]
Hence~\eqref{eq-to-show-in-part-ii} holds for \(x\leq \epsilon{s_n^{1/2}}\).

If \(2\epsilon s_n^{1/2} \le x \le M s_n^{1/2}\), i.e.~if \(2\epsilon \le y \le M\), then
\(u_n(x, s_n) = s_n^{-1/2} v _{\delta_n}^-(y, \log s_n)\) and thus, using the
definition~\eqref{eq-barriers-med-v} for \(v_{\delta_n}^-\), we find for
\(2\epsilon \le y \le M\),
\[
  (u_n)_{xx}(x,s_n) = s_n^{-1/2} (v_{\delta_n}^-)_{yy}(y,\log s_n) \le Cs_n^{-1/2}\,
  \frac{s_n^{3\gamma}}{y^4} \le \frac{C s_n^{-k/3}}{(1+xs_n^{-k/3})^4}\;.
\]

Finally, if \(\epsilon s_n^{1/2} \le x \le 2\epsilon s_n^{1/2}\), then similarly to the
previous two cases we get
\[
  |\psi_n W_{K_2}''(z) \, s_n^{-k/3} + (1 - \psi_n) (U_{\delta_n}^-)_{xx}| \le C
  s_n^{-k/3}\Bigl(1+xs_n^{-k/3}\Bigr)^{-4}.
\]
To bound the remaining terms in \eqref{eq-2-der} it is enough to estimate
\[
  2|\psi_n'||W_{K_2}'(z) - (U_{\delta_n})_x| + |\psi_n''||s_n^{k/3} W_{K_2}(z) -
  U_{\delta_n}^-|.
\]
Both \(\psi_n'\) and \(\psi_n''\) vanish unless
\(\epsilon s_n^{1/2}\leq x\leq 2\epsilon s_n^{1/2}\).  In this region one has
\(xs_n^{-k/3}\geq 1\), and thus our desired upper bound satisfies
\[
  \frac 1C s_n^{k-2}\leq s_n^{-k/3}\bigl(1+xs_n^{-k/3}\bigr)^{-4}\leq C s_n^{k-2}.
\]

By the asymptotic expansion \eqref{eqn-Was} of the Alencar solution \(W\) for large
\(z\), we have \(W_{K_2}(z) = z + \cO(z^{-2})\) and \(W_{K_2}'(z) = 1+ \cO(z^{-3})\).
When \(\epsilon s_n^{1/2}\leq x\leq 2\epsilon s_n^{1/2}\) this implies
\begin{equation}
  \label{eq-W-in-glue-region}
  \begin{gathered}
    s_n^{k/3}W_{K_2}\bigl(xs_n^{-k/3}\bigr)-x = \cO\bigl(s_n^kx^{-2}\bigr) = \cO\bigl(s_n^{k-1}\bigr),\\
    W_{K_2}'\bigl(xs_n^{-k/3}\bigr)-1 = \cO\bigl(s_n^kx^{-3}\bigr) = \cO\bigl(s_n^{k-
    3/2}\bigr).
  \end{gathered}
\end{equation}
In the region \(\epsilon s_n^{1/2}\leq x\leq 2\epsilon s_n^{1/2}\) we have, by
definition, and by the asymptotic expansions of the terms \(f_0^-, f_1, f_2\) in
\eqref{eq-fj-defined},
\begin{align}
  \label{eq-Udelta-in-glue-region}
  U_{\delta_n}^-(x, s_n)
  &= s_n^{1/2} v_{\delta_n}^- (y, \log s_n ) & (\text{where }y=xs_n^{-1/2})    \\
  &= s_n^{1/2}y + s_n^{1/2}\cO\bigl(s_n^{k-3/2}y^{-2}\bigr) \notag \\
  &= x + \cO\bigl(s_n^{k}x^{-2}\bigr).\notag
\end{align}
This expansion may be differentiated with respect to \(x\), resulting in
\begin{equation}
  \label{eq-Udelta-deriv-in-glue-region}
  \left|\bigl(U_{\delta_n}^-\bigr)_x - 1\right| \leq Cs_n^kx^{-3}\leq Cs_n^{k-3/2}.
\end{equation}

The bounds \(|\psi_n'|=\cO(s_n^{-1/2})\) and \(|\psi_n''| = \cO(s_n^{-1})\) now lead
to
\[
  |\psi_n''||s_n^{k/3} W_{K_2}(xs_n^{-k/3}) - U_{\delta_n}^-| \leq Cs_n^{-1}
  s_n^{k-1} = Cs_n^{k-2} \leq \frac{C s_n^{-k/3}}{(1 + x s_n^{-k/3})^4}.
\]
and also
\[
  |\psi_n'||W_{K_2}'(xs_n^{-k/3}) - (U_{\delta_n}^-)_x| \leq Cs_n^{-1/2}s_n^{k-3/2} \le
  \frac{\bar{C} s_n^{-k/3}}{(1 + x s_n^{-k/3})^4}.
\]
This concludes the proof of stated weighted \(C^2\) estimate for \(u_n\) at time
\(t = s_n\).

\subsection{Proof of the third order derivative bound \eqref{eq-un-higher-deriv-bound}}
\label{sssec-3rd-deriv}
We outline the arguments, which are similar to those for the second derivative
estimate.

For \(0 \le x \le \epsilon s_n^{1/2}\) the definition \eqref{eq-u0n-defined} of
\(u_{0n}(x)=u_n(x, s_n)\) directly implies
\[
  |(u_n)_{xxx}(x,s_n)| = |W_{K_2}'''(z)| s_n^{-2k/3}, \text{ where again
  }z=xs_n^{-k/3}.
\]
Using the asymptotic expansion for \(W(z)\) as \(z\to\infty\) one then verifies the
third derivative estimate for \(x\leq \epsilon s_n^{1/2}\).

If \(2 \epsilon s_n^{1/2} \le x \le M s_n^{1/2}\), i.e.~if \(2\epsilon \le y \le M\),
then
\[
  (u_n)_{xxx}(x,s_n) = (U_{\delta_n}^-)_{xxx}(x,s_n) =
  s_n^{-1}(v_{\delta_n}^-)_{yyy}(y,\log s_n),
\]
and the estimate follows from the explicit expression~\eqref{eq-barriers-med-v} for
\(v_{\delta_n}^-(y,\tau)\).

If \(\epsilon s_n^{1/2} \le x \le 2\epsilon s_n^{1/2}\), then \(u_n\) is given by
\[
  u_n(x, s_n) = s_n^{k/3}W_{K_2}(z) + \psi_n(x) \bigl\{s_n^{k/3}W_{K_2}(z) -
  U_{\delta_n}^-(x, s_n)\bigr\} \qquad (z=xs_n^{-k/3}).
\]
The third derivative of the first term can be estimated exactly as in the region
\(x\leq \epsilon s_n^{1/2}\).  After differentiating the second term three times one
ends up with terms of the form
\[
  \psi_n^{(3-\ell)}(x) \left(\frac{\pd}{\pd x}\right)^\ell \Bigl\{s_n^{k/3}W_{K_2}(z)
  - U_{\delta_n}^-(x, s_n)\Bigr\} \qquad (0\leq \ell\leq 3).
\]
Using the asymptotic descriptions we have for \(W\) and \(U_{\delta_n}^-\), and
taking care to cancel the leading terms in these descriptions when
\(\ell\in\{0,1\}\), we get the third derivative bounds in
\eqref{eq-un-higher-deriv-bound}.
The  estimates are similar to the first and second order estimates.

\subsection{Proof that \(x\mapsto u_n(x, s_n)\) is non-decreasing}
\label{sssec-un-monotone}
We consider four regions: the region  \(0\leq x\leq \epsilon s_n^{1/2}\), the region
\(\epsilon s_n^{1/2}\leq x \leq 2\epsilon s_n^{1/2}\) where we glue the inner and
intermediate barriers, the  intermediate region \(2\epsilon s_n^{1/2}\leq x\leq 1\),
and finally the region \(x\geq 1\).

In the region \(0 < x \le \epsilon s_n^{1/2}\) we have
\(u_n(x,s_n) = s_n^{k/3}\, W_{K_2}(xs_n^{-k/3})\), which is an increasing function of
\(x\), because \(W\) is increasing.

In the region \(\epsilon s_n^{1/2} \le x \le 2\epsilon s_n^{1/2}\), we have
\begin{align*}
  (u_n)_x(x,s_n) = &\psi_n'(x) \, \Big(s_n^{k/3} W_{K_2}(x s_n^{-k/3}) -
  U_{\delta_n}^-(x, s_n )\Big) \\
  &+ \psi_n(x) W_{K_2}'(xs_n^{-k/3}) + \bigl(1-\psi_n(x)\bigr)
  (U_{\delta_n}^-)_x(x,s_n).
\end{align*}
Using \eqref{eq-W-in-glue-region}, \eqref{eq-Udelta-in-glue-region}, as well as
\(|\psi_n'(x)|\leq Cs_n^{-1/2}\), we estimate the first term above by

\[
  |\psi_n'(x)| \, \Big|s_n^{k/3} W_{K_2}(x s_n^{-k/3}) - U_{\delta_n}^-(x, s_n )\Big|
  \leq C |\psi_n'(x)| s_n^{k-1} \leq C s_n^{k-3/2}.
\]
Furthermore, \eqref{eq-W-in-glue-region}~and~\eqref{eq-Udelta-deriv-in-glue-region}
imply
\[
  \big|W_{K_2}'(xs_n^{-k/3})- 1\big| + \big|(U_{\delta_n}^-)_x(x, s_n)-1\big| \leq C
  s_n^{k-3/2}.
\]
It follows that
\[
  \big| (u_n)_x(x, s_n)-1\big| \leq Cs_n^{k-3/2}
\]
throughout the region \(\epsilon s_n^{1/2}\leq x\leq 2\epsilon s_n^{1/2}\).  Since
\(s_n\to 0\), and \(k\geq 4\), so \(k-3/2>0\), we see that for large enough \(n\) the
function \(x\mapsto u_n(x, s_n)\) is strictly increasing when
\(\epsilon s_n^{1/2}\leq x\leq 2\epsilon s_n^{1/2}\).

Next, in the region \(2\epsilon \sqrt{s_n}\leq x\leq 1\) we have
\[
  u_n(x,s_n) = U_{\delta_n}^-(x,s_n) = \max \Bigl\{
    s_n^{1/2}v_{\delta_n}^-(xs_n^{-1/2}, \log s_n),\; u^-(x, s_n) \Bigr\}.
\]
if \(xs_n^{-1/2}\leq Y_{\delta_n}\), and \(u_n(x, s_n)=u^-(x, s_n)\) otherwise.
It is easy to see that \(x\mapsto u^-(x, s_n)\) is an increasing function.  Concerning
\(v_{\delta_n}^-(y, \log s_n)\) we recall definition~\eqref{eq-barriers-med-v}, i.e.
\[
  v_{\delta_n}^-(y, \log s_n) = y + (K_1-\delta) s_n^{3\gamma} \varphi_k(y) -
  BK_1^2s_n^{6\gamma}g(y) - s_n^{(p+1)\gamma}y^{-p}.
\]
If we choose \(s_n\) small enough then the last three terms will be uniformly small in
\(C^1\) on the fixed interval \(2\epsilon\leq y\leq Y_{\delta_n}\) compared to the
leading term \(y\), so that \(y\mapsto v_{\delta_n}(y, \log s_n)\) is also increasing on
the interval \(2\epsilon\leq y\leq Y_{\delta_n}\).
It follows that \(x\mapsto u_{n}(x, s_n)\) is increasing on
\(2\epsilon s_n^{1/2}\leq x\leq 1\).

The very last situation we must consider is where \(x\geq 1\).  In this case
\eqref{eq-init-data2} implies
\[
  (U_{\delta_n}^-)_x(x,s_n) = u_0'(x) \ge 0.
\]
Since we have covered all cases, the proof of monotonicity of
\(x\mapsto u_n(x, s_n)\) is complete.

\subsection{Proof of~\eqref{eqn-un-barriers-init}}
We turn to the proof that the initial data \(u_n(x, s_n )\) is sandwiched
between the two barriers \(U_{\delta_n}^\pm\), as in~\eqref{eqn-un-barriers-init}.

\begin{lemma} \label{lem-WK2-sandwich} If \(n_0\) is large enough then, for each
\(n\geq n_0\), we can choose \(s_n\in(0, t_n)\) so small that
\begin{equation}
  \label{eq-un-sandwich}
  U_{\delta_n}^-(x,s_n) \le s_n^{k/3} W_{K_2}(x s_n^{-k/3}) \le U_{\delta_n}^+(x,s_n)
\end{equation}
holds for \(0 \le x \le 2\epsilon s_n^{1/2}\).
\end{lemma}

\begin{proof}
In this proof we abbreviate \(y=xs_n^{-1/2}\) and \(z=xs_n^{-k/3}\).

In the region \(0 \le y \le 2\epsilon\) the barriers \(U_{\delta_n}^\pm\)
as defined in \eqref{eqn-ubu}, \eqref{eqn-lbu} are given by
\begin{align*}
  U_{\delta_n}^+(x, s_n) &= \min\Bigl\{
    s_n^{1/2}v_{\delta_n}^+(y, \log s_n) ,\;
    s_n^{k/3}W_{K^+_2(n)}(z) 
    \Bigr\}\\
  U_{\delta_n}^-(x, s_n) &= \max\Bigl\{
    s_n^{1/2}v_{\delta_n}^-(y, \log s_n) ,\;
    s_n^{k/3}W_{K^-_2(n)}(z) + D s_n^{k-1} 
    \Bigr\}
\end{align*}
where \(K^\pm_2(n)=\bigl(K_2^3\pm2\delta_n\bigr)^{1/3}\) (see
section~\ref{sec-choice-constants}).

In Lemma~\ref{lemma-match-middle-inner} we defined
\(Z_n:=Z_{\delta_n}= \frac 43 \, \delta_n^{\frac{-1}{p-2}}\) and showed that the
functions whose max/min define \(U_{\delta_n}^\pm\) cross in the interval
\(\frac12 Z_n \leq z\leq Z_n \).  To prove~\eqref{eq-un-sandwich}
we therefore must show
\begin{equation} \label{eq-WK2-sandwich-inner}
  s_n^{k/3}W_{K^-_2(n)}(z) + D s_n^{k-1} 
  \leq s_n^{k/3} W_{K_2}(z)
  \leq s_n^{k/3}W_{K^+_2(n)}(z)
\end{equation}
if \(0\leq z\leq Z_n\), and
\begin{equation}\label{eq-WK2-sandwich-intermed}
  s_n^{1/2}v_{\delta_n}^-(y, \log s_n)
  \leq s_n^{k/3} W_{K_2}(z)
  \leq s_n^{1/2}v_{\delta_n}^+(y, \log s_n)
\end{equation}
if \(z\geq \frac12 Z_n\) and \(y\leq 2\epsilon\).

Since \(\kappa \mapsto W_\kappa(z) = \kappa W(z/\kappa)\) is strictly increasing
(see~\eqref{eqn-Worder}) it follows from
\(K^+_{2, n} = \bigl(K_2^3+2\delta_n\bigr)^{1/3} > K_2\) that
\(W_{K_2}(z) \leq W_{K^+_2(n)}(z)\) holds for all \(z\geq 0\).  Thus the second
inequality in~\eqref{eq-WK2-sandwich-inner} holds.

The first inequality in \eqref{eq-WK2-sandwich-inner} is equivalent to
\[
  W_{K_2}(z) - W_{K^-_2(n)}(z) \geq Ds_n^{\frac23 k-1}\text{ for all }z\leq Z_n.
\]
By integrating 
\[
  \frac \pd{\pd\kappa}\frac{\pd}{\pd z} W_\kappa(z)
  = - \frac{z}{\kappa^2}W''(z/\kappa)<0
\]
from \(\kappa=K^-_2(n)\) to \(K_2\) we see that \(W_{K_2}(z) - W_{K^-_2(n)}(z) \)
is a decreasing function of \(z\).  We therefore must guarantee
\[
  W_{K_2}(Z_n) - W_{K^-_2(n)}(Z_n) \geq Ds_n^{\frac23 k-1}.
\]
This holds for each \(n\) provided we choose \(s_n\in(0, t_n)\) small enough.

We now consider \eqref{eq-WK2-sandwich-intermed}, which is equivalent to 
\begin{equation} \label{eq-WK2-sandwich-intermed-z}
  s_n^{-\gamma}  v^-_{\delta_n}(s_n^\gamma z, \log s_n)
  \leq W_{K_2}(z)
  \leq s_n^{-\gamma}v^+_{\delta_n}(s_n^\gamma z, \log s_n),
\end{equation}
and we must establish these inequalities for \(\frac12 Z_n\leq  z\leq
2\epsilon s_n^{-\gamma}\). Both inequalities can be proved in the same way,
and we focus on the one involving \(v_{\delta_n}^-\).

Keeping in mind that \(K_2=K_1^3\), the asymptotics \eqref{eqn-Was} for the
Alencar function~\(W\) imply that there is a constant \(C\) such that 
\begin{equation}\label{eq-WK2-asymptotic}
  z+K_1z^{-2}-Cz^{-3} \leq W_{K_2}(z)  \leq z+K_1z^{-2}+Cz^{-3} 
\end{equation}
for \(z\geq1\).  
On the other hand, the definition~\eqref{eq-barriers-med-v} of \(v_\delta^-\)
implies
\begin{align*}
  s_n^{-\gamma}v_{\delta_n}^-
  &(s_n^{\gamma}z, \log s_n)\\
  &= z + (K_1-\delta_n) s_n^{2\gamma}\varphi_k(s_n^\gamma z) - z^{-(p-1)}
  - BK_1^2 s_n^{5\gamma}g(s_n^{-\gamma}z)\\
  &= z + K_1 s_n^{2\gamma}\varphi_k(s_n^\gamma z) -
  \Bigl\{ \delta_ns_n^{2\gamma}\varphi_k(s_n^\gamma z)+ z^{-(p-1)}\Bigr\}
  - BK_1^2 s_n^{5\gamma}g(s_n^{-\gamma}z).
\end{align*}
For \(y\leq 2\epsilon\) we have
\[
  |\varphi_k(y)-y^{-2}|\leq C \qquad \text{ and } \qquad |g(y)| \leq Cy^{-5}.
\]
Hence
\begin{equation} \label{eq-v-dn-asymptotic}
  s_n^{-\gamma}v_{\delta_n}^-
  (s_n^{\gamma}z, \log s_n)
  \geq z + K_1z^{-2} 
  - \Bigl\{ \delta_n z^{-2}+ z^{-(p-1)}\Bigr\}
  -C\bigl(s_n^{2\gamma} + z^{-5}\bigr),
\end{equation}  
where \(C\) is the same for all sufficiently large \(n\in\N\), and for \(1\leq
z\leq 2\epsilon s_n^{-\gamma}\).

If \(z\geq 1\) then \(z^{-5}\leq z^{-3}\), so \eqref{eq-WK2-asymptotic} and
\eqref{eq-v-dn-asymptotic} together lead to
\begin{equation}\label{eq-un-sandwich-almost-there}
  W_{K_2}(z) - s_n^{-\gamma}v_{\delta_n}^- (s_n^{\gamma}z, \log s_n)
  \ge \delta_n z^{-2} - Cs_n^{2\gamma} + z^{-(p-1)} - C z^{-3}.
\end{equation}  
Now choose \(s_n\) so small that \(s_n < \left(\delta_n
Z_n/C\right)^{1/2\gamma}\). Then for all \(z\geq Z_n\) one has
\[
  \delta_n z^{-2} - Cs_n^{2\gamma} \geq \delta_n Z_n^{-2} - Cs_n^{2\gamma} >0.
\]
If we also require \(n\) to be so large that \(Z_n> C^{1/(4-p)}\), then we
have for all \(z\geq Z_n\)
\[
  z^{-(p-1)} - C z^{-3} 
  \geq \left(z^{4-p} - C\right)z^{-3}
  \geq \left(Z_n^{4-p} - C\right)z^{-3}
  >0.
\]
Applying the last two inequalities to~\eqref{eq-un-sandwich-almost-there} we
conclude that the first inequality in~\eqref{eq-WK2-sandwich-intermed-z}
holds. A slight modification of these arguments also proves the second
inequality in~\eqref{eq-WK2-sandwich-intermed-z}.
\end{proof}

\begin{corollary} \label{cor-between} If for each \(n\geq n_0\) we choose
  \(s_n\in(0, t_n)\) as in Lemma~\ref{lem-WK2-sandwich}, then
  \eqref{eqn-un-barriers-init} holds,
  i.e.~\(U_{\delta_n}^-(x, s_n)\leq u_n(x, s_n)\leq U_{\delta_n}^+(x, s_n)\) for all
  \(x\geq 0\).
\end{corollary}

\begin{proof}
If \(x\geq 2\epsilon s_n^{1/2}\) then \(u_n(x, s_n)=U_{\delta_n}^-(x, s_n)\) and
there is nothing to prove.

If \(0\leq x\leq 2\epsilon s_n^{1/2}\), then \(u_n(x, s_n)\) is a convex
combination of \(U_{\delta_n}^-(x, s_n)\) and \(s_n^{k/3}W_{K_2}(s_n^{-k/3}x)\).
We have just shown that this second function lies between the barriers so the
convex combination \(u_n\) also lies between the barriers \(U_{\delta_n}^\pm\).
\end{proof}

\subsection{Monotonicity and uniform \(C^1\) bound for \(u_n(x,t)\)} 
In the following lemma we show that the initial  uniform \(C^1\) bound
\(\|u_n(\cdot,s_n)\|_{C^1} \le C\) persists for as long as each \(u_n(x,t)\) exists,
provided that \(n\) is sufficiently large.

\begin{lemma}
\label{lemma-der-bound}
If \(C_1\) is the upper bound for \((u_n)_x(x, s_n)\) from
Lemma~\ref{lemma-uniform-C1} then for sufficiently large \(n\) we have
\( 0 \le (u_n)_x (x,t) \le C_1\) for all
\((x,t) \in [0, \infty) \times [s_n,\bar t_n)\).
\end{lemma}
In order to prove this Lemma we will apply the maximum principle to the evolution
equation of \((u_n)_x\).  For this we first need the following observation.

\begin{lemma}
\label{lemma-barrier-above-x}
Let \(M\) be the same constant as in Lemma \ref{lemma-outer}.  There is an
\(\alpha>0\) such that for all sufficiently large \(n\), so that
\(U_{\delta_{n}}^-(x,t) \ge x\) for all \(x\in[0,\alpha]\) and all
\(t\in(0, t_n)\).
\end{lemma}

\begin{proof}
In the part of the outer region where \(M\sqrt t \leq x\leq 1\) we have
\(t\leq M^{-2}x^2\), so that
\begin{align*}
  U^-_{\delta_{n}}(x, t) 
  &= u_0(x) -Mt x^{2(k-2)} \\
  &= x + (K_1+o(1))x^{2(k-1)} - Mt x^{2(k-2)} &(x\to0)\\
  &\geq x +\bigl(K_1 - M^{-1} + o(1)\bigr) x^{2(k-1)} &(x\to0).
\end{align*}
If we choose \(M>2/K_1\) then there is an \(\alpha>0\) such that
\(K_1-M^{-1}+o(1)>0\), and hence so that \(U^-_{\delta_{n}}(x, t) > x\).

In the intermediate region the lower barrier is given by
\(t^{1/2} \, v_{\delta_n}^-(t^{-1/2}x, \log t)\), where in the rescaled variables
\((y,\tau)\) we have \(v_{\delta_n}^-(y,\tau) = y + f^-_{\delta_n}(y, \tau)\).
Lemma~\ref{prop-Nf-estimate} tells us that \(f^-_{\delta_n}(y, \tau)\geq 0\), so in
the intermediate region we have \(v^-_{\delta_n}(y, \tau)\geq y\) and hence
\(U^-_{\delta_n}(x, t)\geq x\).

Finally, in the inner region we have
\[
  U_{\delta_{n}}^-(x,t) = t^{k/3} \, w_{n}^-(t^{-k/3}x, \log t)  
\]
and, according to the definition in Lemma~\ref{lemma-inner-barrier2},
\[
  w_{n}^-(z,\tau) = W_{K_2^-}(z) + D\, e^{2\gamma\tau} > W_{K_2^-}(z) > z,
\]
because \(W_\kappa(z)>z\) for all \(z\geq 0\).  This implies
\(U^-_{\delta_n}(x, t)\geq x\) in the inner region as well.
\end{proof}

\begin{proof}[Proof of Lemma \ref{lemma-der-bound}]
If \(u_n\) is one of the approximating solutions of \eqref{eq-u-pde}, then by
differentiating in \(x\) we find that \(\eta := (u_n)_x\) satisfies
\begin{equation}\label{eqn-ux100}
  \eta_t=\cM_n[\eta] - Q_n(x, t)\eta
\end{equation}
where
\[
  \cM_n[\eta] :=\frac{\eta_{xx}}{1+(u_n)_x^2} 
  + \frac 3x \eta_{x}, \text{ and }
  Q_n(x, t) := \frac{2(u_n)_xx^2}{ \bigl(1+(u_n)_x^2\bigr)^2} 
  - \frac 3{u_n^2} + \frac 3{x^2}.
\]
Lemma~\ref{lemma-barrier-above-x} says that
\(u_n(x, t)\geq U^-_{\delta_n}(x, t)\geq x\), so \(Q_n(x, t)\geq 0\).

If the domain of \(\eta\) were bounded we could directly apply the maximum
principle and conclude that \(\eta\) is bounded by its initial values.  Since the
domain is not bounded we consider \(\Omega(x, t) := x^{-1}+\kappa e^tx^2\) in the
domain \(x>0\), \(0\leq t\leq 1\).  (Without loss of generality we assume that
\(\bar t_n\leq 1\) for all \(n\).)  In this region \(\Omega\) satisfies
\begin{align*}
  \Omega_t - \cM_n[\Omega] + Q_n(x, t)\Omega
  &\geq \kappa e^tx^2 - \frac{2x^{-3}}{1+(u_n)_x^2} +3 x^{-3}
  - \frac{2\kappa e^t}{1+(u_n)_x^2} -6\kappa e^t\\
  &\geq \kappa e^tx^2 - 2x^{-3} +3 x^{-3}
  - {2\kappa e^t} -6\kappa e^t\\
  &\geq \kappa e^tx^2 + x^{-3}  - 8\kappa e^t\\
  &\geq \kappa \bigl(x^2-8e\bigr) + x^{-3}.
\end{align*}
If we choose \(\kappa>0\) sufficiently small then the left hand side is positive
for all \(x>0\) and \(t\in[0,1]\).

For any \(\epsilon>0\) we therefore have
\[
  \left(\frac{\pd}{\pd t}-\cM_n+Q_n\right)(\eta+\epsilon\Omega)>0
  \text{ in }(0,\infty)\times[s_n,\bar t_n).
  \]
  Furthermore \(\eta+\epsilon\Omega\to\infty\) as \(x\to\{0,\infty\}\), so the
  maximum principle implies that \(\eta+\epsilon\Omega\) attains its minimum at the
  initial time \(t=s_n\).  Since \(0\leq u_{n, x}(x, s_n)\leq C_1\) (by
  Lemma~\ref{lemma-uniform-C1}) we find that
  \(\eta(x, t)+\epsilon\Omega(x, t)\geq 0\) for all \(\epsilon>0\), which implies
  that \(u_{n, x}(x, t) = \eta(x, t)\geq 0\) for all \(x>0\) and
  \(t\in[s_n, \bar t_n)\).

  By considering \(\eta-\epsilon\Omega\) for arbitrary \(\epsilon>0\) we similarly
  conclude that \(\eta\) is bounded by its largest initial value,
  i.e.~\((u_n)_x(x, t)=\eta(x, t)\leq C_1\) for all \(x>0\) and \(t\in[s_n, \bar t_n)\).
  This finishes the proof of Lemma \ref{lemma-der-bound}.
\end{proof}

\begin{corollary}
  \label{cor-exist-un}
  Let \(u_n(x,t)\) be a solution to the initial value problem
  \eqref{eq-u-pde}-\eqref{eq-u-initcond} with initial data \(u_n(x,s_n)\) as above, and
  let \(n \geq n_0\) where \(n_0\) is sufficiently large so that all previous results
  hold.  Then, the solution \(u(x,t)\) exists for all \(t\in [s_n, t_{n_0})\) and
  satisfies \(U_{\delta_{n_0}}^-(x,t) \le u_n(x,t) \le U_{\delta_{n_0}}^+(x,t)\) and
  \(0\leq (u_n)_x\leq C_1\), for all \(x \ge 0\) and all \(t\in [s_n,t_{n_0})\), where \(C_1\) is
  as in Lemma~\ref{lemma-uniform-C1}.
\end{corollary}

\begin{proof}
We have shown that \((u_n)_x\) is uniformly bounded, and that
\(u_n\geq U^-_{\delta_n}\) has a positive lower bound, and that
\(u_n(x, t)-u_0(x)\) is uniformly bounded (because \(U^\pm_{\delta_n}-u_0\) is
bounded).  The discussion in Section~\ref{sec-short-time-existence} and \eqref{eq-in-between}  then show that
the solution \(u_n\) can be continued for as long as it is contained between two
barriers, i.e.~at least until \(t_{n_0}\), where \(n_0\) does not depend on \(n\).
\end{proof}

\subsection{Uniform lower bound for \(\bar t_n\)}
Each of the approximating solutions \(u_n\) exists at least until time \(\bar t_n\).
We now argue that if \(n_0\) is large enough, then \(\bar t_n > t_{n_0}\) for all
\(n\geq n_0\).

We have already verified for all \(x\geq 0\) and
\(t\in [s_n, \min \{\bar t_n, t_{n_0}\}]\) that the solution \(u_n(x, t)\) remains
between the barriers \(U^\pm_{\delta_{n_0}}(x, t)\) and that its derivative
\((u_n)_x(x,t)\) is uniformly bounded.  Standard estimates for quasilinear parabolic
equations applied to \eqref{eq-r-pde} or \eqref{eq-a-pde} then imply that higher
derivatives of \(u_n\) also are uniformly bounded.  If we had
\(\bar t_n \leq t_{n_0}\), then \(\lim_{t\nearrow \bar t_n}u(x, t)\) would exist, and
we could extend the solution to a larger time interval.  Therefore \(\bar t_n\) would
not be the maximal time of existence for the solution \(u_n\) after all.

\subsection{Proof of the main existence Theorem~\ref{thm-existence}}
We have constructed the sequence of solutions \(u_n\) and have established \textit{a priori}
bounds for its derivatives, which imply that there is a subsequence \(u_{n_j}\) that
converges locally uniformly to a function \(u:[0,\infty)\times(0, t_{n_0}] \to\R\).
The derivative bounds for the approximating solutions \(u_n\) imply that
\(u_{n}\), \( u_{n, x}\), \(u_{n, xx}\), and \(u_{n, t}\) also converge locally
uniformly, and that the limit \(u\) is a solution of \eqref{eq-u-pde}.

We now verify that \(u\) also satisfies the initial and boundary conditions
\eqref{eq-u-bcond}, \eqref{eq-u-initcond}, as well as the asymptotic
description~\eqref{eqn-inner-description} of the inner region.   

\subsubsection{The initial condition}
Let \(n_0\) be so large that all previous results in this section hold.  Then all
solutions \(u_{n_j}\) are caught between the barriers \(U^\pm_{n_0}\), so the limit
also lies between \(U^\pm_{n_0}\).  In the outer region, defined by
\(x\geq M\sqrt t\), the lower (upper) barriers are defined
in~\eqref{eq-outer-upper-barrier} to be the maximum (minimum) of
\(u^\pm(x, t)=u_0(x) \pm Mt\min\{1, x^{2k-4}\}\), and the barriers defined in the
intermediate region.  This implies that for \(x\geq M\sqrt t\) we have
\[
  u_0(x, t) - Mt\max\{1, x^{2k-4}\} \leq u(x, t) \leq  u_0(x, t) + Mt\max\{1, x^{2k-4}\}.
\]
Therefore \(\lim_{t\searrow 0}u(x, t) = u_0(x)\) uniformly for all \(x>0\).

\subsubsection{Boundary condition}
The solutions \(u_n(x, t)\) all satisfy \(u_{n, x}(0, t) = 0\).  They converge in
\(C^1\) to \(u(x, t)\), so we have \(u_x(0, t)=0\) for all \(t\in (0, t_{n_0}]\).

\subsubsection{Asymptotics in the inner region}
To finish the proof of the theorem, we will show that
\[
  \lim_{\tau \to -\infty} w(z,\tau) = W_{K_2}(z)
\]
uniformly on compact sets in \(z\).  This follows almost immediately from
\eqref{eqn-between} and the definition of our barriers \({\tilde u}_n^\pm(x,t)\) in the inner
region.  Using the definitions \( w_n^-(z,\tau) = W_{K^-_2(n)}(z) + D e^{\gamma \tau}\) and
\( w_n^-(z,\tau) = W_{K^+_2(n)}(z) \) from section \ref{subsec-inner}, \eqref{eqn-between}
implies \(w_n^-(z,\tau) \leq w(z,\tau) \leq w_n^+(z,\tau)\) and hence
\begin{equation}\label{eqn-wb}
  W_{K_2^-(n)}(z)+D e^{\gamma\tau_n} \leq w(z, \tau) \leq W_{K^+_2(n)}(z)
\end{equation} 
for all \(z \in [0, Z_{\delta_n}]\), and \( \tau \leq \tau_n:= \log t_n\).

Since \(Z_{\delta_n}:= \frac 43 \delta_n^{- 1/(p-2)} \to + \infty\) and
\(K^\pm_2(n) = (K_2^3\pm 2\delta_n)^{1/3}\to K_2\) as \(n \to +\infty\)
\eqref{eqn-wb} holds on \([0, Z]\times(-\infty,\tau_n)\) for any \(Z>0\), provided
\(n\) is sufficiently large.  The rescaled Alencar solution \(W_K(z)=KW(z/K)\)
depends continuously on \(K\), so after taking the limit \(n\to\infty\) in
\eqref{eqn-wb} we conclude that \(\lim_{\tau \to 0} w(z, \tau)= W_{K_2}(z)\),
uniformly on any bounded interval \(0\leq z\leq Z\), as claimed in
Theorem~\ref{thm-existence}.

\section{Uniform  \(L^\infty\) bound  on the mean curvature}
\label{sec-Hbounded}
\subsection{Bounding \(H\)}
In Theorem \ref{thm-existence} we showed the short time existence of an
\(O(4) \times O(4)\) symmetric MCF solution \(\mathcal{M}_t\) , \(0 < t \leq t_0\),
which is smooth for \(t >0\) and defined by a profile function
\(u: [0, +\infty) \times (0,t_0] \to \R\) which satisfies the initial value problem
\eqref{eq-u-pde}--\eqref{eq-u-initcond} for the given initial data \(u_0(x)\).  In
this section we will show that the mean curvature of \(\cM_t\) is uniformly bounded
on \( [0, +\infty) \times (0,t_0]\) despite the fact that the initial data \(u_0\) is
singular at the origin.  The life time of the solution is \(t_0=t_{n_0}\) for some
large enough \(n_0\).

\begin{theorem} \label{thm-Hbounded} Let \(\cM_t\), \(0 < t \leq t_0\), be the
\(O(4) \times O(4)\) symmetric MCF solution constructed in Theorem
\ref{thm-existence}.  Then
\begin{equation}\label{eqn-HLinfty}
  \sup_{0<t\leq t_0}\sup_{\cM_t} H <\infty.
\end{equation}
\end{theorem}

To prove this theorem we will first show, using a direct argument, that \(H(x, t)\)
is uniformly bounded in the outer region \(x\ge M\sqrt{t}, \, 0 < t \leq t_0\).
Then, using an argument by contradiction, that is strongly inspired
by Stolarski's approach in \cite{S}, we will show that \(H(x,t)\) is uniformly
bounded in the remaining region \(x \leq M\, \sqrt{t}, \, 0 < t \leq t_0\).

\subsection{Bounding \(H(x,t)\) in the outer region} 
Assume  without loss of generality that  \(t_0 \leq  M^{-2}\).  In this section
we will show that  \eqref{eqn-HLinfty} holds in the outer region \(\cO_M = \{
  (x,t) \,\, | \,\,\, x\ge M\sqrt{t}, \, 0 < t \leq  t_0 \}\), as stated next. 

\begin{lemma}\label{lem-Houter}
There exists a uniform constant \(C > 0\)   so that 
\begin{equation}\label{eqn-HLinfty-outer}
  \sup_{(x,t) \in \cO_M } H(x,t) \leq C
\end{equation}
for all \(t \in (0, t_0]\), provided \(t_0 < M^{-2}\). 

\end{lemma}

\begin{proof}
We fix a point \((x_1, t_1) \in \cO_M\).  We   first deal with the case when \(x_1\in(0,1)\).
Consider the function
\[
  U(\xi, s) = x_1^{-1} u(x_1\xi, t_1 + x_1^2 s).
\]
This function satisfies
\begin{equation}
  \label{eq-U-par}
  U_s = \frac{U_{\xi\xi}}{1+U_\xi^2} + \frac{3}{\xi}U_\xi - \frac{3}{U}
\end{equation}
in the region 
\[
  \mathcal Q = \Big \{(\xi, s) \colon  \tfrac 12 < \xi < \tfrac 32, \,
  -\frac{t_1}{x_1^2} < s\leq 0\Big \}.
\]

By  \eqref{eqn-between} the  solution  \(u\) lies between our upper and lower barriers constructed in Proposition \ref{prop-barriers}.  This
implies that for all \((x,t)\in \mathcal{O}_M\), with \(x\in (0,1)\),
\[
  |u(x,t ) - u_0(x)|\leq Mt\, x^{2k-4}
\]
and hence, for \(\xi\in(\frac 12, \frac 32)\) and \(-t_1x_1^{-2} < s \leq 0\), 
\[
  \left|U(\xi, s) - x_1^{-1} u_0(x_1\xi)\right|
  \leq
  M \, (t_1+x_1^2s) \, x_1^{2k-5}\xi^{2k-4}
  \leq
  CM \, t_1 \, x_1^{2k-5}.
\]
In the outer region we also have  \(x_1^2 \geq t_1\), so
\[
  \left|U(\xi, s) - x_1^{-1} u_0(x_1\xi)\right|
  \leq
  CM x_1^{2k-3}.
\]
The initial profile \(u_0\) satisfies \(x\leq u_0(x) \leq x+C\, x^{2k-2}\) for
\(0<x<2\).  Rescaling leads to
\[
  \left| x_1^{-1} u_0(x_1\xi) - \xi \right|
  \leq C \, x_1^{2k-3}.
\]
The last two inequalities together imply that
\begin{equation}
  \label{eq-U-unif-C0}
  \left| U(\xi, s) - \xi \right| \leq C \, x_1^{2k-3},
\end{equation}
holds on \(\mathcal{Q}\). 
Therefore the function 
\[
  F(\xi, s) \stackrel{\rm def}= \frac{U(\xi, s) - \xi}{x_1^{2k-3}}
\]
which satisfies equation 
\begin{equation}
  \label{eq-outer-F}
  F_s = \frac{F_{\xi\xi}}{1+U_\xi^2} + \frac{3}{\xi} F_\xi + \frac{3}{\xi \, U(\xi, s)} F
\end{equation}
is bounded on \(\mathcal{Q}\) by \(|F(\xi, s)|\leq C\) for some constant \(C\) that does not depend on \((x_1,t_1)\).

\begin{claim}\label{claim-UUU}
  \(U\) and \(1+U_\xi^2\) are H\"older continuous on 
  \[
    \mathcal{Q}' = \left\{(\xi, s) \colon \tfrac 23 < \xi < \tfrac 43, \,
    -\frac{t_1}{2x_1^2} < s\leq 0\right\}
  \]
  uniformly in \((x_1,t_1)\).
\end{claim}

\begin{proof}
By \eqref{eq-U-unif-C0} we have that \(\|U\|_{C^0(\mathcal{Q})} \le C\),
for a uniform constant \(C\), independent of \((x_1,t_1)\), where \(x_1\in (0,1)\).  Furthermore, in \(\mathcal{Q}\) we also have
\begin{equation}
\label{eqn-unx100}
  |U_{\xi}(\xi,s)| = |u_x(x_1\xi, t_1+x_1^2 s)| \le C,
\end{equation}
where \(C\) is a uniform constant, independent of \((x_1,t_1)\).  This follows by Lemma \ref{lemma-der-bound} and the fact that \(u_n(x,t)\) smoothly converges as \(n\to\infty\) to \(u(x,t)\), for all \(x > 0\) and \(t\in (0,t_1]\).  Since \(U(\xi,s)\) satisfies a uniformly parabolic equation \eqref{eq-U-par}, standard regularity theory applied to \eqref{eq-U-par} implies that there exists a uniform constant \(C\), independent of \((x_1,t_1)\) so that \(|U_{\xi\xi}(\xi,s)| \le C\) in \(\mathcal{Q}'\).  All these imply \(U\) and \(1+U_{\xi}^2\) are uniformly H\"older continuous functions on \(\mathcal{Q}'\) as claimed.
\end{proof}

Interior parabolic regularity for \eqref{eq-outer-F} then implies that \(F\), \(F_\xi\), 
and \(F_{\xi\xi}\) are uniformly bounded (and even H\"older) on \(\mathcal{Q}'\).  
We conclude that for some constant \(C\) that does not depend on \((x_1,t_1)\) we have
\[
  |F_s(1,0)| \leq C.
\]
In terms of the original solution \(u(x,t)\) this then implies 
\[
  |u_t(x_1, t_1)| \leq C \, x_1^{2k-4} \leq C,
\]
where we have used \(k\geq 4\) and \(x_1\leq 1\) in the last step.  We conclude that \(|H(x_1, t_1)|\leq |u_t(x_1,t_1)|\) is uniformly bounded 
for all  \((x_1,t_1) \in \cO_M\) with  \(x_1 \leq 1\).  

\smallskip 
Let us now deal with the case where \(x_1 \geq 1\), in which case   \(t_1/x_1^2 \leq t_1\) is small (since \(t_1 \leq t_0\) and we have assumed
that \(t_0 < M^{-2}\) and \(M\) is large).  
The interior regularity estimates then provide a bound for
\(|F_{\xi\xi}(1,0)|\) in terms of \(\sup_{\mathcal{Q}}|F|\) and
\(\sup_{\frac 12<\xi<\frac 32} |F_{\xi\xi}(\xi,-t_1/x_1^2)|\).  We have
\[
  F_{\xi\xi} (\xi,-t_1/x_1^2) = x_1^{-(2k-3)} U_{\xi\xi}(\xi,-t_1/x_1^2)
  = x_1^{-(2k-5)} u_0''(x_1\xi).
\]
By assumption we have \(|u_0''(x)|\lesssim x^{2k-4}\), and hence
\[
  \sup_{1 < \xi < 3/2}\left|F_{\xi\xi} (\xi,-t_1/x_1^2)\right| 
  \lesssim x_1 \lesssim 1.
\]
In our case where \(t_1/x_1^2\) is small, this implies that \(F_{\xi\xi}(1,0)\) and
hence \(F_s(1,0)\) are bounded uniformly.  It follows that
\(|H(x_1, t_1)|\leq |u_t(x_1,t_1)|\) is also uniformly bounded if
\((x_1,t_1) \in \cO_M\) with \(x_1 \geq 1\).

\smallskip

Combining the two cases \(x_1 \in (0,1)\) and \(x_1 \geq 1\) leads to  \eqref{eqn-HLinfty-outer}, finishing the proof of the proposition. 
\end{proof}

\subsection{Second order derivative bounds for \(x\leq  M \sqrt{t}\)} Before we
bound \(H(x,t)\) in the intermediate and inner regions, we will establish the
following  crucial for our purposes weighted  \(C^2\)  bound for our
approximating sequence of solutions \(u_n(x,t)\) which were  defined in Section
\ref{sec-existence}. 

\smallskip 

\begin{lemma}\label{lem-C2-inter} There exists \(n_0\) sufficiently large and a
constant \(C\) independent of \(n\) so that for all \(n \geq n_0\) the bound
\begin{equation} \label{eqn-unxx}
  |(u_n)_{xx}(x,t)| \leq C\, t^{ - k/3}
  \bigl(1+t^{-k/3} x \bigr)^{-4}
\end{equation}
holds for all \(0 \leq x \leq M \sqrt{t}\), \, \(t \in [s_n, t_0]\).
\end{lemma}
\begin{proof}
The proof follows from scaling and standard regularity theory for linear and
quasilinear parabolic equations.  We repeatedly use the first order derivative
bound \(0\leq u_x(x, t)\leq C_1\) from Corollary~\ref{cor-exist-un}, as well
as the derivative bounds
\begin{equation}
\label{eq-un-higher-deriv-bound}
  |\partial^j u_n(x,s_n)| \leq C\, s_n^{ - (j-1) k/3} \bigl (1+s_n^{-k/3} x
  \bigr)^{-(j+2)}, \qquad j=2,3. 
\end{equation}
holding at the initial time \(s_n\), which were shown in Lemma
\ref{lemma-uniform-C1}.

Since our solutions \(u_n(x,t)\) scale differently in the intermediate and inner
regions we need to treat the cases
\(x \in [2R \, t^{k/3} , M t^{1/2} ] \) and
\(x \in [0, 2R \, t ^{k/3}]\) separately.  We will choose \(R\) in
the proof of Case 1 below to be a sufficiently large constant which is independent
of \(n\).  Then for this choice of \(R\) we will show that Case 2 holds.  In both
cases we will assume that \(n \geq n_0\) and \(s_n \leq t \leq t_0\), and \(n_0\)
will be chosen sufficiently large and \(t_0\) will be chosen to be sufficiently
small, uniformly in \(n\).  

\smallskip We start by fixing \(n \geq n_0\) and a point \((x_1, t_1)\) where
\(0 \leq x_1 \leq M \sqrt{t_1}\), \, \(t_1 \in [s_n, t_0]\).

{\bf Case 1 :} Assume
\(x_1\in [2R \, t_1^{k/3} , M t_1^{ 1/2} ] \), where \(R\) is a
sufficiently large constant.  Similarly to the proof of Lemma \ref{lem-Houter}, we
consider the rescaling
\[
  {\tilde u}_n(\xi, s) = x_1^{-1} u_n(x_1\xi, t_1 + x_1^2 s)
\]
which satisfies equation
\begin{equation}
  \label{eq-U-parn}
  ({\tilde u}_n)_s = \frac{({\tilde u}_n)_{\xi\xi}}{1+{\tilde u}_{n\xi}^2} + \frac{3}{\xi}({\tilde u}_n)_\xi - \frac{3}{{\tilde u}_n}
\end{equation}
in the region
\[
  \mathcal Q_n = \Big \{(\xi, s) \colon \tfrac 12 < \xi < \tfrac 32, \, -
  \frac{t_1-s_n}{x_1^2} < s\leq \frac{t_0-t_1}{x_1^2}\Big \}.
\]
We subdivide into the {\em two cases}
\({\displaystyle \frac{t_1-s_n}{x_1^2} > \frac 1{2M^2}}\) and
\({\displaystyle \frac{t_1-s_n}{x_1^2} \leq \frac 1{2M^2}}\).

{\bf Case 1a :} If \({\displaystyle \frac{t_1-s_n}{x_1^2} > \frac 1{2M^2}}\) then
the parabolic square
\[
  \cQ_M'= \Big \{(\xi, s) \colon \tfrac 12 < \xi < \tfrac 32, \, - \frac 1{2M^2}<
  s\leq 0\Big \}
\]
has fixed size (independent of \((x_1,t_1)\) and \(n\)) and satisfies
\(\cQ_M' \subset \mathcal Q_n\).  We will restrict to \(\cQ_M'\).

For any \((\xi,s) \in \cQ_M'\) we have
\(x:=x_1 \xi \in [R \, t_1^{k/3} , 2M t_1^{1/2} ] \) and
\(t:= t_1 + x_1^2 s \in [t_1/2, t_1]\).  In particular we have
\(y:=x t^{-\frac 12} \in [R \, t_1^\gamma, 2\sqrt{2} M]\), i.e., \((x,t)\) lies in
the intermediate region, a fact that will be used momentarily.  To obtain the desired
bound on \(u_{xx}(x_1,t_1)\), we will bound \(U_{\xi\xi}(1,0)\) by applying interior
parabolic regularity estimates to the function \({\tilde u}_n(\xi,s) -\xi\) defined in
\(\cQ_M'\).  We first estimate the \(L^\infty\) norm of this function on \(\cQ_M' \)
by bounding \(|u_n(x,t)-x|\), for \(x= x_1 \xi, t=t_1 + x_1^2 s\) where
\((\xi,s) \in \cQ_M'\).

By \eqref{eqn-between} the solution \(u\) lies between our upper and lower barriers
constructed in Proposition \ref{prop-barriers}.
Hence,
\begin{equation}\label{eqn-444} |u_n(x,t ) - x| \leq \max \big \{ |U_{\delta_n}^+(x,t
  ) - x|, |U_{\delta_n}^-(x,t ) - x| \big \}
\end{equation}
for all \(n \geq n_0 \) sufficiently large.  Using the definition of our barriers
\({\tilde u}_n^\pm(x,t ) \) (see \eqref{eqn-ubu} and \eqref{eqn-lbu}) the difference
\( |{\tilde u}_n^\pm(x,t ) - x|\) for \(n \geq n_0\) is bounded by
\(t^{\frac 12} |f^{\pm}_{\delta_{n_0}}( x t^{-\frac 12}, t)| \)
(\(f^{\pm}_{\delta_{n_0}}\) was defined in \eqref{eq-barriers-med}).  The latter can
be bounded by \(2K_1 t^{k-1} \varphi_k ( x \, t^{-\frac 12}) \), provided that
\(t_0\) is sufficiently small.  This follows from the definition of
\(f^{\pm}_{\delta_{n_0}}\) and our estimates in section \ref{subsec-inter}, after
expressing these estimates in the \((x,t)\) variables using \eqref{eq-par-scaling}.
Since \(\varphi_k (y) \leq C_k \, \big ( y^{2k-2} + y^{-2} \big )\) with
\(y:=x t^{-\frac 12} \in [R t_1^\gamma, 2\sqrt{2} M]\) and \(t \in [t_1/2, t_1]\), we
get
\begin{equation}\label{eqn-446}
  \max \big \{ |U_{\delta_n}^+(x,t ) - x|, |U_{\delta_n}^-(x,t ) - x| \big \}
  \leq C   \, t^{k-1} \,  (x\,  t^{-\frac 12})^{-2} 
  \le C \,  x_1^{-2}  t^{k}
\end{equation}
for some constant \(C\) (depending only on \(k, M\)) which is uniform in
\((x_1,t_1)\) and \(n\).  Combining \eqref{eqn-444} and \eqref{eqn-446} while using
\(t=t_1+x_1^2s \leq t_1\) yields
\begin{equation}\label{eqn-445}
  \left | {\tilde u}_n(\xi,s) - \xi \right |
  \leq C \, x_1^{-3} t^k_1 \,\qquad \mbox{in}\,\, \cQ_M'.
\end{equation}
It follows that the function
\[
  F_n(\xi, s) \stackrel{\rm def}= x_1^3\, t^{-k}_1\, \big ( {\tilde u}_n(\xi, s) - \xi \big
  )
\]
which satisfies equation
\begin{equation}
  \label{eqn-Fn}
  (F_n)_s = \frac{(F_n)_{\xi\xi}}{1+{\tilde u}_{n\xi}^2} + \frac{3}{\xi} (F_n)_\xi + \frac{3}{\xi \, {\tilde u}_n(\xi, s)} F_n 
\end{equation}
is uniformly bounded in the parabolic cube \(\cQ'_M\), namely
\(\|F_n\|_{C^0(\cQ_M')} \le C\), where the constant \(C\) is independent of
\((x_1,t_1)\) and \(n\).

\begin{claim}\label{claim-UUUn}
  \({\tilde u}_n\) and \(1+{\tilde u}_{n\xi}^2\) are H\"older continuous on the parabolic cube
  \[
    \cQ_M''= \Big \{(\xi, s) \colon \tfrac 14 < \xi < \tfrac 54, \, - \frac 1{4M^2}<
    s\leq 0\Big \} \subset \cQ_M'
  \]
  uniformly in \((x_1,t_1)\) and \(n\).  Furthermore \(1/4 \leq {\tilde u}_n(\xi,s) \leq 2\),
  for all \((\xi,s) \in \cQ''_M\).

  \begin{proof}
  Since \(x_1 \geq R \, t_1^{k/3}\), by \eqref{eqn-445} we have that
  \( \left | {\tilde u}_n(\xi,s) - \xi \right | \leq C R^{-3} \), and since the constant \(C\)
  doesn't depend on \(R\), we may choose \(R\) sufficiently large so that
  \(1/4 \leq {\tilde u}_n(\xi,s) \leq 2\) for all \((\xi,s) \in \cQ_M'\).  In addition
  \eqref{eqn-unx100} implies that
  \(|{\tilde u}_{n\xi}(\xi,s)| = |(u_n)_x(x_1\xi, t_1+x_1^2 s)| \le C\) in \(\cQ_M'\), where
  in both cases \(C\) is a uniform constant, independent of \((x_1,t_1)\) and \(n\).
  It follows that \({\tilde u}_n(\xi,s)\) satisfies in \(\cQ_M'\) a uniformly parabolic
  equation \eqref{eq-U-parn} with bounded coefficients, and therefore standard
  interior (in space-time) regularity theory applied to the quasilinear equation
  \eqref{eq-U-parn} implies the existence of a uniform constant \(C\), independent of
  \((x_1,t_1)\) and \(n\), so that \(|{\tilde u}_{n\xi\xi}(\xi,s)| \le C\) in
  \(\cQ_M'' \subset \cQ_M'\).  All the above give us that \({\tilde u}_n\) and
  \(1+{\tilde u}_{n\xi}^2\) are uniformly H\"older continuous functions on \(\cQ_M''\) as
  claimed.
  \end{proof}

\end{claim}

Claim \ref{claim-UUUn} implies that equation \eqref{eqn-Fn} is uniformly parabolic in
\(\cQ_M''\) and its coefficients are H\"older continuous (uniformly in \((x_1,t_1)\)
and \(n\)).  Interior (in space-time) Schauder theory applied to \eqref{eqn-Fn} in
\(\cQ_M''\) bounds \(|(F_n)_{\xi\xi}(1,0)|\) in terms of \(\|F_n\|_{C^0(\cQ_M'')}\),
concluding that \(|(F_n)_{\xi\xi}(1,0)| \leq C\), for a uniform constant \(C\).
Equivalently, \(|({\tilde u}_n)_{\xi\xi}(1,0)| \le C \, x_1^{-3} t_1^k\) and converting back
to the original solution gives the bound
\(|(u_n)_{xx}(x_1, t_1)| \leq C \, x_1^{-4}\, t_1^k.\) In the considered region we
have \(x_1 t^{-\frac k3} \geq R\), thus
\(t^k_1 \, x_1^{-4} = t^{-\frac k3}_1\big ( t_1^{-\frac k3}x_1 \big )^{-4} \leq C\,
t^{-\frac k3}_1 \big (1+ x_1 t_1^{-\frac k3} \big )^{-4} \) (where \(C\) depends on
\(R\)).  We conclude that the desired bound \eqref{eqn-unxx} holds when
\(x_1\in [2R \, t_1^{k/3} , M t_1^{\frac 12} ] \) and
\(\frac{t_1-s_n}{x_1^2} > \frac 1{2M^2}\).

\smallskip

{\bf Case 1b :} If \, \({\displaystyle \frac{t_1-s_n}{x_1^2} \leq \frac 1{2M^2}}\),
then \(x_1 \leq M t_1^{\frac 12}\) implies that
\({\displaystyle t_1-s_n \leq \frac {x_1^2}{2M^2} \leq \frac {t_1}2}\), and hence in
this case \(t_1 \in [s_n, 2s_n]\).  This in turn gives \(x_1 \leq M \sqrt{2s_n}\),
implying in particular that
\({\displaystyle \frac{t_0-t_1}{x_1^2} \geq \frac{t_0-2s_n}{2M^2s_n} \geq 1}\),
provided that \(n\geq n_0\) with \(n_0\) sufficiently large.  Hence the cube
\[
  \cQ_n' = \Big \{(\xi, s) \colon \tfrac 12 < \xi < \tfrac 32, \, -
  \frac{t_1-s_n}{x_1^2} < s\leq - \frac{t_1-s_n}{x_1^2} + 1 \Big \}
\]
has fixed size and satisfies \(\cQ_n' \subset \cQ_n\).  The difference between this
and the previous case is that the cube \(\cQ_n'\) starts at
\(s=- \frac{t_1-s_n}{x_1^2}\) corresponding to initial time \(t=s_n\) for the
solution \(u_n(x,t)\).  This means that our estimates need to include bounds on the
initial data \(u_n(x,s_n)\).

As in the previous case, we will begin by bounding \(|{\tilde u}_n(\xi,s) - \xi|\) in
\(\cQ_n'\).  For any \((\xi,s) \in \cQ_n'\) we have
\(x:=x_1 \xi \in [R \, t_1^{k/3} , 2M \sqrt{t_1} ] \subset [R \,
s_n^{k/3} , 2M \sqrt{2s_n}]\) (using \(t_1 \in [s_n, 2s_n]\)) and
\(t:= t_1 + x_1^2 s \in [s_n,  (2M^2 +2) \, s_n]\) (using
\(x_1 \leq M t_1^{\frac 12}\)).  Hence,
\(y:=x t^{-\frac 12} \in [ \frac R{\sqrt{2} M} s_n^\gamma , 2 \sqrt{2} M] \) which
shows that the point \((x,t)\) belongs to the intermediate region.  Now similar
arguments as in Case 1a imply that bounds \eqref{eqn-444} and \eqref{eqn-446} hold
(with \(s_n\) instead of \(t_1\)).  We conclude that
\(|u_n(x,t) -x| \leq C \, x_1^{-2} \, s_n^{3\gamma + \frac 32}\) holds at
\(x=x_1 \xi\), \(t:=t_1+s \xi_1^2\), for any \((\xi,s) \in \cQ'_n\), where \(C\) is
independent of \((x_1,t_1)\) and \(n\).  In terms of \({\tilde u}_n(\xi,s)\) we obtain
\begin{equation}\label{eqn-447}
  \left | {\tilde u}_n(\xi,s) - \xi \right |  \leq C \,  x_1^{-3} s_n^{3\gamma + \frac 32} \leq  C \,  x_1^{-3} t_1^{k} \qquad   \mbox{in}\,\, \cQ_n'. 
\end{equation}

\smallskip

\begin{claim}\label{claim-UUUn2}
  \({\tilde u}_n\) and \(1+{\tilde u}_{n\xi}^2\) are H\"older continuous on the parabolic cube
  \[
    \cQ_n'' := \Big \{(\xi, s) \colon \tfrac 34 < \xi < \tfrac 54, \, -
    \frac{t_1-s_n}{x_1^2} < s\leq - \frac{t_1-s_n}{x_1^2} + 1 \Big \} \subset \cQ_n'
  \]
  uniformly in \((x_1,t_1)\) and \(n\).  Furthermore, \(1/4 \leq {\tilde u}_n(\xi,s) \leq 2\)
  for all \((\xi,s) \in \cQ_n'\).

  \begin{proof}
  Similarly to Claim \ref{claim-UUUn}, the bounds \eqref{eqn-447} and
  \eqref{eqn-unx100} imply that on \(\cQ'_n\) we have \(1/4 \leq {\tilde u}_n \leq 2\) and
  \(|{\tilde u}_{n\xi}| \le C\).  In addition, for \(j=2,3\) we have
  \begin{equation}\label{eqn-Un222}
    \sup_{ \frac 12 \leq \xi \leq \frac 32}
    \Big  | \partial^j_\xi \, {\tilde u}_n \bigl(\xi, - \frac{t_1-s_n}{x_1^2}\bigr)\Big  |
    \leq x_1^{j-1}, \qquad
    \sup_{ \frac {x_1}2 \leq x \leq \frac{3x_1}2}
    \big| \partial_x^j \, u_n(x, s_n) \big| \leq C  \,  x_1^{-3} s_n^k \leq C
  \end{equation}
  where we used \eqref{eq-un-higher-deriv-bound} and our assumption
  \(x_1 \geq 2R t_1^{\frac k3}\) combined with \(t_1\in [s_n, 2s_n]\).  In all the
  above bounds \(C\) is a uniform constant, independent of \((x_1,t_1)\) and \(n\).
  Since \({\tilde u}_n(\xi,s)\) satisfies a uniformly parabolic equation \eqref{eq-U-parn} in
  \(\cQ_n'\), standard interior (in space) theory for quasilinear equations applied
  to \eqref{eq-U-parn} yields the \(C^2\) bound
  \(\|{\tilde u}_{n\xi\xi}\|_{C^2(\cQ_n'')} \le C\) (and even a \(C^{2,1}\) bound), where
  \(C\) is a constant that depends only on \(\|{\tilde u}_n\|_{C^0(\cQ_n')}\) and
  \(\|{\tilde u}_n(\cdot, - \frac{t_1-s_n}{x_1^2}) \|_{C^3([\frac \xi2, \frac {3\xi}2])}\),
  therefore \(C\) is uniform in \((x_1,t_1)\) and \(n\), since these bounds are as
  well.  We conclude that \({\tilde u}_n\) and \(1+{\tilde u}_{n\xi}^2\) are uniformly H\"older
  continuous functions on \(\cQ_n''\), finishing the proof of the claim. \end{proof}

\end{claim}

Consider the function
\(F_n(\xi, s) := x_1^{3}t^{-k}_1 \, \big ({\tilde u}_n(\xi, s) - \xi \big )\) on \(\cQ_n''\)
which satisfies equation \eqref{eqn-Fn} and the uniform bound
\(\|F_n\|_{C^0(\mathcal{Q}''_n)} \le C\), where \(C\) is independent of \((x_1,t_1)\)
and \(n\).  Claim \ref{claim-UUUn2} implies that \(F_n(\xi,s)\) satisfies a uniformly
parabolic equation \eqref{eqn-Fn} on \(\cQ_n''\) with coefficients which are
uniformly H\"older continuous.  Therefore, standard interior (in space) Schauder
estimates applied to \eqref{eq-U-parn} on the cube \(\cQ_n''\) imply that
\(|(F_n)_{\xi\xi}(1,0)|\) can be bounded in terms of
\(\|F_n\|_{C^0(\mathcal{Q}''_n)}\) and
\(\| F_n(\cdot, - \frac{t_1-s_n}{x_1^2})\|_{C^{2,1}([\frac 34,\frac 54]) }\).  We
have just seen that \(\|F_n\|_{C^0(\mathcal{Q}''_n)} \leq C\).  We will next show the
bound
\(\| F_n(\cdot, - \frac{t_1-s_n}{x_1^2})\|_{C^{3}([\frac 34,\frac 54]) } \leq C\).
First, \eqref{eqn-Un222} and the definition of \(F_n\), give
\(|\partial_\xi^j F_n(\xi, - \frac{t_1-s_n}{x_1^2}) = x_1^3\, t^{-k}_1\,
|\partial_\xi^j {\tilde u}_n(\xi, s)| \leq C t^{-k}_1 s_n^{k} \leq C\), for \(j=2,3\) and all
\(\xi \in [\frac 34,\frac 54]\).  The bound for \(j=1\) follows similarly from
\(0 \leq (u_n)_x (x, s_n) \leq C\).  In all the above bounds \(C\) is independent of
\((x_1, t_1)\) and \(n\).

\smallskip

We conclude that \(|(F_n)_{\xi\xi}(1,0) | \leq C\), where \(C\) is independent of
\((x_1,t_1)\) and \(n\), and similarly to the Case 1a, the desired bound
\eqref{eqn-unxx} holds for \(x_1\in [2R \, t_1^{k/3} , M t_1^{1/2} ] \) and
\(\frac{t_1-s_n}{x_1^2} \leq \frac 1{2M^2}\).  This completes the argument in Case
1b.

\smallskip

{\bf Case 2 :} Suppose next that \(x_1\in [0, R \, t_1^{k/3}]\), that
is \((x_1,t_1)\) belongs to the tip region.  Here \(R\) is a large fixed constant,
chosen as in Case 1.  In this case we will not scale around \(x_1\), but around the
origin and we will show 
\begin{equation}\label{eqn-unxx2}
  \sup_{x \in [0, R\, t_1^{k/3} ]} |(u_n)_{xx} (x,t_1) | \leq C\, t^{-k/3}_1,
  \qquad 0 < t_1 \leq t_0
\end{equation}
for a uniform constant \(C\) independent of \(n\) and \(t_1\) (\(C\) may depend on
\(R\)).  This estimate is equivalent to \eqref{eqn-unxx} because in the considered
region one has \(x_1 t_1^{-\frac k3} \leq R\).

To this end we set \(\alpha:= \frac k3 \geq 1\) for simplicity, and introduce the
rescaled function
\begin{equation}\label{eqn-Un333}
  \bU_n(\xi, s) = t_1^{-\alpha}\, u_n( t_1^{\alpha}\, \xi, t_1 + t_1^{2\alpha} s)
\end{equation}
which satisfies equation \eqref{eq-U-parn} in the region
\[
  \mathcal Q_n = \Big \{(\xi, s) \colon 0 \leq \xi \leq 2R, \, -
  \frac{t_1-s_n}{t_1^{2\alpha}} < s\leq \frac{t_0-t_1}{t_1^{2\alpha}}\Big \}.
\]
Bound \eqref{eqn-unxx2} is equivalent to
\begin{equation}\label{eqn-U300}
  \sup_{\xi \in [0,R]} |({\tilde u}_n)_{\xi\xi}(\xi,0) | \leq C
\end{equation}
and will follow by applying standard regularity theory to equation \eqref{eq-U-parn}
in an appropriate cube \(\cQ'_n \subset \cQ_n\).

First, one needs to bound \({\tilde u}_n\) on \(\cQ_n'\) from above and below away from zero.
To this end, observe that \eqref{eqn-UUU}, \eqref{eqn-ubu}--\eqref{eqn-lbu} and the
definition of the inner region barriers in section \ref{subsec-inner} give
\begin{equation}\label{eqn-un500}
  t^\alpha W_{K_2^-(n_0)}\left( x \,  t^{-\alpha}\right) + D \,  e^{2\gamma \log t} \leq u_n(x,t) \leq  t^\alpha W_{K_2^-(n_0)}\left( x  t^{-\alpha}\right)
\end{equation}
for all \(n \geq n_0\) sufficiently large and all \(x \in [0, Z\, t^\alpha]\) (for
any \(Z >0\)) and \(t \leq t_0\).  Here \(D >0\), thus we can drop the small term
\(D \, e^{2\gamma \log t}\).  The above estimate when expressed in terms of
\({\tilde u}_n(\xi,s)\) gives
\begin{equation}\label{eqn-un666}
  \vartheta_n(s) \,  W_{K_2^-(n_0)}\big( \frac{\xi}{\vartheta_n(s)} \big)   
  \le \bU_n(\xi, s) \le \vartheta(s) 
  \, W_{K_2^+(n_0)}\big(  \frac{\xi}{\vartheta_n(s)}  \big) 
\end{equation}
where \(\vartheta_n(s) := t^\alpha t_1^{-\alpha} =
\bigl(1+t_1^{2\alpha-1}s\bigr)^\alpha\).  Note that in order to obtain
\eqref{eqn-un666}  from \eqref{eqn-un500} we need  to have  \(\frac{ \xi}{
  \vartheta_n(s)}   \leq Z\),  for all \((\xi,s) \in \cQ_n'\),  for some \(Z>0\)
which is  independent of \((\xi,s) \in \cQ_n'\).  This will be checked below.
We   need to consider two cases,  
\((t_1-s_n)t_1^{-2\alpha} >1\) and 
\((t_1-s_n)t_1^{-2\alpha}  \leq  1\), and choose \(\cQ'_n\) appropriately. 

{\bf Case 2a :} If \((t_1-s_n)t_1^{-2\alpha} >1 \) then we
restrict to the parabolic cube of fixed size
\[
  \cQ' = \left\{(\xi, s) \colon 0 \leq \xi \leq 2R, \, -1 < s\leq 0\right\}
\]
(independent of \(t_1\) and \(n\)), which obviously satisfies
\(\cQ' \subset \mathcal Q_n\).  We will restrict to \(\cQ'\), where \(s \in (-1,0]\)
readily implies the bounds
\(\vartheta_n(s) \geq (1- t_1^{2\alpha-1})^\alpha \geq 1/2\) and
\( \vartheta_n(s) \leq 1\) and (for the former use \(t_1 \leq t_0\), where \(t_0 \)
can be chosen sufficiently small).

Using \(\xi \, \vartheta_n^{-1} \leq 4R\) and \(1/2 \leq \vartheta_n \leq 1\), we
readily conclude from \eqref{eqn-un666} that there exist a uniform in \(n\) and
\(t_1\) constant \(C >0\) (depending on \(\inf_{z \in [0,4R]} W_{K_2^-(n)}(z)\) and
\(\sup_{z \in [0,4R]} W_{K_2^+(n)}(z)\)) such that
\begin{equation}\label{eqn-Unab}
  0 < C^{-1}   \leq {\tilde u}_n(\xi, s) \leq C, \qquad \mbox{for all} \,\, (\xi,s) \in \cQ'.
\end{equation}
Furthermore, by \eqref{eqn-unx100} we have \(\| {\tilde u}_{n\xi} \|_{C^0(\cQ')} \leq C\),
where \(C\) is again independent of \(n\) and \(t_1\).  Standard interior (in
space-time) regularity theory applied to \eqref{eq-U-parn} implies that there exists
a uniform constant \(C\), independent of \(n\) and \(t_1\), so that
\(\sup_{\xi \in [0, R] } |({\tilde u}_n)_{\xi\xi}(\xi,0)| \le C\), that is \eqref{eqn-U300}
holds.  In terms of the original solution \(u_n(x,t)\) this implies the desired bound
\eqref{eqn-unxx2} in the case \((t_1-s_n)t_1^{-2\alpha} \),
with \(\alpha = \frac k3\).

\smallskip

{\bf Case 2b :} Finally, if \((t_1-s_n)t_1^{-2\alpha} \leq 1 \), then since
\(t_1 \leq t_0\) is small and \(\alpha \geq 1\), we have \(t_1 \leq s_n +
t_1^{2\alpha} \leq s_n + t_1/2\), that is \(t_1 \in [s_n, 2s_n]\). In this case
we restrict to the parabolic cube of fixed size
\[
  \cQ_n' = \Big \{(\xi, s) \colon 0 \leq \xi \leq 2R, \, -
  \frac{t_1-s_n}{t_1^{2\alpha}} < s\leq - \frac{t_1-s_n}{t_1^{2\alpha}} + 1\Big \}.
\]
which contains the point \((1,0)\) and satisfies \(\cQ_n' \subset \cQ_n\).  Since
\({\displaystyle 0 < \frac{t_1-s_n}{t_1^{2\alpha}} \leq 1 }\), for any
\((\xi,s) \in \cQ'_n\) we have \(s \in [-1,1]\), thus
\(\vartheta_n := (1+t_1^{2\alpha-1} s)^{\alpha}\) satisfies the bounds
\(1/2 \leq \vartheta_n(s) \leq 3/2\), for all \(t_1 \leq t_0\) with \(t_0\)
sufficiently small.

\begin{claim} The bounds \(0 < C^{-1} \leq {\tilde u}_n(\xi, s) \leq C\) and
  \(|({\tilde u}_n)_{\xi}(\xi,s)| \le C\) hold on \(\cQ'_n\).  Furthermore,
  \(\| {\tilde u}_n \big (\cdot, - \frac{t_1-s_n}{t_1^{2\alpha}} \big ) \big \|_{C^3([0, 2R])}
  \leq C\).  In all these bounds \(C\) is a uniform constant independent of \(n\) and
  \(t_1\).
\end{claim}
\begin{proof}
Since \(1/2 \leq \vartheta_n(s) \leq 3/2\), similarly to Case 2a we can apply
\eqref{eqn-un666} to obtain that \(0 < C^{-1} \leq {\tilde u}_n(\xi, s) \leq C\) holds in
\(\cQ_n'\).  Also, similarly to the previous cases, \(|({\tilde u}_n)_{\xi}(\xi,s)| \le C\)
in \(\cQ_n'\) follows from \eqref{eqn-unx100}.  For the third bound it is
sufficient to just estimate second and third order derivatives.  To this end we use
\eqref{eq-un-higher-deriv-bound} which implies that
\(|\partial_x^j u_n (x, s_n) | \leq C s_n^{- (j-1) \frac k3}\) for \(j=2,3\) and
for all \(x \in [0, 2R t_1^{\frac k3}]\) (recall that \(t_1\sim s_n\)).

In terms of \({\tilde u}_n\) we get
\(\big | \partial ^j_\xi {\tilde u}_n \big (\xi, - \frac{t_1-s_n}{t_1^{2\alpha}} \big ) \big
| \leq C\) for \(j=2,3\) and for all \(\xi \in [0, 2R]\).  The above bounds imply
that
\(\| {\tilde u}_n \big (\cdot, - \frac{t_1-s_n}{t_1^{2\alpha}} \big ) \big \|_{C^3([0, 2R])}
\leq C\).  In all the these bounds the constant \(C\) is uniform, independent of
\(n\) and \(t_1\).
\end{proof}

The previous claim and standard interior (in space) regularity theory applied to
\eqref{eq-U-parn} on the cube \(\cQ_n'\) implies that
\(\sup_{0 \leq \xi \leq R} \big | ({\tilde u}_n)_{\xi\xi}(\xi,0) |\) (even
\(\|{\tilde u}_n (\cdot, 0) \|_{C^{2,1}([0,R])}\)) can be bounded in terms of
\(\| {\tilde u}_n \|_{C^0(\cQ_n')}\) and
\(\| {\tilde u}_n \big (\cdot, - \frac{t_1-s_n}{t_1^{2\alpha}} \big ) \big \|_{C^3([0,
2R])}\), and thus both are bounded by a constant \(C\) which is uniform in \(t_1\)
and \(n\).  We conclude that \eqref{eqn-U300} holds, which expressed in terms of
\(u_n(x,t)\) gives that \eqref{eqn-unxx2} holds in the last case where
\((t_1-s_n)t_1^{-2\alpha} >1 \), with
\(\alpha= \frac k3\).

\smallskip

Combining Cases 1a-1b and Cases 2a-2b, concludes the proof that the desired bound
\eqref{eqn-unxx} holds for all \((x,t)\) satisfying \(0 \leq x \leq M \sqrt{t}\), \,
\(t \in [s_n, t_0]\) and all \(n \geq n_0\), provided \(n_0\) is sufficiently large
and \(t_0 >0\) is sufficiently small.
\end{proof}

\subsection{Bounding \texorpdfstring{\(H\)}{H} in the intermediate and inner regions}
We will now show that \(H(x,t)\) is bounded in region \(x \leq M \sqrt{t}, \, 0 <  t \leq t_0\). 
Instead of showing that \(H\) is bounded, we will prove  that 
\[
  h(x, t) \stackrel{\rm def}= u_t = H\sqrt{1+u_x^2}
\]
is bounded.  Since \(u_x\) is uniformly bounded (Lemma \ref{lemma-der-bound}),
the bounds for \(h\) and \( H\) are equivalent. Arguments in this section have
been inspired by arguments from \cite{S}.

\smallskip 
The PDE for \(u\) implies that \(h=u_t\) satisfies
\[
  h_{t} = \frac{\partial}{\partial x}\left(\frac{h_{x}}{1+u_x^2}\right)
  + \frac 3x h_{x}+\frac 3{u^2}h. 
\]

\smallskip For \(n \geq n_0\), define \(h_n(x, t) := \partial_t u_n(x, t)\), where
\(u_n: [0,\infty) \times [s_n, t_0] \to \R\) is our approximating sequence of
solutions from the proof of Theorem \ref{thm-existence} in section
\ref{sec-existence}.  We choose a fixed \(m\in(2, 3)\) and set
\[
  \Lambda_n = \max \, \big \{ \bigl(1+t^{-\frac k3}x\bigr)^m  |h_n(x, t)| : 
  0\leq x \leq M \, \sqrt{t}, \, t \in [s_n, t_0]  \big \}.
\]
We claim the following holds. 
\begin{lemma}\label{prop-Lambdaj-bounded}
We have   \(\sup_{n}  \Lambda_n <\infty\).
\end{lemma}
This lemma implies that \(|h_n(x, t)|\) is uniformly bounded, and hence that
\(H_n=h_n/\sqrt{1+u_x^2}\) is also uniformly bounded.  Since the bound is uniform in
\(n\), by passing to the limit as \( n \to +\infty\) we will then obtain that the
mean curvature \(H(x,t)\) of our solution is bounded for
\(0 \leq x \leq M \sqrt{t}, \,\, 0 \leq t \leq t_0\).

%\smallskip 
\subsection{Choice of the blow-up sequences}
For the proof of Lemma \ref{prop-Lambdaj-bounded} we argue by contradiction and
assume that \(\sup_n \Lambda_n =\infty.\) Then we can pass to a subsequence so that
we may assume without loss of generality that
\begin{equation}\label{eqn-Lan}
  \lim_{n \to \infty} \Lambda_n = +\infty.
\end{equation}
Our goal in this section is \emph{to contradict \eqref{eqn-Lan}. }

The bound \eqref{eqn-unxx} for \(u_n\) implies the same bound for \(h_n\), namely, we
have
\begin{equation}
  \label{eq-hi-1+tx-bound}
  |h_n(x, t)|\lesssim t^{-k/3}\left(1+t^{-k/3}x\right)^{-4}
  \qquad (x\leq M \sqrt{t}, \,\, t\in [s_n, t_0]).
\end{equation}
The quantity \((1+t^{-k/3}x)^m|h_n(x, t)|\) attains its maximum in the region
\(\{(x, t)\mid 0\leq x\leq M\sqrt t, s_n\leq t\leq t_0\}\), so we can choose
\(T_n\in[s_n, t_0]\) and \(a_n\in [0, M\sqrt{T_n}]\) such that
\begin{equation}
  \label{eq-choice-an}
  |h(a_n, T_n)| = \Lambda_n \left(1+T_n^{-k/3}a_n\right)^{-m}.
\end{equation}
The inequality \eqref{eq-hi-1+tx-bound} implies

\[
  T_n^{k/3}\left(1+T_n^{-k/3}a_n\right)^{4-m}\lesssim \Lambda_n ^{-1}
\]
and thus
\[
  \max\left\{ T_n^{k/3}, T_n^{(m-3)k/3}a_n^{4-m} \right\} \lesssim
  \Lambda_n^{-1}.
\]
Since \(\Lambda_n\to\infty\) we find that \(T_n\to0\), and also
\[
  a_n \ll T_n^{\frac{3-m}{4-m}\frac k3}.
\]
At this point we use our assumption that \(k>3\) and choose \(m\) so close to
\(m=2\) that the exponent of \(T_n\) satisfies \( {\frac{3-m}{4-m}\frac k3} >
\frac 12, \) which then implies
\begin{equation}
  \label{eq-an-ll-sqrt-Tn}
    a_n \ll T_n^{{\frac 12}}.
\end{equation}
To complete the proof we distinguish between two cases \(a_n\lesssim
T_n^{\frac k3}\) and \(T_n^{\frac k3} \ll a_n\ll T_n^{\frac12}\), depending on
where the maximum \(a_n\) is attained.

\subsection{Case 1: \(a_n\lesssim T_n^{\frac k3}\)}

We choose the scale \(\alpha_n = T_n^{\frac k3}\) and form the following blow-up
sequences:
\begin{align}
  \label{eq-un-bar-def}
  \bar u_n(\xi, s) &= \alpha_n^{-1} u_n\left(\xi\alpha_n, T_n+s\alpha_n^2\right) \\
  \label{eq-hn-bar-def}
  \bar h_n(\xi, s) &= \Lambda_n^{-1} h_n\left(\xi\alpha_n, T_n+s\alpha_n^2\right).  
\end{align}
These functions are defined for
\[
  \xi>0,\qquad -S_n \leq s\leq 0 \quad \text{ where }S_n = \frac{T_n - s_n}{\alpha_n^2}
\]
and they satisfy the equations 
\begin{align}
  \frac{\partial \bar u_n}{\partial s}
  &=\frac{\bar u_{n\xi\xi}}{1+\bar u_{n\xi}^2} 
  + \frac{3}{\xi}\bar u_{n\xi} - \frac{3}{\bar u_n}
  \label{eq-bar-ui}
  \\
  \frac{\partial \bar h_n}{\partial s}
  &=\frac{\partial}{\partial \xi}
  \left(\frac{\bar h_{n\xi}}{1+\bar u_{n\xi}^2}\right)
  + \frac 3\xi \bar h_{n\xi} + \frac{3}{\bar u_n^2}\bar h_n.
  \label{eq-bar-hi}
\end{align}
Use~\eqref{eq-un-bar-def} with \(\alpha_n = T_n^{k/3}\) and the definition of the
inner region rescaling \(w_n(z,\tau)\) of \(u_n(x,t)\), i.e.,
\[
  u_n(x, t) = t^{\frac k3} w_n\left(t^{-\frac k3}x, \log t\right),
\]
with \(t=T_n+T_n^{\frac {2k}3}s\) to express \(\bar u_n(\xi, s)\) in terms of
\(w_n(z,\tau)\).  We get
\[
  \bar u_n(\xi, s) = \vartheta_n(s) w_n\left(\frac{\xi}{\vartheta_n(s)}, \log t \right)
\]
where
\[
  \vartheta_n(s) := t^{\frac k3} \, T_n^{-\frac k3}
  =  \bigl(T_n +T_n^{\frac {2k}3}s\bigr)^{\frac k3} T_n^{-\frac k3}=
  \bigl(1+T_n^{\frac {2k}3-1}s\bigr)^{\frac k3}.
\]
Since \(T_n\to0\) we have \(\vartheta_n(s)\to1\) uniformly for bounded \(s\),
and thus
\[
  \log t= \log T_n + 
  \frac 3k \, \log\vartheta_n(s)   \to -\infty
\]
uniformly for bounded \(s\).  Similarly to the last statement of  Theorem
\ref{thm-existence} we claim the following.
\begin{claim}
  \label{claim-inner-convergence}
  \(\bar u_n(\xi, s)\to
  W_{K_2}(\xi) \) in \(C^\infty_{\rm loc}\).
\end{claim}

\begin{proof}
  For every fixed \(\xi > 0\) there exists a \(n_0\) so that for all \(n \ge n_0\) we
  have
  \[
    \vartheta_n\, w_n^-\Big (\frac{\xi}{\vartheta_n(s)}, \log t \Big) \le
    \bar{u}_n(\xi,s) \le \vartheta_n\, w_n^+\Big(\frac{\xi}{\vartheta_n(s)}, \log t
    \Big)
  \]
  where \(\log t = \log T_n + \frac 3k \, \log\vartheta_n(s) \) and \(w_n^-\) and
  \(w_n^+\) are the lower and the upper barriers in the inner region, respectively.
  See Lemmas~\ref{lemma-inner-barrier1} and~\ref{lemma-inner-barrier2}.  This implies
  \[
    \vartheta_n W_{K_2^-(n)}\Big (\frac{\xi}{\vartheta_n}\Big) + D(T_n
    \vartheta_n^{\frac 3k})^{2\gamma} \le \bar{u}_n(\xi,s) \le \vartheta_n \,
    W_{K_2^+(n)}\Big(\frac{\xi}{\vartheta_n}\Big),
  \]
  where we recall that \((K_2^{\pm}(n))^3 = K_2^3 \pm \delta_n\).  Since
  \(\lim_{n \to\infty} T_n = 0\), \(\lim_{n \to\infty} \vartheta_n = 1\) and
  \(\lim_{n\to\infty} K_2^{\pm}(n) = K_2\), we conclude that
  \(\bar u_n(\xi, s) \to W_{K_2}(\xi)\) uniformly for bounded \(\xi\geq 0\) and
  bounded \(s\).

  Furthermore, since
  \( (\bar u_n)_{\xi\xi}(\xi, s) = \vartheta_n(s)^{-1}(w_n)_{zz} (z, \tau)\) is
  uniformly bounded for bounded \(\xi\) and \(s\), it follows that \(\bar u_{n\xi}\)
  also converges locally uniformly.  After bootstrapping the non-degenerate parabolic
  equation \eqref{eq-bar-ui} for \(\bar u_n\) we find that
  \(\bar u_n(\xi, s)\to W_{K_2}(\xi) \) in \(C^\infty_{\rm loc}\).

\end{proof}

Recall next that by the definition of \(\Lambda_n\) we have 
\begin{align*}
  |\bar h_n(\xi, s)|
  \leq \left(1+T_n^{\frac k3} \xi \, \bigl(T_n+T_n^{\frac {2k}3}s\bigr)^{-\frac k3}\right)^{-m} 
  =\left(1+\xi \, \bigl(1+T_n^{\frac {2k}3-1}s\bigr)^{-\frac k3}\right)^{-m}. 
\end{align*}
For \(s\leq 0\) and \(\xi>0\) this implies
\[
  |\bar h_n(\xi, s)| \leq \frac{1}{(1+\xi)^m}.
\]

\begin{lemma}
  Let \(\Phi(\xi) = W(\xi)- \xi \, W'(\xi)\), where \(W(\xi)\) is a solution to
  \eqref{eqn-Alencar}.  Then for any \(S_*>0\) there is a sequence \(\lambda_n\to 0\)
  such that \(e^{\lambda_n s}\Phi(\xi)\) is a super solution for \eqref{eq-bar-hi} in
  the region \(-\min\{S_n, S_*\}\leq s\leq 0\), where
  \(S_n= \frac{T_n-s_n}{\alpha_n^2}\).
\end{lemma}
\begin{proof}
Expanding the derivative in \eqref{eq-bar-hi} leads to
\[
  \bar h_{ns} = \frac{\bar h_{n\xi\xi}}{1+\bar u_{n\xi}^2} +\left\{
    \frac 3\xi - \frac{2\bar u_{n\xi}\bar u_{n\xi\xi}}{\bigl(1+\bar u_{n\xi}^2\bigr)^2}
    \right\}\, \bar h_{n\xi} + \frac{3}{\bar u_n^2} \bar h_n
  \stackrel{\rm def}{=} \cM_n(\bar h_n).
\]
If \(-\min\{S_n, S_*\}\leq s\leq 0\), then the \(C^\infty_{\rm loc}\) convergence of \(\bar u_n(\xi,s)\) to \(W(\xi)\) noted previously implies that 
the coefficients of the operator
\(\cM_n\) in this equation converge uniformly as \(n\to\infty\), so we can
write the RHS as
\[
  \cM_n[\bar h_n] =  \cM_\infty[\bar h_n] + \cR_n[\bar h_n]
\]
where
\[
  \cM_\infty[\eta] 
  \stackrel{\rm def}{=}  \frac{\partial}{\partial \xi}\left\{\frac{\eta_{\xi}}{1+W'(\xi)^2}\right\}
  + \frac{3}{\xi} \eta_{\xi} + \frac{3}{W(\xi)^2} \eta
\]
and where the remainder satisfies
\[
  |\cR_n[\eta]|\leq \lambda_n  \left(|\eta_{\xi\xi}| + |\eta_{\xi}|+|\eta|\right).
\]
with \(\lambda_n\to 0\).  Since \(\cM_\infty [\Phi] =0\), and
since \(|\Phi''(x)|+|\Phi'(x)|\lesssim \Phi(x)\) we find that
\[
  \cM_n[\Phi] \leq C\lambda_n \Phi. 
\]
Therefore \(e^{C\lambda_n s}\Phi(x)\) is an upper barrier for \(\bar h_{ns} =
\cM_n [\bar h_n]\).
\end{proof}
\begin{lemma}
\(S_n\to\infty\).
\end{lemma}
\begin{proof}
We argue by contradiction.  Assume that there is a subsequence of \(S_n\), along which
the limit is finite.  Without loss of generality we can take this to be \(S_n\)
itself, that is assume that
\[
  S_n = \frac{T_n - s_n}{\alpha_n^2} \leq  \bar S < +\infty, \qquad \forall j.
\]
This implies that \(T_n \leq s_n + \bar{S} \, \alpha_n^2 = s_n + \bar S \, T_n^{\frac {2k}3}. \)
Since \(T_n \to 0\) and \(k > 3\), we then conclude that \(T_n \leq 2 s_n\), for
\(n\gg 1\).

\smallskip 

We will now apply the maximum principle to \(\bar h_n \) in the region
\[
  -S_n\leq s\leq 0, \qquad 0\leq \xi\leq \epsilon \, T_n^{-(\frac k3-{\frac 12})}.
\]

Observe first that the construction of our initial data \(u_n(x, s_n)\) is such
that the surface coincides with an Alencar surface in the region \(y=o(1)\),
i.e.~for \(x\ll t^{{\frac 12}}\).  This implies that \(h_n(x, s_n) = 0\) for
\(x\ll \sqrt{s_n} \).  Using that \(T_n \leq 2 s_n\), for \(n \gg 1\), we conclude that
by taking \(n \gg 1\) and \(\epsilon\) sufficiently small we can guarantee that
\(\bar h_n(\xi, -S_n) = \Lambda_n^{-1} h_n(\alpha_n \xi, s_n)=0\) for
\( \xi \leq \epsilon \, \alpha_n^{-1} T_n^{{\frac 12}} = \epsilon \, T_n^{-(\frac k3-{\frac 12})}\).

At the end of this region where \(\xi = \epsilon \, T_n^{\frac 12-\frac k3}\) we have
\[
  |\bar h_n(\xi, s)| \leq 2 (1+\xi)^{-m} =2(1+\xi)^{-2} (1+\xi)^{-(m-2)} \lesssim
  T_n^{(m-2)(\frac k3-\frac 12)} \Phi(\xi).
\]
Choosing \(C\) as at the end of the  proof of the previous Lemma, we see that  by the same Lemma, for suitably large
\(\tilde{C}\) the function
\[
  \tilde C \, T_n^{(m-2)(\frac k3-\frac 12)} e^{C\lambda_n s} \Phi(\xi)
\]
is an upper bound for \(\bar h_n(\xi, s)\) while \(-S_n\leq s\leq 0\), and for all
\(n\).

Finally, at \(s=0\) this implies
\[
  |\bar h_n(\xi, 0)| \lesssim T_n^{(m-2)(\frac k3-{\frac 12})} \to 0 \quad \text{ as } \,  n \to\infty.
\]
This cannot be because \(\max _\xi |\bar h_n(\xi, 0)| = 1\), thus 
showing  that \(S_n\to\infty\).
\end{proof}

We can now complete the blow up argument, at least in the case where
\(a_n\lesssim T_n^{\frac k3}\).
Since \(S_n\to\infty\), we can pass to another subsequence along which \(\bar h_n\)
converges in \(C^\infty_{\rm loc}\) to an ancient solution \(\bar h\) of
\begin{equation}
  \label{eq-linearization-at-Alencar}
  \bar h_s 
  = \frac{\partial}{\partial \xi}\left(\frac{\bar h_\xi}{1+W'(\xi)^2}\right)
  + \frac{3}{\xi}\bar h_\xi + \frac{3}{W(\xi)^2} \bar h. 
\end{equation}
The ancient solution \(\bar h\) satisfies the bound
\[
  |\bar h(\xi, s)|\leq (1+\xi)^{-m}, \qquad (\xi\geq 0, \,\, s\leq 0).
\]
By the definition of \(a_n\) (see \eqref{eq-choice-an}) the function \((1+\xi)^m|\bar h_n(\xi, s)|\) attains its
maximum at \(\xi_n=a_n T_n^{-\frac k3}\).  We assumed here   that \(a_n\lesssim T_n^{\frac k3}\),
so we may assume also that \(\xi_n\to\bar\xi\) for some finite \(\bar\xi\geq 0\).
Thus we have
\begin{equation}
  \label{eq-bar-h-value}
  \bar h(\bar \xi, 0) = (1+\bar\xi)^{-m}.
\end{equation}
To complete the proof we compare this ancient solution with the stationary solution
\(\Phi(\xi) = W(\xi) - \xi\,  W'(\xi)\).  By the asymptotic expansion of the Alencar
solution we have
\[
  \Phi(\xi)= \bigl(\Gamma_1+o(1)\bigr) \, \xi^{-2} \quad (\xi\to\infty)
\]
for some constant \(\Gamma_1>0\).

Choose a large number \(\ell > 0\) and consider the function
\[
  \Psi(\xi) = \Phi(\xi) - \frac 12\Phi(\ell).
\]
Since \(\Phi(\xi)\) is a decreasing function of \(\xi\) we have
\[
  \frac 12 \Phi(\xi) \leq \Psi(\xi) \leq \Phi(\xi) \quad 
  \text{ for all }\xi\in[0,\ell].
\]
Furthermore, it follows from \(\cM_\infty [\Phi]=0\) that
\[
  \cM_\infty[\Psi](\xi) = -\frac{3\Phi(\ell)}{2W(\xi)^2}.
\]
Since \(W(\xi) = \xi+o(1)\) and \(\Phi(\xi) \sim \xi^{-2}\) for large \(\xi\), there
is a \(c>0\) such that \(W(\xi)^{-2}\geq c\, \Phi(\xi) \geq c \, \Psi(\xi) \).  There is also
a constant \(c>0\) with \(\Phi(\ell)\geq c  \ell^{-2}\).  Therefore we get
\[
  \cM_\infty[\Psi] \leq -c\ell^{-2}\Psi(\xi) \,\,\, 
  \text{ for }\xi\in[0, \ell].
\]
It follows that
\[
  \hat h(\xi, s) = e^{-c\ell^{-2}(s+s_0)}\Psi(\xi)
\]
satisfies \(\hat h_s \geq \cM[\hat h]\).

We will next  compare \(\bar h\) with \(\hat h\) in the domain \(\{0<\xi<\ell, -s_0<s<0\}\) which will lead to  contradiction.  At \(\xi=\ell\) we have 
\[
  \frac{|\bar h(\ell, s)|}{\hat h(\ell, s)}
  \leq
  \frac{(1+\ell)^{-m}}{\Psi(\ell)}e^{c\ell^{-2}(s+s_0)}. 
\]
Using
\[
  \Psi(\ell) \geq \frac 12 \Phi(\ell) \geq \frac 1C \, (1+\ell)^{-2}
\]
we therefore find for \(-s_0\leq s\leq 0\)
\[
  \frac{|\bar h(\ell, s)|}{\hat h(\ell, s)}
  \leq
  C \, (1+\ell)^{-(m-2)} e^{c\ell^{-2}(s+s_0)}
  \leq 
  C\, (1+\ell)^{-(m-2)} e^{c\ell^{-2}s_0}. 
\]
Since \(\Psi (\xi) \geq c\, (1+\xi)^{-2}\) for a uniform \(c\), at time  \(- s_0\) we have
\[
  \frac{|\bar h(\xi, -s_0)|}{\hat h(\xi, -s_0)}
  \leq \frac{(1+\xi)^{-m}}{\Psi(\xi)} \leq C (1+\xi)^{-(m-2)}
  \leq C.
\]
To conclude our argument, for any given \(\ell>0\) we choose \(s_0>0\) so large that
\[
  C \, (1+\ell)^{-(m-2)}e^{c\ell^{-2}s_0}>1.
\]  
By the maximum principle applied to the linear equation \(h_s=\cM_\infty [ h ] \) on the
domain \(\{0<\xi<\ell, -s_0<s<0\}\), we have 
\[
  \frac{|\bar h(\xi, s)|}{\hat h(\xi, s)}
  \leq 
  C\, (1+\ell)^{-(m-2)} e^{c\ell^{-2}s_0}  \,\,\, 
  \text{ for  }\,  \qquad 0 \leq \xi \leq \ell, \,\,\,\, -s_0 \leq s\leq 0. 
\]
In particular, 
\[
  \frac{|\bar h(\xi, 0)|}{\hat h(\xi, 0)} \leq 
  C (1+\ell)^{-(m-2)} e^{c\ell^{-2}s_0} \,\,\, 
  \text{ for }\, \qquad 0\leq \xi\leq \ell,
\]
and hence, using the definition of \(\hat h\),
\[
  |\bar h(\xi, 0)|\leq C(1+\ell)^{-m}\Psi(\xi) \qquad 
  \text{ for }\, 0\leq \xi\leq \ell. 
\]
The constant \(C\) does not depend on \(\ell\) so by choosing \(\ell\) large enough
we reach a contradiction if \(\bar h(\xi, 0)\ne 0\) for some \(\xi\geq 0\), since \eqref{eq-bar-h-value} needs to hold at the same time as well.

This completes the proof of Lemma \ref{prop-Lambdaj-bounded} in  the case \(a_n\lesssim T_n^{-\frac k3}\).

\subsection{Case 2: \(a_n\gg T_n^{-\frac k3}\)}\label{sec-blowup-case-2}
If we are not in Case 1, i.e.~if it is not true that \(a_n\lesssim T_n^{-\frac k3}\),
then there is a subsequence along which \(a_nT_n^{\frac k3}\to\infty\).
In this case we choose our scale to be \(\alpha_n=a_n\), and we define the
following blow-ups
\begin{equation}
  \label{eq-case2-blowups}
  \bar u_n(\xi, s)= a_n^{-1} u_n\left(a_n\xi, T_n+a_n^2 s\right), \qquad 
  \bar h_n(\xi, s)= \frac{h_n\left(a_n\xi, T_n+a_n^2 s\right)}{h_n(a_n, T_n)}.
\end{equation}
These blow ups are defined for all \(\xi\geq 0\) and for 
\[
  -S_n\leq s\leq 0, \quad  \text{ with } \, \, S_n = \frac{T_n-s_n}{a_n^{2}}. 
\]
By our  intermediate region asymptotics for \(u_n^-\) and \(u_n^+\), since \(e^{(\gamma+{\frac 12})\tau} \ll a_n \ll T_n^{{\frac 12}}\), and \(u_n^-(x,s) \le u_n(x,s) \le u_n^+(x,s)\),   we have  
\[
  \bar u_n(\xi, s) \to \bar u_{\infty}(\xi) = \xi,
\]
uniformly for bounded \(\xi\geq 0\) and \(s\), and in \(C^\infty_{\rm loc}\) for
\(\xi>0\) and \(s\leq 0\).

\begin{lemma}
For \(\bar h_n(\xi, s)\) we have the pointwise bound
\begin{equation}
  \label{eq-barhj-bound-case2}
  |\bar h_n(\xi, s)| \leq
  \Big (1+\frac{T_n^{\frac k3}}{a_n}\Big ) \Big (1+\frac{a_n^2s}{T_n}\Big )^{\frac{km}3}\,\xi^{-m}
\end{equation}
for all \(\xi\) with \(0<a_n\xi\leq \eta_0\).  
In particular, for large enough \(n\) we also have
\begin{equation}
  \label{eq-barhj-bound-case2-uniform}
  |\bar h_n(\xi, s)| \leq 2\, \xi^{-m}
\end{equation}
for all \(\xi\) with \(0<a_n\xi\leq \eta_0\), and for bounded \(s\).
\end{lemma}
\begin{proof}
By definition of \(\Lambda_n, a_n\), and \(T_n\) we have for all \(x\leq M \sqrt{t}\) and
\(t\in[s_n, t_0]\)
\[
  |h_n(x, t)|\leq \Lambda_n \left(1+t^{-\frac k3}x\right)^{-m},
  \quad
  |h_n(a_n, T_n)| = \Lambda_n \left(1+T_n^{-\frac k3}a_n\right)^{-m}.
\]
Hence
\[
  \left|\frac{h_n(a_n\xi, T_n+a_n^2s)}{h_n(a_n, T_n)}\right| 
  \leq
  \left\{\frac{1+T_n^{-\frac k3}a_n}{1+\bigl(T_n+a_n^2s\bigr)^{-\frac k3}a_n\xi}\right\}^m. 
\]
Discarding the ``\(+1\)'' in the denominator and mulitplying numerator and
denominator with \(T_n^{\frac k3}a_n^{-1}\) we find
\[
  \Big |\frac{h_n(a_n\xi, T_n+a_n^2s)}{h_n(a_n, T_n)}\Big | 
  \leq \Big (\frac{T_n^{\frac k3}}{a_n}+1\Big )^{m}
  \Big (1+\frac{a_n^2s}{T_n}\Big )^{\frac {mk}3} \xi^{-m}.
\]
This proves \eqref{eq-barhj-bound-case2}.
Since \(T_n^{\frac k3}\ll a_n\ll T_n^{\frac 12}\) (recall that we have assumed  \(a_n\ll T_n^{{\frac 12}}\)) we have
\[
  \Big (\frac{T_n^{\frac k3}}{a_n}+1\Big )^{m}
  \Big (1+\frac{a_n^2s}{T_n}\Big )^{\frac {mk}3} \to 1
\]
uniformly for bounded \(s\) which implies \eqref{eq-barhj-bound-case2-uniform}.
\end{proof}

This lemma tells us we have a sequence of solutions \(\bar h_n\) of the linear
equation
\begin{equation}
  \label{eq-barhj-case2-pde}
  \begin{aligned}
    \frac{\partial \bar h_n}{\partial t} 
    &=\frac{\partial}{\partial \xi}\biggl\{\frac{\bar h_{n\xi }}{1+\bar u_{n\xi}^2}\biggr\}
    +\frac{3}{\xi} \frac{\partial\bar h_n}{\partial\xi} + \frac{3}{\bar u_n^2}
    \bar h_n\\ 
    &=\frac{\bar h_{n\xi\xi}}{1+\bar u_{n\xi}^2} + \biggl\{\frac 3\xi - \frac{2\bar
    u_{n\xi}u_{n\xi\xi}}{\bigl(1+\bar u_{n\xi}^2\bigr)^2}\biggr\} \frac{\partial
    \bar h_n}{\partial\xi} + \frac{3}{\bar u_n^2}\bar h_n
  \end{aligned}
\end{equation}
which satisfies the uniform bound  \eqref{eq-barhj-bound-case2-uniform} for all \(n \geq n_0 \gg 1\). 
As before we have:
\begin{lemma}
\(S_n\to\infty\).
\end{lemma}
\begin{proof}
Assume that \(S_n\) is bounded, and, after passing to a subsequence, that we have 
\(S_n\to S_\infty\).

The function \(\bar u_n\) converges in \(C^{\infty}_{\rm loc}\) to \(\bar
u_\infty(\xi, s)=\xi\), so interior estimates for the divergence form
equation~\eqref{eq-barhj-case2-pde} imply that \(\bar h_n\) is locally uniformly
H\"older continuous for \(\xi>0\) and \(-S_n\leq s\leq 0\).
Moreover,  by the construction of \(u_n(\cdot,s_n)\)  we have  that \(\bar h_n(\xi,-S_n)=0\) for all \(a_n\xi\ll T_n^{{\frac 12}}\).
We may therefore assume that there is a convergent subsequence \(\bar{h}_n(\xi,
s)\to \bar h(\xi, s)\) where
\[
  |\bar h(\xi, s)|\leq \xi^{-m}
\] 
for all \(\xi>0\) and \(s\in[-S_\infty,0]\), and where \(\bar h\) is a solution
of
\[
  \bar h_s = \frac 12 \bar h_{\xi\xi}+ \frac{3}{\xi} \bar h_\xi + \frac{3}{\xi^2}\bar h
  \stackrel{\rm def}{=} \cM_0 [\bar h] 
\]
with \(\bar h(1, 0) = \pm1\), and \(\bar h(\xi, -S_\infty)=0\) for all
\(\xi>0\).  The limiting function \(\bar h\) is smooth for \(\xi>0\),
\(-S_{\infty}\leq s\leq 0\).  We note that \(\hat h(\xi) =\xi^{-2}+\xi^{-3}
\) is a stationary solution of \(\hat h_s = \cM_0  [\hat h]\), so that for any
\(\eta>0\) the functions \(\pm\eta\hat h\) provide upper and lower barriers for
\(\bar h\), provided we can show that \(-\eta\hat h<\bar h<\eta\hat h\) as
\(\xi\to0\) or \(\xi\to\infty\).  This boundary condition is fulfilled because
\(|\bar h(\xi, s)|\leq \xi^{-m}\) with \(2<m<3\).  The maximum principle
therefore implies that \(|\bar h|\leq \eta\hat h\) for all \(\eta>0\).  Letting \(\eta\to 0\) this yields \(\bar h(\xi, s)=0\) for all \(\xi>0\) and all \(s\in[-S_\infty,
0]\).
This contradicts \(\bar h(1, 0) = \pm1\) and shows that the sequence  \(S_n\) is unbounded. 
\end{proof}

We will next show that \(\bar h(1, 0) = 0\) which contradicts the fact that \(\bar h(1,0)=\pm 1\), and therefore {\em completes
the proof of Lemma~\ref{prop-Lambdaj-bounded}.} 

\begin{lemma}
\(\bar h(1, 0) = 0\).
\end{lemma}
\begin{proof}
Choose a small \(\epsilon>0\) and consider the function
\[
  k(\xi,s) = \bar h(\xi, s) - \epsilon \xi^{-2} - \epsilon\xi^{-3}.
\]
This function is a solution of the linear equation \(k_s=\cM_0[k]\).  In view of
the bound \(\bar h(\xi,s)\leq \xi^{-m}\), which holds for all \(\xi>0\) and
\(s\leq 0\), we have
\[
  k(\xi, s) \leq \xi^{-m} - \epsilon\xi^{-2} - \epsilon\xi^{-3}.
\]
Since \(2<m<3\) this implies that \(k(\xi, s) <0\) if \(\xi\leq
\epsilon^{\frac{1}{3-m}}\) or \(\xi\geq \epsilon^{\frac{-1}{m-2}}\).

The differential operator \(\cM_0\) is a standard Sturm-Liouville operator with
smooth coefficients on the interval \(I_\epsilon =[\epsilon^{\frac{1}{3-m}},
\epsilon^{\frac{-1}{m-2}}]\).  Since \(\xi^{-2}\) is a strictly positive
solution of \(\cM_0[\phi]=0\), the principal eigenvalue \(\lambda_0\) of 
\[
  \cM_0[\Omega(\xi)] = -\lambda_0 \, \Omega(\xi) , \qquad
  \Omega\bigl(\epsilon^{\frac{1}{3-m}} \bigr) = 
  \Omega\bigl(\epsilon^{\frac{-1}{m-2}}\bigr) = 0,
\]
is positive, and the corresponding eigenfunction \(\Omega(\xi)\) is also
positive for all \(\xi\) in the interior of the interval \(I_\epsilon\).
Choose \(C_\epsilon>0\) so that 
\[
  \xi^{-m} - \epsilon\xi^{-2} - \epsilon\xi^{-3} \leq C_\epsilon \, \Omega(\xi)
\]
for all \(\xi\in I_\epsilon\).

For any given \(s_0>0\) we then have 
\[
  k(\xi, - s_0) \leq C_\epsilon \, \Omega(\xi) \quad \text{ for all }\xi\in I_\epsilon.
\]
Moreover, \(\hat k(\xi, s) = C_\epsilon \, e^{-\lambda_0(s+s_0)}\Omega(\xi)\) is
a solution of \(\hat k_s = \cM[\hat k]\), so the maximum principle applied on
the domain \(I_\epsilon\times[-s_0,0]\) implies that at time \(s=0\) we have
\[
  k(\xi, 0) \leq \hat k(\xi, 0) = C_\epsilon e^{-\lambda_0s_0}\Omega(\xi).
\]
Since this is true for all \(s_0>0\) we conclude \(k(\xi, 0)\leq 0\).  By definition of \(k(\xi, s)\) this implies that \(\bar h(\xi, 0) \leq \epsilon\xi^{\frac{1}{3-m}}+\epsilon\xi^{\frac{-1}{m-2}}\) for all \(\xi\in I_\epsilon\).  In particular, this holds for \(\xi=1\) where it implies \(\bar h(1, 0) \leq 2\epsilon\).  This argument goes through for all \(\epsilon>0\), so we find \(\bar h(1, 0)\leq 0\).

Applying the whole argument once more to \(\tilde k(\xi, s) = -\bar h(\xi, s) -\epsilon\xi^{\frac{1}{3-m}}-\epsilon\xi^{\frac{-1}{2-m}}\) instead, we find \(-\bar h(1,0)\leq 0 \).  Hence \(\bar h(1, 0)=0\), as claimed.
\end{proof}

The proof of  Lemma~\ref{prop-Lambdaj-bounded} in now complete.  We can now conclude the proof of Theorem \ref{thm-Hbounded}.

\begin{proof}[Proof of Theorem \ref{thm-Hbounded}]
  Lemma \ref{prop-Lambdaj-bounded} implies \(\sup_n \Lambda_n < \infty\).  Using the
  definition of \(\Lambda_n\) this implies that
  \(|H_n| = \frac{|h_n|}{\sqrt{1+u_{nx}^2}}\) is also uniformly bounded.  Letting
  \(n\to \infty\), using Corollary \ref{cor-exist-un}, which implies that the
  \(\lim_{n\to \infty} u_n(x,t)= u(x,t)\), uniformly smoothly for \(t\in (0,t_0]\),
  we get that \(|H(x,t)| \le C\), for all \(0 \le x \le M\sqrt{t}\) and
  \(t \in (0,t_0]\).  Finally, combining this with Lemma \ref{lem-Houter} concludes
  the proof of Theorem \ref{thm-Hbounded}.
\end{proof}

\section{Appendix}
\label{sec-appendix}
\subsection{The linear equation in the intermediate region}
\label{sec-appendix-linear}
The eigenvalue equation \(\cL\varphi= (k-\frac 32)\varphi\) is 
\[
  \frac 12 \varphi_{yy} + \left(\frac 3y + \frac y2\right) \varphi_y
  +\left(\frac 3{y^2}-\frac 12\right)\varphi = \left(k-\frac 32\right)\varphi
\]
i.e.
\[
  \varphi_{yy} + \Bigl(\frac 6y + y\Bigr) \varphi_y
  +\frac 6{y^2}\varphi = 2(k-1)\varphi.
\]
Substitution: let \(\varphi(y) = y^{-2} \chi_k(y)\).  Then \( \chi_k\) satisfies the equation
\[
  \chi_k'' + \left(\frac 2y + y\right)  \chi_k' = 2k\, \chi_k.
\]
For every \emph{real} \(k>0\) there is a unique solution with \( \chi_k(0)=1\),
\( \chi_k'(0)=0\).  This solution is monotone increasing and for large \(y\) has
the expansion
\[
  \chi_k(y)=C_k y^{2k}+o(y^{2k})\qquad (y\to\infty).
\]
It is given by the series expansion
\begin{equation}
  \label{eq-phik-real}
  \chi_k(y) = \sum_{n=0}^\infty
  \frac{k(k-1)\cdots (k-n+1)}{n! (2n+1)!!}\,{y^{2n}},
\end{equation}
where \((2n+1)!! \stackrel{\rm def}{=} 1\cdot3\cdot5\cdot7\cdots(2n+1)\).  
This defines \(\varphi_k\) for all real \(k\).
We will only need these functions for integer values of \(k\), in which case \(\chi_k\) is a polynomial, and \(\varphi_k(y) = y^{-2}\chi_k(y)\) is given by
\begin{equation}\label{eq-phik}
  \varphi_k(y) = y^{-2}\sum_{n=0}^k \binom{k}{n}\frac{y^{2n}}{(2n+1)!!}.
\end{equation}
There is a second solution \(\hat\chi_k\) that satisfies 
\[
  \hat\chi_k(y) = e^{-y^2/2+o(y^2)} \qquad(y\to\infty).
\]
At \(y=0\) this solution is singular,
\[
  \hat\chi_k(y) = \frac{C}y + \cO(y) \qquad (y\to 0).
\]

\subsection{Proof of Lemma \ref{lem-g}}
\label{sec-g-exists}
The homogeneous equation \( 6\gamma\varphi - \cL\varphi = 0 \) has solutions of the
form
\[
  \varphi = C \, \varphi_k^1(y) + B  \, \psi_k^1(y), \qquad C,B \in \R
\]
where \( \varphi_k^1(y)\) and \(\psi_k^1\) are solutions with
\[
  \varphi_k^1(y) =
  \begin{cases}
    y^{-2} & (y\to0) \\
    \cO(y^{4k-5} ) & (y\to\infty)
  \end{cases}
\]
and
\[
  \psi_k^1(y) =
  \begin{cases}
    y^{-3} & (y\to0) \\
    \cO (e^{-y^2/2+o(y^2)}) & (y\to\infty).
  \end{cases}
\]
Since \(y=0\) is a regular singular point for the differential equation
\(6\gamma g-\cL g=G(y) = y^{-7} + y^{4k-7}\), one look for the solution in the form
of a power series.  From
\begin{equation}
  \label{eq-L-of-yr}
    (6\gamma-\cL)[y^r] = -\frac 12(r+2)(r+3) y^{r-2} + \frac 12 (4k-7-r) y^r
\end{equation}
it follows that \eqref{eq-g} has a particular solution of the form
\[
  g_{0p}(y) = C_0 y^{-5}P_0(y^2)+ C_1 y^{-3}\log(y)\, P_1(y^2),
\]
where \(P_j(y^2)\) are power series in \(y^2\) with \(P_j(0)=1\).  The logarithmic
term appears because \(r=-3\) is one of the characteristic exponents.  The
coefficient \(C_0\) is obtained by substitution in the equation.  One finds
\(C_0=-\frac 13\).

Every solution \(\varphi \) of the homogeneous equation satisfies,
\(\varphi = \cO(y^{-3}) = o(g_{0p})\) as \(y \to 0\), and therefore every solution
\(g\) of the inhomogeneous equation satisfies
\begin{equation}
  \label{eq-inhom-at-zero}
    g = g_{0p} + \cO(y^{-3}) = -\frac 13 y^{-5} + \cO(y^{-3}\log y), \qquad \mbox{as}\,\,\, y \to 0.
\end{equation}

The differential equation \(6\gamma g - \cL g= G\) has an irregular singular point at
\(y=\infty\), so we cannot use the power series method.  Instead, we obtain a
solution using sub and super solutions.  For any \(m\in\R\) the functions
\(g_\pm (y) = y^{4k-7}\pm m y^{4k-9}\) satisfy
\[
  (6\gamma-\cL) g_\pm
  = \Bigl(-\frac12(4k-5)(4k-4) \pm m\Bigr)y^{4k-9} + \cO(y^{4k-11})
  \qquad (y\to\infty).
\]
For \(m>\frac12(4k-5)(4k-4)\) it follows that \(g_-<g_+\) are sub and super solutions
for \(6\gamma g-\cL g=G\) on the interval \([y_0, \infty)\), if \(y_0\) is large
enough.  Hence there is a particular solution \(g_{\infty p}\) satisfying
\[
  g_{\infty p}(y) = y^{4k-7} + \cO(y^{4k-9}) \qquad (y\to\infty).
\]
At \(y=0\) all solutions satisfy \eqref{eq-inhom-at-zero} so \(g_{\infty p}\) also
satisfies \(g_{\infty p}(y) = -\frac13 y^{-5} + \cO(y^{-3}\log y)\).  The general solution
of the non-homogeneous equation \eqref{eq-g} is then of the form
\(g:= g_{\infty p} + C \, \varphi_k^1 + B \, \psi_k^1\), for \(C, B \in \R\).
However, the boundary condition \(g(y)=y^{4k-5}+o(y^{4k-5})\) as \(y\to\infty\)
restricts \(C=0\).  One concludes that \(g_B:= g_{\infty p} + B\, \psi_k^1\),
\(B \in \R\) is an one parameter set of solutions to \eqref{eq-g} which satisfies the
conditions of our lemma, thus finishing the proof.

\subsection{The Alencar solution}
\label{sec-Alencar}
\begin{lemma}
\label{lemma-alencar}
Let \(W:[0,\infty) \to\R\) be the solution of
\[
  \frac{W_{zz}}{1+W_z^2} + \frac 3z W_z - \frac{3}{W} = 0,
  \qquad W(0)=1, \quad W'(0)=0.
\]
Then \(W_{zz}>0\) and \( 0\leq W - zW_z \leq 1  \) for all \(z\geq 0\).

For large \(z\) the solution
\(W(z)\) has the expansion
\begin{equation}\label{eq-Alencar-expansion}
  W = z + \frac{\Gamma_2}{z^2} + \frac{\Gamma_3}{z^3} + \frac{\Gamma_5}{z^5} + \cdots
\end{equation}
for certain coefficients \(\Gamma_i\in\R\).
\end{lemma}
\begin{proof}
  The differential equation for \(W\) has been thoroughly studied.  In particular,
  \(W_{zz}>0\) and \(W>zW_z\) were shown by Velázquez in \cite[Prop.~2.2]{Velaz},
  (\(B''(u)>0\) and \(G_a(r)<0\) in his notation).  Here we prove that \(W(z)\) has
  the stated asymptotic expansion.  Let
  \[
    P=W_z\text{ and } Q = \frac{z}{W}.
  \]
  Then \((P,Q)\) as a function of \(\log z\) satisfy an automonomous system of
  differential equations,
  \begin{equation}\label{eq-PQ-Alencar-system}
    \left\{
      \begin{aligned}
        zP_z & = 3\bigl(1+P^2\bigr)(Q-P) \\
        zQ_z & = P-P^2Q
      \end{aligned}
    \right.
  \end{equation}
  This system has two fixed points, the origin \((0,0)\) and the point \((1,1)\).

  The origin corresponds to the boundary condition \(W_z=0, z=0\), while the fixed
  point corresponds to the Simons cone on which \(W=z\) and \(W_z=1\).

  The matrix of the linearization at \((0,0)\) is
  \(
    \left(\begin{smallmatrix}
      1 & 0 \\3 & - 3
    \end{smallmatrix}
  \right) \).  Its eigenvalues are \(\lambda_1=+1\) an \(\lambda_2=-3\).  The
  eigenvector corresponding to the unstable eigenvalue is \(\binom{4}{3}\).  The
  unique orbit in the unstable manifold of the origin is the Alencar solution.  It
  approaches the fixed point \((1,1)\) as \(z\to\infty\).  The matrix of the
  linearization at \((1,1)\) is
\(\left(    \begin{smallmatrix}
      -1 & -1 \\ 6 & -6
    \end{smallmatrix}
  \right)\) with eigenvalues/vectors \(\lambda_1 = -3, \vec v_{1}= \binom{1}{2}\) and
  \(\lambda_2=-4\), \(\vec v_{2} = \binom 13\).  The eigenvalues are both negative
  and they satisfy the ``no resonance'' condition, i.e.~neither eigenvalue is an
  integer multiple of the other.  This implies that there is a real analytic
  conjugacy of the nonlinear system~\eqref{eq-PQ-Alencar-system} near the fixed point
  \((1,1)\) with the linearization (see the chapter on normal forms and Poincaré's
  theorem in \cite{Arnold88}).  The general solution of the linear system is
  \[
    C_1 z^{-3}\binom{1}{2} + C_2 z^{-4}\binom{1}{3} =
    \begin{pmatrix}
      C_1z^{-3}+ C_2z^{-4} \\
      2C_1z^{-3}+ 3C_2z^{-4}
    \end{pmatrix}
    .
  \]
  This in turn implies that all solutions of~\eqref{eq-PQ-Alencar-system} that
  converge to \((1,1)\) are convergent power series in \(z^{-3}\) and \(z^{-4}\).  In
  particular, \(1/Q = W/z\) has an expansion of the form
  \[
    \frac Wz = 1 + C_3 z^{-3} + C_4z^{-4} + C_6z^{-6}+ C_7z^{-7}+\cdots
    =1+\sum_{l,m\geq1} C_{l,m} z^{-3l-4m}.
  \]
  Therefore \(W(z)\) satisfies
  \[
    W = z + C_3 z^{-2} + C_4z^{-3} + C_6z^{-5}+ C_7z^{-6}+\cdots =z+\sum_{l,m\geq1}
    C_{l,m} z^{1-3l-4m}.
  \]
  So if we set \(\Gamma_{m} = C_{m+1}\) we have proved the expansion
  \eqref{eq-Alencar-expansion}
\end{proof}

% When we rescale the Alencar solution we get
% \[
%   W_K(z) = KW(z/K)
%   = z +
%   K^3 \frac{\Gamma_2}{z^2} + K^4\frac{\Gamma_3}{z^3} + K^6\frac{\Gamma_5}{z^5} + \cdots
% \]

% \subsection{\(u_x\) bound}
% If \(u\) is a solution of
% \[
%   u_t=\frac{u_{xx}}{1+u_x^2}+\frac 3x u_x - \frac 3u
% \]
% then \(u_x\) satisfies
% \begin{align*}
%   u_{xt}
%   &=\frac{u_{xxx}}{1+u_x^2} - \frac{2u_{xx}^2}{\bigl(1+u_x^2\bigr)^2} u_x
%     + \frac 3x u_{xx} - \frac{3}{x^2} u_x + \frac 3{u^2} u_x\\
%   &=\frac{(u_x)_{xx}}{1+u_x^2} + \frac 3x (u_x)_x 
%     - Q(x,t) \, u_x.
% \end{align*}
% where
% \[
%   Q(x,t) = \frac{2u_{xx}^2}{ \bigl(1+u_x^2\bigr)^2} - \frac 3{u^2} + \frac 3{x^2}.
% \]
% If we assume that \(u(x,t) > x\) for all \(t,x\) then it follows that \(Q(x,t)>0\)
% and hence \(\max_{x>0}u_t(x,t)\) does not increase with time.  Also, if
% \(u_x(x,t)\geq 0\) at some \(t=t_n\) (such as \(t_n=e^{-n}\), corresponding to
% \(\tau=-n\)), then \(u_x\geq0\) for all \(t\geq t_n\).

\end{document}